\documentclass[12pt]{amsart}

\usepackage[utf8]{inputenc}
\usepackage{amsmath,amssymb,amsthm,graphicx,mathrsfs,url,amscd}
\usepackage[usenames,dvipsnames]{xcolor}
\usepackage[colorlinks=true,linkcolor=Blue,citecolor=PineGreen]{hyperref}
\usepackage{enumitem}
\usepackage[french, english]{babel}
\usepackage{csquotes}
\usepackage[all]{xy}
\usepackage{tabularx}
\usepackage[T1]{fontenc}
\usepackage{pdfpages}
\usepackage{pifont}

\newtheorem{theo}{Theorem}[section]
\newtheorem{prop}[theo]{Proposition}

\newtheorem{lemm}[theo]{Lemma}

\newtheorem{corr}[theo]{Corollary}
\newtheorem{coro}[theo]{Corollary}

\theoremstyle{definition}
\newtheorem{defi}[theo]{Definition}

\newtheorem{rema}[theo]{Remark}
\newtheorem{rem}[theo]{Remarque}

\numberwithin{equation}{section}

\newcommand{\mscr}{\mathscr}

\newcommand{\dd}{\mathrm{d}}

\newcommand{\R}{\mathbb{R}}
\newcommand{\Z}{\mathbb{Z}}

\newcommand{\N}{\mathbb{N}}

\newcommand{\wt}{\widetilde}

\newcommand{\mrm}{\mathrm}

\newcommand{\ml}[1]{\mathcal{#1}}
\newcommand{\todo}{\ding{168}}
\newcommand{\rrm}{\mathrm{r}}

\DeclareMathOperator{\supp}{supp}

\DeclareMathOperator{\Id}{Id}

\newcommand{\srm}{{\mathrm{s}}}
\newcommand{\urm}{{\mathrm{u}}}

\usepackage{afterpage}

\usepackage{bm}
\usepackage[margin=3cm]{geometry}

\newcommand{\mfS}{\mathfrak{S}}
\newcommand{\Ss}{_{\mathfrak S, \sigma}}

\newcommand{\oM}{\operatorname{Op}^{M_1}}
\newcommand{\oW}{\operatorname{Op}^{\mathrm {w}}}
\newcommand{\ohM}{\operatorname{Op}_h^{M_1}}

\begin{document}


\title[Nonlinear chaotic Vlasov equations]{Nonlinear chaotic Vlasov equations}

\author{Yann Chaubet}
\author{Daniel Han-Kwan}
\author{Gabriel Rivi\`ere}

\address{Laboratoire de Math\'ematiques Jean Leray, Nantes Universit\'e, UMR CNRS 6629, 2 rue de la Houssini\`ere, 44322 Nantes Cedex 03, France}


\email{yann.chaubet@univ-nantes.fr}

\email{daniel.han-kwan@univ-nantes.fr}

\email{gabriel.riviere@univ-nantes.fr}

\begin{abstract}

In this article, we study nonlinear Vlasov equations with a smooth interaction kernel on a compact manifold without boundary where the geodesic flow exhibits strong chaotic behavior, known as the Anosov property.
We show that, for small initial data with finite regularity and supported away from the null section, there exist global solutions to the nonlinear Vlasov equations which weakly converge to an equilibrium of the free transport equation, and whose potential strongly converges to zero, both with exponential speed. Central to our approach are microlocal anisotropic Sobolev spaces, originally developed for studying Pollicott-Ruelle resonances, that we further refine to deal with the geometry of the full cotangent bundle, which paves the way to the analysis of nonlinear Vlasov equations.


\end{abstract}

\maketitle

\tableofcontents

\section{Introduction}

Let $(\Sigma,\mathrm{g})$ be a smooth, compact, connected, oriented and Riemannian manifold of dimension $n\geqslant 2$ which has no boundary. All along the article, we make the assumption that $(\Sigma,\mathrm{g})$ is an \emph{Anosov} manifold \cite{anosov1967geodesic}, which means that the geodesic flow displays chaotic features, for instance strong sensitivity to initial conditions (we refer to Section~\ref{subsec:geom} below for a proper definition). In particular, this condition holds as soon as $(\Sigma,\mathrm{g})$ has negative sectional curvature;  this includes the special case of compact quotients of the hyperbolic plane (see e.g. \cite{FH}). We are interested in kinetic mean-field equations, namely  in the \emph{nonlinear} Vlasov equation
\begin{equation}\label{eq:vlasov}
\left\{
\begin{aligned}
 &\partial_t u+ \left\{\mathrm H + \Phi(u), \,u\right\}=0, \\
 &u|_{t=0}=u_0\in L^1(\mathrm{T}^*\Sigma). 
\end{aligned}
 \right.
\end{equation}
Here $\{\cdot, \,\cdot\}$ denotes the Poisson bracket in $\mathrm{T}^*\Sigma$, $\mathrm H$ stands for the classical Hamiltonian associated with the geodesic flow, that is
$
\mathrm H(x, \xi) = \displaystyle{\frac{1}{2}|\xi|^2_x},
$
and the interaction potential $\Phi(u) \in \mathscr C^\infty(\Sigma)$ is given by
\begin{equation}\label{eq:Poisson}
\Phi(u)(x) = \mathrm K \pi_*u(x)=\int_{\mrm{T}^*\Sigma} K(x,y)u(y,\eta)\dd \mrm L(y,\eta),
\end{equation}
where $\mrm L$ is the Liouville measure on $\mathrm{T}^*\Sigma$, $\pi : (x,\xi)\in\mathrm{T}^*\Sigma \mapsto x\in\Sigma$ is the natural projection and $\mathrm K$ is an integral operator acting on $L^2(\Sigma)$. Its kernel $K \in \mscr C^\infty(\Sigma \times \Sigma,\R)$ is assumed to satisfy 
$$
\int_{\Sigma} K(x,y) \dd \text{vol}_{\mathrm{g}}(y) = 0 \quad \text{for each }x \in \Sigma,
$$
with $\text{vol}_{\mathrm{g}}$ the Riemannian volume induced by the metric $\mathrm{g}$. From the physical point of view, when $u_0 \geqslant 0$, the function $u(t)$ can be interpreted as the distribution function in phase space of a population of particles, whose dynamics in $\Sigma$ is steered by a self-induced potential $\Phi(u)$, created by their non-trivial spatial density $\pi_* u(t,x)=\int_{\mrm{T}_x^*\Sigma}u(t,x,\xi)\,\dd \mrm L_x(\xi)$.

The main objective of the present paper is to prove the following result.

\begin{theo}\label{thm:main} There are $N_0\in \N_{\geqslant 1}$ and $\vartheta_0>0$ depending only on $(\Sigma, \mathrm g)$ (and not on the interaction kernel $K$) such that the following holds. 

For any $0<r_0<1$, there are $C, \varepsilon>0$ such that, if $u_0$ satisfies
$$
\|u_0\|_{\mscr C^{N_0}} \leqslant \varepsilon \quad \text{and} \quad \supp u_0 \subset \{(x,\xi) \in \mathrm{T}^*\Sigma~:~ |\xi|_x \in (r_0,r_0^{-1})\},
$$
then \eqref{eq:vlasov} admits a unique solution $u\in\mathscr{C}^1(\R,L^1(\mathrm{T}^*\Sigma))$. Moreover, there exist functions $h_{\pm\infty}\in  L^\infty_{\mathrm{comp}}(\R_{>0})$ such that, for all $\psi\in \mathscr{C}^\infty_c(\mrm T^*\Sigma)$,
$$
\left|\int_{\mrm T^*\Sigma} u(\pm t) \, \psi \, \dd {\mrm L}-\int_{\mrm T^*\Sigma} \left(h_{\pm \infty}\circ \mathrm H\right) \, \psi \, \dd {\mrm L}\right|\leqslant C e^{-t\vartheta _0 r_0}\|\psi\|_{\mathscr{C}^{N_0}}\|u_0\|_{\mscr C^{N_0}},\qquad t\geqslant 0 .
$$
In addition, for every $N\geqslant 0$, there is $C_N>0$ such that 
$$
\|\Phi(u(\pm t))\|_{\mathscr{C}^{N}}\leqslant C_N e^{-t\vartheta_0 r_0} \|u_0\|_{\mscr C^{N_0}},\qquad t\geqslant 0.
$$

\end{theo}

Theorem~\ref{thm:main} provides a fine description of the large time behavior of small solutions with finite regularity, initially compactly supported away from the null section $\{\xi=0\}$, to the nonlinear Vlasov equation~\eqref{eq:vlasov}. Namely, it shows that the distribution function $u(t)$ weakly converges to an equilibrium of the the free transport equation
$$
 \partial_t f+ \left\{\mathrm H, \,f\right\}=0,
 $$
 and the potential $\Phi(u(t))$ strongly converges to $0$, both with exponential speed. 
 A direct consequence of this theorem, by a rescaling argument, is the following statement which states convergence to equilibrium for {weakly} nonlinear perturbations of the free Vlasov equation.
\begin{coro}\label{coro:main} There exists $N_0\geqslant 1$ such that, for any $0<r_0<1$ and any $M_0>0$, one can find $\varepsilon_0> 0$ so that if $\|u_0\|_{\mscr C^{N_0}} \leqslant  M_0$ and 
$$
\supp u_0 \subset \{(x,\xi) \in \mathrm{T}^*\Sigma~:~ |\xi|_x \in (r_0,r_0^{-1})\},
$$
then, for every $0\leqslant\varepsilon\leqslant\varepsilon_0$ the solution of the problem 
\begin{equation}\label{eq:vlasov2}
\left\{
\begin{aligned}
&\partial_t u+ \left\{\mathrm H + \varepsilon \Phi(u), \,u\right\}=0,\\
&u|_{t=0}=u_0,
\end{aligned}
\right.
\end{equation}
satisfies the same properties as in the conclusion of Theorem~\ref{thm:main}.
\end{coro}  
Under this form, our main result can be understood as an analogue for the Vlasov equation of the recent results of Bahsoun, Liverani and S\'elley on globally coupled Anosov diffeomorphisms, acting discretely on distribution functions~\cite{BahsounLiveraniSelley2023} (see also \cite{SelleyTanzi2021,Galatato2022,Tanzi2023} and references therein). More precisely, they consider weakly nonlinear mean-field perturbations of transfer operators associated with Anosov diffeomorphisms. They prove that such nonlinear transfer operators admit a unique physical (or Sinai-Ruelle-Bowen) measure and that convergence to this equilibrium measure holds at an exponential rate. The strategy of~\cite{BahsounLiveraniSelley2023} is based on a fixed point theorem argument which is tailored for discrete time models.  

\subsection{Exponential mixing for the geodesic flow with the Anosov property}  In the free case (that is when $K\equiv 0$), Theorem~\ref{thm:main} is a consequence of  by now classical results. The weak convergence of $u$ without explicit speed of convergence is a consequence of  Anosov's seminal work~\cite{anosov1967geodesic} following earlier contributions by Hopf on the ergodicity of the Liouville measure~\cite{Hopf1939,Hopf1940}. 
Anosov's result was refined for manifolds of \emph{constant} negative curvature by providing speed of convergence using tools from harmonic analysis on Lie groups~\cite{ratner1987rate, moore1987exponential, pollicott1992exponential}. 
The extension of these works to variable negative curvature was obtained in dimension $2$ using methods from Markov partitions by Dolgopyat~\cite{dolgopyat1998decay} and in any dimension using Banach spaces with anisotropic regularity adapted to Anosov flows by Liverani~\cite{liverani2004contact} (see also~\cite{chernov1998markov} for earlier results showing subexponential decay when $r_0>0$). The outcome of the results in these references is that, when $K\equiv 0$, we have
\begin{equation}\label{eq:exp-mixing-classical}
\left|\int_{\mrm{S}^*\Sigma} u(t)\psi \,\dd {\mrm L}_1- \int_{\mrm{S}^*\Sigma}\left(\int_{\mrm{S}^*\Sigma} u_0\dd {\mrm L}_1\right) \psi\, \dd {\mrm L}_1\right|\leqslant Ce^{-\vartheta_0 |t|}\|u_0\|_{\mathscr{C}^{N_0}}\|\psi\|_{\mathscr{C}^{N_0}},
\end{equation}
where $\mrm{S}^*\Sigma=\{(x,\xi)\in \mrm{T}^*\Sigma: |\xi|_x=1\}$ and ${\mrm L}_1$ is the desintegration of the Liouville measure on $\mrm{S}^*\Sigma$. Integrating this expression over $\mrm{T}^*\Sigma$ yields the following theorem for the \emph{free} Vlasov equation:
\begin{theo}[Dolgopyat-Liverani]\label{t:dolgopyatliverani} Let 
$K\equiv0$. 
There exist an integer $N_0\geqslant 1$ and two constants $C,\vartheta_1>0$ such that the following holds. For any $0\leqslant r_0<1$, for any $u_0\in \mathscr{C}^{N_0}(\mathrm{T}^*\Sigma)$ such that 
$$
\supp u_0 \subset \left\{(x,\xi) \in \mathrm{T}^*\Sigma~:~ |\xi|_x \in [r_0,\infty[\right\},
$$
and for any $\psi\in \mathscr{C}^\infty_c(\mrm T^*\Sigma)$, one has
$$\forall t\geqslant 0,\quad
\left|\int_{\mrm T^*\Sigma} u({\pm} t) \, \psi \,  \, \dd {\mrm L}-\int_{\mrm T^*\Sigma} \left(h_{\operatorname{lin}}\circ \mathrm H\right)  \psi \,  \dd {\mrm L}\right|\leqslant C \frac{e^{-t\vartheta_1 r_0}}{(1+t)^{n}}\|\psi\|_{\mathscr{C}^{N_0}}\|u_0\|_{\mscr C^{N_0}},
$$
where
$$
h_{\operatorname{lin}}(r)=\int_{\mrm{S}^*\Sigma} u_0(x,r\xi_1) \, \dd {\mrm L}_1(x,\xi_1).
$$
\end{theo}
Note that, in this free setting, the case $r_0= 0$ is allowed but only algebraic decay is reached in that case.
For the reader's convenience, see Appendix~\ref{a:liverani} for a short proof of this theorem based on the mixing property~\eqref{eq:exp-mixing-classical}. Let us emphasize that we will not be able to use this theorem directly and that we will rather rely on microlocal refinements of this result (see Theorem~\ref{t:nonnenmacherzworski} below). Indeed, \eqref{eq:exp-mixing-classical} was subsequently recovered and improved using methods from microlocal analysis~\cite{tsujii2010quasi, tsujii2012contact, nonnenmacher2015decay, faure2017semiclassical, faure2021micro}, with a strategy initiated in~\cite{baladi2007anisotropic, faure2008semi, faure2011upper} and which can display analogies with the study of quantum resonances~\cite{AguilarCombes1971, BaslevCombes1971, HelfferSjostrand1986}.
These alternative approaches allowed to give sharp expressions on the high frequency value of the decay rate $\vartheta_0$ and to obtain results on the distribution of the so-called Pollicott--Ruelle resonances. On top of that, the Fourier analysis machinery behind these works turns out to be relevant for applications to nonlinear Vlasov equations as we shall see in this article. In fact, in all the above references, the upper bound in~\eqref{eq:exp-mixing-classical} involves the norm of $u_0$ and $\psi$ in some Banach space of distributions with anisotropic H\"older or Sobolev regularity adapted to the free dynamics. For the application to  nonlinear Vlasov equations, the microlocal anisotropic Sobolev norms of Faure and Sj\"ostrand~\cite{faure2011upper} appearing in the upper bound obtained by Nonnenmacher and Zworski in~\cite[\S9]{nonnenmacher2015decay} will be appropriate and this microlocal approach will allow us to show that 
\begin{equation}\label{eq:integralremainder}
h_{\pm \infty}(r)=h_{\text{lin}}(r)-\int_{\mrm{S}^*\Sigma}\left(\int_0^{\pm\infty}\{\Phi(u(s)),u(s)\}(x,r\xi_1)ds\right)\dd {\mrm L}_1(x,\xi_1)
\end{equation}
is well defined and that it is indeed the limiting distribution of the solution to~\eqref{eq:vlasov} provided $\|u_0\|_{\mathscr{C}^N}$ is small enough. Yet, it may happen that other types of anistropic norms could be used modulo some careful adaptations of our nonlinear arguments. For instance, the aforementioned recent work~\cite{BahsounLiveraniSelley2023} (see also~\cite{BahsounLiverani2024}) makes use of geometric anisotropic Banach spaces of distributions as developped in~\cite{BahsounLiverani2022} to study nonlinear mean-field perturbations of Anosov diffeomorphisms, which can be somehow viewed as discrete analogues of the Vlasov equation.



  
\subsection{Other Vlasov equations of geometric origin}
  
Other geometric Vlasov equations were previously studied with motivations from physics and with an emphasis on the fine large time behavior of solutions. Such equations on manifolds are in fact natural in the context of general relativity. See for instance \cite{Ring}, \cite{ACGS} or the recent works about the stability of the Minkowski space-time for Einstein--Vlasov \cite{Tay,LinTay,FJS,BFJST} and references therein. Compared with the present work, the free Hamiltonian is given in these references by $\mrm H(x,\xi)=|\xi|_x^2/\sqrt{|\xi|^2_x+m}$ (with $m\geqslant 0$), the potential $\Phi$ is taken to be identically $0$ and the nonlinearity comes from the fact that the (time dependent) metric $\mathrm{g}$ is given by the Einstein equation with a source term depending linearly on certain averages of $u$. We also refer to \cite{ABJ,Big,Vel} which establish various decay estimates of spacetime observables for the (linear) massless Vlasov equation outside of black hole geometries, whose associated flows display hyperbolic features. In other directions, we may finally mention \cite{Moschidis} for the instability of the Anti-de Sitter space-time or  \cite{HL1,HL2,Tou1,Tou2} for works related to derivations of Einstein--Vlasov from the Einstein vacuum equation in the high frequency regime, the so-called Burnett conjecture.

More closely related to our geometric framework, Velozo Ruiz and Velozo Ruiz \cite{VRVR2} treated recently the case of the free Vlasov equation set on a \emph{noncompact} asymptotically hyperbolic manifold $\Sigma$ whose trapped set is empty and established exponential convergence to $0$ of $\pi_*(u)(t)$ when $r_0>0$ and polynomial convergence when $r_0\equiv 0$ and $(\Sigma,\mathrm{g})=(\mathbb{H}^2,y^{-2}(\dd x^2+\dd y^2))$ (see also \cite{Salort} which considers the case of a general metric on $\R^n$). In this non-compact and non-trapping setting, dispersion prevails, leading to decay for the spatial density $\pi_* (u)(t)$. Finally, together with Bigorgne, they studied the stability of vacuum for the Vlasov--Poisson equation on $\mathrm{T}^*\R^n $, in the presence of an external potential $-\frac{1}{2} |x|^2$ which induces hyperbolic features of the free Hamiltonian flow. Indeed, the linear flow has then a non-empty trapped set which is the \emph{hyperbolic fixed point} $(x,\xi)=(0,0)$. For these models and for the Poisson interaction $\mathrm K=\pm\Delta^{-1}$, exponential decay to $0$ for $\pi_* u(t)$ was proved in~\cite{VRVR1} while,  after proper renormalization of the solution along the unstable manifold, the first term in the asymptotic expansion was identified in~\cite{BVRVR} when $n=2$ together with a fine asymptotic description of the limit distribution.


\subsection{The flat torus case}

Besides these models of geometric flavour, it is also instructive to compare our work to what is known for Vlasov equations set on the cotangent bundle of the flat torus $\mathbb{T}^n$, which has been, as far as we know, the unique compact manifold where progress has been made concerning fine large time behavior of solutions to~\eqref{eq:vlasov}. In this simple geometric framework, the Vlasov equation reads
\begin{equation}\label{eq:vlasov-torus}
 \partial_t u+ \xi \cdot \nabla_x u - \nabla_x \Phi(u) \cdot \nabla_\xi u = 0,\quad u|_{t=0}=u_0\in \mscr{C}^N(\mathrm{T}^* \mathbb{T}^n,\R)\cap L^1(\mathrm{T}^*\mathbb{T}^n),
\end{equation}
with the identification $\mathrm{T}^*\mathbb{T}^n \simeq \mathbb{T}^n  \times \mathbb{R}^n$. 
In a breakthrough work \cite{MV}, Mouhot and Villani have proved that small, \emph{analytic} (or at least \emph{Gevrey} with high regularity index) solutions to~\eqref{eq:vlasov-torus}  weakly converges to a stationary state to free transport, and the potential $\Phi(u)$ strongly converges to $0$, both with exponential speed (sub-exponential in the Gevrey case), a result referred to in the mathematical literature as \emph{Landau Damping}.
Prior to that, we mention that \cite{Deg} studied the linear spectral problem, and in \cite{CM,HV} the authors were able to construct examples of solutions that display this behavior. It is also important to note that these references do not restrict to smooth interaction kernels $K$, and in particular treat  Vlasov--Poisson equations; 
note though that they impose a convolution kernel, that is $K(x,y)$ must be of the form $k(x-y)$.

 The linear mechanism at work in the torus is the so-called {\it phase mixing}, which can be briefly summarized as follows. On the torus, the free motion is \emph{integrable}, and the solution to the free transport equation has an explicit formula, namely we have 
 $ u(t,x,\xi) = u(0,x-t\xi, \xi)$ and
 \begin{align*}
 \int_{\mathrm T^*\mathbb{T}^n} u(t) \psi \dd {\mrm L} &=  \int_{\mathrm T^*\mathbb{T}^n} \left(\int_{\mathbb{T}^n} u_0 \, \dd x\right) \psi \dd x \dd \xi + \sum_{k \neq 0} \int_{\R^n} \mathcal{F}_{x} u_0(k,\xi) \mathcal{F}_{x}\psi(-k,\xi) e^{it k \cdot \xi} \, \dd \xi \\
  &=\int_{\mathrm T^*\mathbb{T}^n} \left(\int_{\mathbb{T}^n} u_0 \, \dd x\right) \psi \dd x \dd \xi + \sum_{k \neq 0} \int_{\R^n} \mathcal{F}_{x,\xi} {u_0}(k,v-tk) \mathcal{F}_{x,\xi}\psi(-k,-v) 
 \, \dd v,
 \end{align*}
 where $ \mathcal{F}_{x}$ (resp. $ \mathcal{F}_{x,\xi}$) refers to the Fourier coefficients in $x$ (resp. Fourier coefficients in $x$ and Fourier transform in $\xi$).
 Therefore, $u$  converges weakly to its space average as a consequence of Riemann--Lebesgue's theorem. Furthermore,  the speed of convergence can be quantified, depending on the regularity with respect to $\xi$ of $u_0$. In particular, we have
 $$
  \left|\int_{\mathrm T^*\mathbb{T}^n} u(t) \psi \, \dd {\mrm L} - \int_{\mathrm T^*\mathbb{T}^n} \left(\int_{\mathbb{T}^n} u_0 \dd x\right) \psi \, \dd x \dd \xi
  \right| \leqslant C e^{-tR} \| u_0\|_{\operatorname{analytic}} \| \psi\|_{\operatorname{analytic}} ,
  $$
  where $\| \cdot \|_{\operatorname{analytic}}$ stands for some appropriate analytic norm and $R$ is related to the radius of analyticity \cite{Deg}. This exponential convergence can be compared to the one appearing in~\eqref{eq:exp-mixing-classical}. The linear mechanism we exploit in~\eqref{eq:exp-mixing-classical}, i.e. the strong \emph{chaotic} features of the geodesic flow in compact Anosov manifolds, is thus of very different nature to that of the torus, which relies on the (strong) regularity of the solution to the Vlasov equation. In the nonlinear setting, it results into the fact that $\Phi(u(t))$ converges exponentially fast to $0$ in analytic topology. For what concerns solutions $u(t,x,\xi)$ to~\eqref{eq:vlasov-torus}, they weakly converge to some limit distribution $h_{\pm\infty}(\xi)$, still with an exponential speed, where $h_{\pm\infty}(\xi)$ is an analytic function on $\R^n$. The main difficulty in the nonlinear analysis on $\mathbb{T}^n$ is due to resonances referred to as {\it plasma echoes}  which can lead to catastrophic growth phenomena; they are the  reason why the result of \cite{MV} is restricted to solutions with infinite regularity. Heuristically, these echoes give a threshold for the Gevrey regularity of the initial data, which is related to the regularity of the interaction kernel $K$.
  
  Moreover, \cite{MV} went beyond the small data regime, as they identify a class of stable homogeneous equilibria $h(\xi)$, namely satisfying the so-called Penrose stability condition. In particular, in the repulsive case, this is automatically satisfied for smooth positive radial symmetric $h$ in dimension $n\geqslant 3$.
 Large time behavior is then established for analytic data in a small neighborhood of such equilibria.
 Shortly after, Lin and Zeng proved the existence of low regularity periodic in time objects, the so-called BGK waves, in the vicinity of Penrose stable equilibria~\cite{LZ11,LZ12}, thus implying that the main result of \cite{MV} does not hold in such a low regularity.
 In \cite{Bed21}, Bedrossian showed that in high but finite regularity, the scenario of \cite{MV} cannot hold, precisely due to the aforementioned plasma echoes (see also \cite{Zil21,Zil23}).
 Over the years, the seminal work \cite{MV} was sharpened and the general strategy got simplified, first by Bedrossian, Masmoudi and Mouhot in \cite{BMM} where the presumed optimal Gevrey regularity index was almost reached, and more recently by Grenier, Nguyen and Rodnianski in \cite{GNR21}. Following these works, Ionescu, Pausader, Wang and Widmayer did obtain the optimal Gevrey index~ \cite{IPWW24}. We finally refer to \cite{Bed22} for a more complete discussion and review of the relevant mathematical literature. 

 In another direction, for the so-called Vlasov-HMF (Hamiltonian Mean Field) equation, that is when the potential $\Phi(u)$ only depends on a \emph{finite} number of Fourier modes of the density $\pi_*u$, Faou and Rousset showed in \cite{FR} that Gevrey regularity can be dispensed with, as they obtained stability results in finite Sobolev regularity. This is possible because ``infinite chains'' of plasma echoes cannot occur in this setting. See also~\cite{BOY, FHR}.  
On Anosov manifolds, the linear mixing mechanism is strong enough so that, even if we do not restrict to an interaction potential involving a finite number of modes, resonances such as plasma echoes  do not seem to appear in the nonlinear analysis, which allows to obtain finite regularity results.

\subsection{Further questions}
Whereas on the torus spaces of analytic functions were the right spaces to capture a spectral gap for the dynamics, on Anosov manifolds, the right spaces are \emph{anisotropic Sobolev spaces} (equipped with the aforementioned anisotropic Sobolev norms as first constructed by Faure and Sj\"ostrand~\cite{faure2011upper}) which we shall introduce soon enough. This ultimately leads to the finite regularity result of Theorem~\ref{thm:main}. However it is fair to acknowledge two shortcomings of our analysis:
\begin{itemize}
\item the kernel $K$ needs to be smooth (at least of class $\mathscr C^k$ with $k$ large enough); in particular our work does not handle the Poisson case $\mathrm K=\pm\Delta_g^{-1}$;
\item the initial support of the distribution function needs to be away from the null section; this is because the aforementioned mixing property of the geodesic flow degenerates and we have not been able to cope with this degeneracy in the nonlinear analysis. In the free case, the decay results from~\cite{ratner1987rate, dolgopyat1998decay, liverani2004contact, tsujii2010quasi, tsujii2012contact, nonnenmacher2015decay, faure2021micro}, ensure algebraic decay to equilibrium when the initial data are non trivial near the null section (recall Theorem~\ref{t:dolgopyatliverani})  but it is not clear how to exploit this weaker mixing property for the nonlinear problem.
\end{itemize}
Hopefully, they will be overcome in the future, but this likely will require substantial work. For what concerns results in a small vicinity of a non-trivial equilibrium $f_0(H)$ (where $f_0$ is not necessarily identically $0$), our strategy of proof should in principle allow to write an abstract Penrose stability condition as in the Euclidean setting but we still need to understand in which practical situations it would be indeed relevant.

\subsection{Organization of the article} Section~\ref{sec:strategy} mainly serves pedagogical purposes: we describe an heuristic scheme of proof for a toy model, namely a nonlinear perturbation of a contact Anosov vector field.
The main example is that of the geodesic flow on $\mathrm S^\star \Sigma$. Some of the difficulties of the analysis already appear at this level; yet, the main technical issues actually occur when dealing with nonlinear Vlasov equations set on the full cotangent bundle $\mrm{T}^*\Sigma$. 

With this perspective in mind, in Section~\ref{sec:escape}, we proceed to geometric constructions which are motivated by the way we will design the anisotropic Sobolev spaces adapted to our nonlinear equations. More precisely, these geometric preliminaries consist in building nice enough functions which decay along the flow lines of the symplectic lift on $\mathrm{T}^*\mathrm{T}^*\Sigma$ of the free motion, the so-called escape functions. This kind of geometric approach is natural when dealing with the study of linear operators having continuous spectrum on standard $L^2$ spaces and it originates in the study of quantum resonances~\cite{AguilarCombes1971, BaslevCombes1971, HelfferSjostrand1986} whose fine description is related to the long time dynamics of Schr\"odinger type equations. In the context of hyperbolic dynamical systems, these geometric constructions appeared in~\cite{faure2008semi, faure2011upper, DyatlovZworski2016, faure2017semiclassical} for the study of the free motion on $\mathrm{S}^*\Sigma$. Here, we will build new escape functions adapted to the free dynamics on the full cotangent bundle $\mathrm{T}^*\Sigma$, in the spirit of these references. Namely, we introduce, study and compare to each other what we call \emph{sliced} and \emph{global} escape functions. These two kinds of escape functions fulfill different purposes which cannot be achieved at the same time. Loosely speaking, the sliced escape functions enjoy appropriate properties when restricted to $\mathrm S^\star \Sigma$, whereas the global escape functions are rather adapted to the full dynamics on  $\mrm{T}^*\Sigma$.
 
Equipped with these objects, we can introduce, in Section~\ref{sec:anisotropic}, Sobolev spaces with anisotropic regularity that are adapted to the free dynamics, referred to as \emph{sliced} and \emph{global} anisotropic Sobolev spaces. As suggested by the names, these spaces are associated with the sliced or global escape functions.
We study their main properties, most notably the exponential decay properties they induce on the free transport. The most technical statement is a bilinear estimate, adapted to the structure of the Vlasov equation, which displays exponential decay and relates sliced and global anisotropic Sobolev norms; this turns out to be one of the keys of the proof. 

The end of the paper is devoted to the actual nonlinear analysis. First, in a short Section~\ref{sec:global}, we prove the existence and uniqueness of a solution to~\eqref{eq:vlasov}. In Section~\ref{sec:mainproof}, we gather all pieces together to prove Theorem~\ref{thm:main}, using a bootstrap argument. In particular, all the abstract functional estimates of  Section~\ref{sec:anisotropic} are specifically tailored to this aim.

Finally, we provide in Appendix~\ref{a:pseudo} a detailed toolbox on the classical microlocal tools at the heart of our proofs: these are mostly used in Section~\ref{sec:anisotropic} and in~\S\ref{ss:proof-microlocal} and alluded to in the heuristic scheme of Section~\ref{sec:strategy}. Appendix~\ref{a:liverani} presents a derivation of Theorem~\ref{t:dolgopyatliverani} from~\eqref{eq:exp-mixing-classical}.

 To conclude, we point out that, all along the article, for the sake of simplicity, we consider the Vlasov equation~\eqref{eq:vlasov} in positive times but the argument can be straightforwardly adapted to negative times by considering the negative geodesic vector field which still has the Anosov property.

\subsection*{Acknowledgements} We thank Matthieu L\'eautaud for early discussions about the matter of this work back in 2019 and Yannick Guedes Bonthonneau for his insights related to the construction of escape functions. The three authors acknowledge the support of the Centre Henri Lebesgue (ANR-11-LABX-0020-01). GR is also partially supported by the Institut Universitaire de France and the PRC grant ADYCT (ANR-20-CE40-0017).

\section{An heuristic scheme of proof for a toy model}
\label{sec:strategy}

Due to the use of microlocal methods involving anisotropic symbols (adapted to the Anosov dynamics), the proof of our main theorem is fairly technically involved. Yet, its general scheme is rather natural from the perspective of nonlinear PDEs. In this section, for expository purposes, we aim at explaining the general idea on a related toy model,
without paying attention to technical details, that are left to the following of the paper.
Several difficulties already appear at this level; however, specific issues also arise for nonlinear Vlasov equations such as~\eqref{eq:vlasov} -- they will be discussed in the end of the section.

\subsection{The toy model: nonlinear perturbations of contact Anosov vector fields}\label{ss:strategy}

Let $M_1$ be a smooth compact manifold endowed with a \emph{contact} vector field $X_1$ that has the \emph{Anosov} property\footnote{As already raised in the introduction, this is in particular the case for the geodesic vector field on the unit cotangent bundle $\mathrm S^*\Sigma$ of a negatively curved surface $(\Sigma,\mathrm g)$.}. Denote by $(\varphi_t)_{t \in \R}$ the associated flow. Let $ {\mrm L}_1$ be the (contact) volume measure preserved by $X_1$ and we fix a smooth vector field $V$ that also preserves ${\mrm L}_1$.  Let us consider the toy model 
\begin{equation}\label{eq:cauchytoy}
\partial_tu=X_1u+\omega(u)Vu,\quad u|_{t=0}=u_0,
\end{equation}
where
$$
\omega(u)=\int_{M_1}u(t,z_1)\phi(z_1)\dd {\mrm L}_1(z_1)
$$
with $\phi \in \mathscr C^\infty(M_1)$ being a smooth function that is fixed once and for all. The function $\phi$ is assumed to satisfy 
$$ \int_{M_1}\phi(z_1)\dd {\mrm L}_1(z_1)=0.$$ 
By Duhamel's principle, the solution to the Cauchy problem \eqref{eq:cauchytoy} can be written as
\begin{equation}\label{eq:duhamelintro}
u(t)=\varphi_t^*u_0+\int_0^t\omega(u(s)) \varphi_{t-s}^* V(u(s))\dd s.
\end{equation}
Note that we put aside existence issues and focus on the description of the large time behavior of $u$. In order to prove weak convergence, we take a smooth test function $\psi$ and  write
$$
\langle u(t),\psi\rangle=\langle \varphi_t^*u_0,\psi\rangle+\int_0^t\omega(u(s)) \left\langle \varphi_{(t-s)}^* V(u(s)),\psi\right\rangle \dd s.
$$
The first term on the right-hand side converges (exponentially fast) to 
$$
\int_{M_1}u_0\,\dd {\mrm L}_1\int_{M_1}\psi\, \dd {\mrm L}_1
$$
thanks to the decay of correlations result for contact Anosov flows~\cite{liverani2004contact} (cf. \eqref{eq:exp-mixing-classical} in the case of geodesic flows). The second term splits in two parts by decomposing $\psi$ as
$$
\psi=\left(\psi-\int_{M_1}\psi\, \dd {\mrm L}_1\right)+\int_{M_1}\psi \, \dd {\mrm L}_1 = \mathrm{P}_1(\psi) + (\Id - \mathrm{P}_1)(\psi)
$$
Since the vector field $V$ preserves ${\mrm L}_1$, the contribution coming from $\int_{M_1}\psi \dd {\mrm L}_1$ yields a term equal to $0$. Hence we just need to understand if the integral 
\begin{equation}\label{eq:intesmall}
\int_0^t\omega(u(s)) \left\langle  u(s),V\varphi_{-(t-s)}^*\mathrm{P}_1{\psi}\right\rangle \dd s
\end{equation}
converges when $t \to \infty$.

Suppose now that we are given an essentially self-adjoint invertible operator $B$ acting on $\mathscr C^\infty(M_1)$ and $C, N_0, \vartheta > 0$ such that for each $0 \leqslant s < \infty$ we have
\begin{equation}\label{eq:intropropB}
 \|B \varphi_{{s}}^*\mathrm{P}_1 B^{-1}\|_{L^2\to L^2}\leqslant Ce^{-\vartheta s} \quad \text{and} \quad  \|B^{-1} V \varphi_{{-s}}^*\mathrm{P}_1\|_{\mathscr C^{N_0}\to L^2}\leqslant Ce^{-\vartheta s}.
 \end{equation}
 Such an operator exists thanks to~\cite{faure2011upper,nonnenmacher2015decay}, up to the fact that a loss of derivative occurs, which can be fixed by the introduction of an appropriate resolvent of the form $(X_1+1)^{-N}$, with $N$ large enough, in the estimates. In the present discussion, this technical aspect will be put aside for clarity of exposure but we will have to deal with this issue in the actual proof of our main results. 
 
 The second bound of \eqref{eq:intropropB} allows to bound the integrand of \eqref{eq:intesmall} as
 $$
 \left|\omega(u(s)) \left\langle  u(s),V\varphi_{-(t-s)}^*\mathrm{P}_1{\psi}\right\rangle\right| \leqslant C |\omega(u(s))|\,\|\varphi\|_{\mathscr C^{N_0}} \|B u(s)\|_{L^2} e^{-(t-s)\vartheta}.
 $$
 In particular, if we are able to show that for some $N>0$ and $\vartheta ' \in \left]0, \vartheta \right[$, one has
\begin{equation}\label{eq:esteps}
 \|Bu(s)\|_{L^2}\leqslant C \|u_0\|_{\mathscr{C}^N}  \quad \text{and} \quad |\omega(u(s))|=\mathcal{O}(e^{-\vartheta' s})
\end{equation}
then we would obtain that the integral \eqref{eq:intesmall} would of size $\mathcal O(e^{-\vartheta' t})$ hence converging exponentially fast towards zero.

To estimate $\omega(u(s))$, we proceed with a bootstrap argument, introducing the interval
$$
\mathcal{J}=\left\{T\in\R_+:\ \forall t\in[0,T],\ |\omega(u(s))|\leqslant e^{-\vartheta' t}\right\}.
$$
If $u_0$ is chosen small enough, this is a closed and nonempty interval. In particular, to prove the second part of \eqref{eq:esteps}, it suffices to show that $\mathcal J$ is open which would imply that $\mathcal{J}=\R_+$.
With this aim in mind, we write
$$
\omega(u(t))=\omega(\varphi_t^*u_0)+\int_0^t\omega (u(s))\left\langle u(s), V\varphi_{s-t}^*\phi\right\rangle \dd s.
$$
Then, arguing as above, one can derive the bound, for all $t \in \mathcal J$,
$$
|\omega(u(t))|\leqslant C(e^{-\vartheta t}\|u_0\|_{\mathscr{C}^{N_0}}+e^{-\vartheta' t}  \sup_{s \in \mathcal J}  \|B u(s)\|_{L^2}),
$$
since $\omega(\varphi_t^*u_0)$ converges exponentially fast towards $0$, by the decay of correlation result~\eqref{eq:exp-mixing-classical}. Therefore if one can show the bound
\begin{equation}\label{eq:estwewantintro}
\sup_{s \in \mathcal J}  \|Bu(s)\|_{L^2}\leqslant C \|u_0\|_{\mathscr{C}^N}
\end{equation}
this would imply that $\mathcal J$ is open, provided that 
$$
\|u_0\|_{\mathscr{C}^N}\ll 1,
$$
hence we would get the bound \eqref{eq:esteps}, and the desired convergence would follow.

To obtain \eqref{eq:estwewantintro}, it is natural to use energy estimates. However, if the operator $B$ is \textit{not} pseudo-differential (which will be the case when dealing with the general problem on $\mrm T^*\Sigma$), such estimates are difficult to obtain as we cannot rely on symbolic calculus. It is nevertheless possible to circumvent this problem, as follows. Suppose one can find an invertible self-adjoint pseudo-differential operator $A$, that satisfies the following properties:
\begin{itemize}
\item[(i)] the principal symbol of $A$ is of the form $e^{\lambda}$, with $\lambda \in \mathscr C^\infty(\mrm T^*M_1)$ satisfying
 \begin{equation}\label{eq:introdecay}
 \{\langle \zeta_1, X_1\rangle,\lambda\}\leqslant 0,
\end{equation}
 where $\langle \zeta_1, X_1 \rangle$ is the principal symbol associated with $X_1$;\vspace{0.1cm}
 \item[(ii)] one has the comparison estimate
 \begin{equation}\label{eq:compareintro}
 \|BVv\|_{L^2} \leqslant \|Av\|_{L^2}, \quad v \in \mathscr C^\infty(M_1).
 \end{equation}
\end{itemize}
The existence of such an operator $A$ is related to the construction of the function $\lambda$, which is referred to as an escape function for the operator $X_1$.  As we shall see below, the decay estimate \eqref{eq:introdecay} will allow to obtain polynomial bounds on $\|Au(t)\|$. Heuristically, $\|A \cdot\|_{L^2}$ can be seen as a ``top-order'' norm, while $\|B \cdot\|_{L^2}$ is ``low-order''. As is often the case in nonlinear PDEs, the low-order norm will be bounded, while we allow a growth for the top-order one. 

Since $A$ is pseudo-differential, we may now use a microlocal energy estimate, which yields
\begin{equation}
\label{eq:heuristiccommutator}
\begin{aligned}
\frac{\dd}{\dd t}\|Au(t)\|_{L^2}^2&=2\operatorname{Re}\langle AX_1u(t),Au(t)\rangle+2\omega (u(t))\operatorname{Re}\langle AVu(t),Au(t)\rangle\\
&=2\operatorname{Re}\langle [A,X_1]A^{-1}A u(t),Au(t)\rangle \\
& \qquad \qquad \qquad+2\omega (u(t))\operatorname{Re}\langle [A,V]A^{-1}Au(t),Au(t)\rangle,
\end{aligned}
\end{equation}
where we used that $V$ and $X_1$ are volume preserving in the second line. By the composition rule for pseudo-differential operators, we see that $[A,V]A^{-1}$ is (almost) a pseudo-differential operator of order $0$. Indeed, there is a logarithmic loss here and that will lead us to consider operators $A_t$ depending on time. Again, we will not discuss all this in detail to keep the presentation simple. Using the Calder\'on-Vaillancourt theorem and the Cauchy-Schwarz inequality, we obtain
$$
\frac{\dd}{\dd t}\|Au(t)\|_{L^2}^2\leqslant 2
\operatorname{Re}\langle [A,X_1]A^{-1}A u(t),Au(t)\rangle+C\omega (u(t))\|Au(t)\|_{L^2}^2.
$$
If the operator $[A,X_1]A^{-1}$ were nonpositive this would yield
$$
\frac{\dd}{\dd t}\left(e^{-\int_0^tC\omega (u(s)) ds}\|Au(t)\|_{L^2}^2\right)\leqslant 0,
$$
and hence, integrating the above expression, we would get $\|Au(t)\|_{L^2}\leqslant \widetilde{C}\|Au_0\|_{L^2}$ for $t \in \mathcal J$, for some $\widetilde C > 0$. Unfortunately, the nonpositivity of $[A,X_1]A^{-1}$ cannot be exactly achieved, but almost, thanks to the G\aa{}rding inequality, using \eqref{eq:introdecay} satisfied by $A$. 
The outcome is that the operator is nonpositive, but only up to the addition of a remainder with lower Sobolev regularity. Hence, by an induction scheme, 
 one may obtain a polynomial bound on $\|Au(t)\|_{L^2}$, that is
\begin{equation}
\label{eq:Apolysec2}
\|Au(t)\|_{L^2} \lesssim \langle t \rangle^p \|u_0\|_{\mathscr C^N}, \quad t \in \mathcal J,
\end{equation}
for some $p \in \mathbb{N}$. 

This polynomial bound implies the boundedness of $\|Bu(t)\|_{L^2}$. Indeed, use \eqref{eq:duhamelintro} to get
\begin{equation}
\label{eq:Btoy}
B\mathrm{P}_1u(t)={B}\varphi_t^*\mathrm{P}_1u_0+\int_0^t\omega(u(s))B\varphi_{(t-s)}^*\mathrm{P}_1B^{-1}B VA^{-1}Au(s)ds.
\end{equation}
Eventually, we can plug in the polynomial bound~\eqref{eq:Apolysec2} for $\|Au(s)\|_{L^2}$ into~\eqref{eq:Btoy}. Then, using the first part of \eqref{eq:intropropB} and \eqref{eq:compareintro}, we obtain the bound \eqref{eq:estwewantintro} with $Bu(s)$ replaced by $\mathrm BP_1u(s)$. However $X$ and $V$ preserve $\mathrm L_1$ and hence $B(\mathrm{Id}-\mathrm P_1)u(s)$ is constant equal to $(\mathrm{Id}-P_1)u_0$. Thus we get \eqref{eq:estwewantintro} and this completes the proof of the convergence result.


\subsection{The case of the Vlasov equation on $\mrm{T}^*\Sigma$}

This scheme of proof for Anosov vector fields lays the foundation of the strategy that we will follow when dealing with the Vlasov equation~\eqref{eq:vlasov}. The main additional difficulty that we shall encounter lies in the fact that the Vlasov equation is posed on the full cotangent bundle $M=\mrm{T}^*\Sigma$ on which the geodesic flow does not have the Anosov property. This property indeed only holds on each energy layer $M_r=\{(x,\xi):\|\xi\|=r\}$, for $r>0$. This leads to the following issues when carrying out the above strategy.
\begin{enumerate}
 \item It would be natural to consider integrated versions (over the $r$-variable) of the anisotropic pseudo-differential operators $B$ (as constructed in~\cite{faure2011upper}). This would result into operators $\mathbf B$ defined on $M$. Proceeding like this indeed allows to obtain exponential decay estimates in the spirit of \eqref{eq:intropropB}. However, doing so would have a major drawback coming from the fact that the vector fields $V$ appearing in~\eqref{eq:vlasov} involve derivatives with respect to the variable $r$. In particular, as the integrated version of $B$ does not have good pseudo-differential properties on the full cotangent bundle $M$, and, as mentioned above, we would not be able to carry out the commutator argument in~\eqref{eq:heuristiccommutator}. In order to overcome this problem, we define a new family of pseudo-differential operators $\mathbf{A}$ adapted to the dynamics on $M$. This relies on the construction of a {\it global} escape function for the dynamics lifted on $\mrm{T}^*M$ in the spirit of~\cite{faure2011upper}. This dynamical issue is tackled in Section~\ref{sec:escape} and then implemented in Section~\ref{sec:anisotropic} to define pseudo-differential operators which are tailored for the microlocal arguments used to deal with~\eqref{eq:heuristiccommutator}.
 \item A crucial point in the above strategy was the comparison property \eqref{eq:compareintro} between the norms defined by $A$ and $B$. We will construct {\it sliced} escape functions, in the spirit of~\cite{faure2011upper, nonnenmacher2015decay}, giving rise to an operator $B$ acting on functions of $\mrm S^*\Sigma$. The {\it sliced} escape functions are defined in such a way to be comparable with the previous {\it global} escape functions. This in turn allows to compare the anisotropic norms induced by $\mathbf{A}$ and the integrated version $\mathbf {B}$ (on the full cotangent bundle) of $B$. This procedure is also discussed in Section~\ref{sec:escape}. 
\end{enumerate}

Once we will have gathered these elements, we will be in position to apply the strategy of~\S\ref{ss:strategy} to the Vlasov equation on $M=\mrm{T}^*\Sigma$. This is achieved in Section~\ref{sec:mainproof} after showing in Section~\ref{sec:global} the existence and uniqueness of global solutions to the nonlinear problem~\eqref{eq:vlasov}. As a matter of fact, the above strategy can be implemented at the expense of paying attention to two more issues. First, the Anosov assumption degenerates on $\{\xi=0\}$ and we need to ensure that the solution to~\eqref{eq:vlasov} avoids this null section. This is why we consider initial data supported away from the null section. Second, the nonlinear perturbations are much more general than the $\omega(u(t))V$ considered in the above toy problem. This yields some extra (mostly technical) complications in implementing the argument.

\section{Sliced and global escape functions} 
\label{sec:escape}

In this section, we build escape functions in the spirit of \cite{faure2011upper}, which are adapted to the free flow of the Vlasov equation~\eqref{eq:vlasov}. We start in Subsection~\ref{s:diffFS} by explaining why the escape functions of~\cite{faure2011upper} cannot be directly used and why new constructions are required.
Subsection~\ref{subsec:geom} is then dedicated to geometric preliminaries related to the geodesic flow and its Anosov property; several useful notations are introduced along the way. In Subsections \ref{subsec:global} and \ref{subsec:sliced}, we introduce the key notions of {\it global} and {\it sliced} order functions together with the associated notions of global and sliced escape functions.  The actual construction of relevant sliced and global escape functions $F$ and $G$ is performed in  Subsections \ref{subsec:globalescape} to \ref{subsec:comparison}. The central result of this section is Proposition~\ref{prop:comparisonresult}, which gathers the properties of $F$ and $G$ and explains how they compare to each other.  This result is at the heart of our analytical arguments in the upcoming Sections.

\subsection{Differences with the escape functions of Faure and Sj\"ostrand}
\label{s:diffFS}
As already alluded to in the previous sections, we will introduce, in Section~\ref{sec:anisotropic}, functional spaces that are well suited to study the solutions of \eqref{eq:vlasov}. The presentation of this analytical set-up will be close to the one given by Faure and Sj\"ostrand in \cite{faure2011upper} to which we refer for more details and motivations. Recall that the key point in this reference compared with earlier related works on the spectral theory of Anosov flows \cite{liverani2004contact, butterley2007smooth, ButterleyLiverani2013} is the use of microlocal methods. Such methods turn out to be well suited for the nonlinear problems we study (see also~\cite{baladi2007anisotropic,faure2008semi} for earlier use of microlocal tools for the study of hyperbolic dynamical systems in discrete times). 
The main difference with \cite{faure2011upper} is that we have to deal with functions involving the radial variable $r=\sqrt{2H}$ in $M = {\mathrm T}^*\Sigma$ and that we do not only work with test functions in $M_1 = {\mrm S}^*\Sigma$ where the flow is Anosov. In some sense, this situation is reminiscent of what happens when studying Anosov actions \cite{bonthonneau2020ruelle}, meaning that there is an extra natural vector field $\partial_r$ which commutes with $r^{-1}X$ where $X$ is the Hamiltonian vector field generated by $\mrm H$.

A crucial ingredient in the work \cite{faure2011upper} is the construction of an \emph{escape function} $\mrm T^* M_1 \to \R$ which decreases along the flow lines of the symplectic lift 
$$
\Phi_{t, 1} : \mrm T^* M_1 \to \mrm T^* M_1
$$
of the geodesic flow $\varphi_t : M_1 \to M_1$. Since we are dealing here with the global cotangent space $M = \mrm T^*\Sigma$, a first natural idea would be to extend an escape function $\mrm T^*M_1 \to \R$ chosen as in \cite{faure2011upper} to a map $G : \mrm T^*M \to \R$, called a \textit{sliced} escape function, simply by saying that it does not depend on $r$ while it depends on $\rho$ in a trivial way (here $\rho$ is the dual variable of $r$). Proceeding like this indeed allows to build a space $\mathcal G$ for which the results of \cite{nonnenmacher2015decay} about the exponential decay of correlations can be used (at least far from the null section $\{r = 0\}$).  

However, dealing with nonlinear terms using such an approach requires an \textit{a priori} estimate on the norm of the solution of \eqref{eq:vlasov} in a space related to $\mathcal G$. Ideally, one would like to obtain such a control by energy estimates. The latter estimates require mainly two features~: first, the escape function $G$ needs to be a good symbol in $\mrm T^*M$ ; second, it needs to decrease along the flow lines of 
the symplectic lift 
$$
\Phi_t : \mrm T^* M \to \mrm T^* M
$$
of the geodesic flow $\varphi_t : M \to M$. Both properties unfortunately cannot be achieved for the sliced escape function $G$ introduced above --- for example, the decreasing property fails due to the fact that the Liouville form $\alpha$ is preserved by $\Phi_{t, 1}$ but not by $\Phi_t$.

Therefore, we are led to construct a \textit{global} escape function $F : \mrm T^*M \to \R$ which indeed satisfies these two properties. Nervertheless, although the function $F$ allows to build a space $\mathcal F$ on which one can derive appropriate energy estimates, there is no hope that one can directly apply the results of \cite{nonnenmacher2015decay} on this space as it turns out that the function $F$ cannot be seen as an escape function on $\mrm T^*M_1$. As a consequence, one is invited to make use of both spaces $\mathcal G$ and $\mathcal F$ and combine the estimates obtained on these spaces. To that aim, it will be important to be able to compare the two associated norms, that is, to compare $G$ and $F$. In particular, the construction of the two escape functions needs to be handled simultaneously in a careful way, which is one of the main difficulties of the analysis of this section.

\subsection{Geometrical setting and notations}\label{subsec:geom}
We denote $M = \mathrm{T}^*\Sigma$ and for each $r \geqslant 0$ we set
$$
M_r = \{z=(x, \xi) \in M~:~|\xi|_x = r\}.
$$
 Let $\varphi_t : M \to M$ be the geodesic flow associated to $(\Sigma, g)$. We let $X = \left.\frac{\dd}{\dd t}\right|_{t = 0} \varphi_t$ be the generator of the flow. In what follows we will denote $M^\times = M \setminus \underline 0$ for simplicity, where $\underline 0$ denotes the null section in $M$. In our development, it will be useful to use polar coordinates
\begin{equation}\label{eq:polar}
(r, z_1) : M^\times \overset{\sim}{\longrightarrow} \R_{> 0} \times M_1
\end{equation}
and where we set
$$
r(x, \xi) = |\xi|_x, \quad z_1(x, \xi) = \left(x, \frac{\xi}{|\xi|_x}\right), \quad (x, \xi) \in M^\times.
$$
Of course the flow $\varphi_t$ preserves the energy layers $\{r = c\}$ for each $c \geqslant 0$. In particular we may see $X$ as a vector field $X_r$ on $M_r$ for each $r$ and, in the polar coordinates \eqref{eq:polar}, one has $X_r(z_1)=rX_1(z_1)$.

\begin{rema}
 Observe that $X=rX_1(z_1)$ extends smoothly to a vector field defined on $\R\times M_1$.
\end{rema}
We assume that the geodesic flow $\varphi_t : M_1 \to M_1$ has the Anosov property \cite{anosov1967geodesic, FH}. This means that for each $z_1 \in M_1$ we have a decomposition
$$
\mathrm{T}^*_{z_1} M_1 = E_\urm^*(z_1) \oplus E_\srm^*(z_1) \oplus E_0^*(z_1)
$$
depending continuously on $z_1$, such that 
$$
\left|\dd \varphi_t(z_1)^{-\top} \zeta_1 \right| \leqslant 
\left\lbrace
\begin{matrix} 
C \exp\left(- \vartheta_1 t\right) |\zeta_1|, &\zeta_1 \in E_\srm^*(z_1), & t \geqslant 0, \vspace{0.2cm} \\
C \exp\left(- \vartheta_1 |t|\right) |\zeta_1|, &\zeta_1 \in E_\urm^*(z_1), & t \leqslant 0,
\end{matrix}
\right.
$$
for some constants $C, \vartheta_1 > 0$ not depending on $(z_1;\zeta_1)$. Here we have 
$$
E_0^*(z_1) = \R\alpha(z_1)\quad\text{and}\quad E_\urm^*(z_1) \oplus E_\srm^*(z_1)=\text{ker}\left(\alpha(z_1)\right)
$$
where $\alpha \in \Omega^1(M)$ is the Liouville form restricted to $M_1$, which is a contact form. We could in fact pick any smooth norm $|\cdot|$ on $\mathrm{T}^*M_1$ but we will assume that it is the Sasaki induced by $g$ in view of having orthogonality between $E_0^*$ and $\text{ker}(\alpha)$. The space $E_\urm^*(z_1)$ (resp. $E_\srm^*(z_1)$) is called the co-unstable (resp. co-stable) bundle at $z_1$. 

For each $r \geqslant 0$ we denote by $\Psi_r : \mathrm{T}^*\Sigma \to \mathrm{T}^*\Sigma$ the map $(x, \xi) \mapsto (x, r \xi)$ and we denote by $\partial_r$ the vector field on $M^\times$ given by 
$$
\partial_r = \left.\frac{\dd}{\dd r}\right|_{r = 0} \Psi_r.
$$
For each $r > 0$ and $z \in M_r$ we have $T_zM = T_zM_r \oplus \R \partial_r$ and this decomposition yields an inclusion map
$$
T_z^*M_r \hookrightarrow T_z^*M.
$$
Next, note that we have the commutation relation
$$
\varphi_t \circ \Psi_r = \Psi_r \circ \varphi_{tr}, \quad r \geqslant 0,  \quad t \in \R.
$$
In particular, for each $r > 0$, the flow $\varphi_t|_{M_r} : M_r \to M_r$ is also Anosov. Thus $E_\urm^*(z)$ and $E_\srm^*(z)$ are well defined for each $z = (x, \xi) \in M^\times$ and for such $z$ we have the decomposition
$$
\mathrm{T}^*_{z} M^\times = E_\mathrm r^*(z) \oplus E_\urm^*(z) \oplus E_\srm^*(z) \oplus E_0^*(z),
$$
where $E_\mathrm r^*(z)$ is the co-radial bundle defined by
$$
E_\mathrm r^*(z) = \left\{(x, \xi; \zeta) \in {\mathrm T}^*M^\times~:~{\langle \zeta,  {\mathrm T}^*M_{|\xi|_x} \rangle = 0}\right\}.
$$
For simplicity we will denote
$$
E_{0,\rrm}^* = E^*_\rrm \oplus E_0^*.
$$
In the polar coordinates $z=(r,z_1)$ given by \eqref{eq:polar}, this splitting reads
\begin{equation}\label{eq:coordinates-tangent}
\mathrm T_z^*M^\times
\simeq \R \dd r\oplus E_\urm^*(z_1) \oplus E_\srm^*(z_1)\oplus E_0^*(z_1)
=\R \dd r\oplus\text{ker}(\alpha(z_1))\oplus \R\alpha(z_1)
.
\end{equation}
Recall that the last splitting is smooth while the first one is only H\"older in general \cite{anosov1967tangent, FH}. The corresponding coordinates are denoted by $\zeta=(\rho;\zeta_\perp;\zeta_0)\in\R^{1+2(n-1)+1}$ and we will use the metric
\begin{equation}\label{e:metric-fibers}
|\zeta|^2=\rho^2+|\zeta_\perp|^2+\zeta_0^2
\end{equation}
on ${\mathrm T}^*_{r,z_1}M^\times$.
\begin{rema} Note that this is not exactly the Sasaki metric induced by $g$ on $M$. The standard Sasaki metric on $M$ would depend on $r$ and would rather write as
$$
|\zeta|_{\mathrm{Sasaki}}^2 = \rho^2+r^2|\zeta_\perp|^2+r^2\zeta_0^2.
$$
\end{rema}
Finally, we record the expression of the geodesic flow in the polar coordinates \eqref{eq:polar},
\begin{equation}\label{eq:flow-coordinates}
\varphi_t(r,z_1)=\left(r,\varphi_{rt}(z_1)\right),\quad (r,z_1)\in\R_{>0}\times M_1.
\end{equation}
Again, this flow can be extended to $\R\times M_1$.

\subsection{A global order function}\label{subsec:global}
Recall that $\Phi_t$ is the flow on $\mathrm{T}^*M$ given by the symplectic lift of $\varphi_t$. In other words,
$$
\Phi_t(z, \zeta) =\left(\varphi_t(z), \dd \varphi_t(z)^{-\top} \zeta\right) , \quad (z, \zeta)\in \mathrm{T}^*M,
$$
or, in the polar coordinates $(z,\zeta)\simeq(r,z_1;\rho,\zeta_\perp,\zeta_0)\in{\mathrm T}^*(\R_{>0}\times M_1)$,
\begin{equation}\label{eq:symplectic-flow-coordinates}
\Phi_t(z, \zeta)=\left(r,\varphi_{rt}(z_1);\,\rho,\, \dd \varphi_{rt}(z_1)^{-\top} \zeta_\perp,\, \zeta_0 -rt \rho\right)
\end{equation}
We denote by $\mathfrak X \in \mathscr C^\infty(\mrm T^*M,{\mathrm T}\mrm T^*M)$ the associated vector field. Our first step consists in constructing escape functions for $\Phi_t$ (at least outside compact subsets of ${\mathrm T}^*M$). 

Let $\mrm S^*M$ be the co-sphere bundle of $M$, that is
$$
{\mrm S}^*M = ({\mathrm T}^*M \setminus \underline 0) /\, \R_{>0},
$$
where $\underline 0$ denotes the null section in ${\mathrm T}^*M$. For $\zeta \in {\mathrm T}^*M \setminus \underline 0$ we denote by $[\zeta]$ its image in $\mrm S^*M$. The flow $(\Phi_t)$ induces a flow $(\widetilde \Phi_t)$ on $\mrm S^*M$ and we denote by $\widetilde{\mathfrak X}$ the associated vector field. If $E$ is a nontrivial vector subbundle of  $\mrm T^*M$, we denote by $\widetilde E \subset \mrm S^*M$ the image of $E \setminus \underline 0$ under the projection ${\mathrm T}^*M \setminus \underline 0 \to \mrm S^*M$. Also we denote
$$
\mrm T^*_{M_1}M = \{(z, \zeta) \in \mrm T^*M~:~z \in M_1\}\quad \text{and} \quad \mrm S^*_{M_1}M = \{(z, \zeta) \in \mrm S^*M~:~z \in M_1\}.
$$

\begin{defi}[Global order function] A function $\mu : \mrm T^*M^\times \to [-1, 1]$ is called a \textit{global order function} if it is of the form
\begin{equation}\label{eq:defglobord}
\mu(z, \zeta) = \wt{\mu}([z_1; \rho, \zeta_1])\, \chi\left(|\zeta|\right), \quad (z,\zeta) = (r, z_1; \rho, \zeta_1) \in \mrm T^*M^\times,
\end{equation}
for some smooth functions $\wt{\mu} : \mrm S^*_{M_1}M \to [-1,1]$ and $\chi : \R_+ \to [0,1]$ such that $\chi(t) = 0$ for $t$ small.
\end{defi}
As mentioned above, the purpose of this subsection is to construct global order functions that are adapted to the dynamics, in the following sense.
\begin{prop}\label{lem:order2}
For any $0<\delta<1$ and any small conical neighborhoods $\Upsilon_{\urm}$, $\Upsilon_\srm$ and $\Upsilon_{0,\rrm}$ of $E_\urm^*$, $E_\srm^*$ and $E_{0, \rrm}^*$ respectively, there exist a global order function $\mu$ of the form \eqref{eq:defglobord} and $\eta>0$ such that in the region $\{|\zeta|\geqslant 1\}$, we have the following properties. 
\begin{enumerate}[label = \emph{(\roman*)}]
\item $\mu = 1$ (resp. $0, -1$) near $E_\srm^*$ (resp. $E_{0, \rrm}^*$, $E_\urm^*$)\,;
\item $\mu > 1 - 2\delta$ on $\Upsilon_\srm$\,;
\item $ \mu < -1 + 2 \delta$ on $\Upsilon_\urm$\,;
\item ${\mathfrak X} {\mu} < -\eta $ outside $\Upsilon_\urm \cup \Upsilon_\srm \cup \Upsilon_{0, \rrm}$;
\item \label{item:v} ${\mathfrak X} {\mu} \leqslant 0$ everywhere.
\end{enumerate}
Any global order function satisfying the above properties for some $0<\delta <1$ and some neighborhoods $\Upsilon_{\urm}$, $\Upsilon_\srm$ and $\Upsilon_{0,\rrm}$ as above will be labeled as \emph{good}.
\end{prop}

To prove this result we will proceed with the method of \cite{faure2011upper} to which we refer if more details are needed. The main difference is that we need to pay attention to the extra variable $\rho$. We begin by constructing the function $\wt{\mu}$. 

\begin{lemm}\label{lem:order}
Let $0<\delta<1$. Then there are $\eta > 0$ and arbitrarily small neighbourhoods $\widetilde \Gamma_\urm$, $\widetilde \Gamma_\srm$, $\widetilde \Gamma_{\urm\urm}$ and $\widetilde \Gamma_{\srm\srm}$ of $\widetilde{E^*_\urm \oplus E^*_{0,\rrm}}$, $\widetilde {E^*_\srm \oplus E^*_{0, \rrm}}$, $\widetilde E^*_\urm$  and $\widetilde E^*_\srm$ in $\mrm S^*_{M_1}M$, and smooth functions $\wt{\mu}_\urm, \wt{\mu}_\srm : \mrm S^*_{M_1}M \to [0, 1]$ such that
\begin{enumerate}[label=\emph{(\roman*)}]
\item \label{point:iorder} $\mrm{\wt m}_{\srm} \equiv 1$ near $\wt E_\srm^*$ and $\mrm{\wt m}_{\urm} \equiv 1$ near $\wt E_\urm^*$ ;
\item \label{point:iiorder}$\mrm{\wt m}_{\srm} \equiv 0$ near $\widetilde{E^*_\urm \oplus E^*_{0,\rrm}}$ and $\mrm{\wt m}_{\urm} \equiv 0$ near $\widetilde{E^*_\srm \oplus E^*_{0,\rrm}}$ ;
\item \label{point:iiiorder} $\widetilde{\mathfrak X}\wt{\mu}_\srm \leqslant -\eta$ outside $\widetilde{\Gamma}^{\urm}\cup \widetilde{\Gamma}^{\srm\srm}$ and $\widetilde{\mathfrak X}\wt{\mu}_\urm \geqslant \eta$ outside $\widetilde{\Gamma}^{\srm}\cup \widetilde{\Gamma}^{\urm\urm}$;
\item \label{point:ivorder}$\wt{\mu}_\srm < \delta$ on $\widetilde \Gamma_\urm$ and $ \wt{\mu}_\urm < \delta$ on $\widetilde \Gamma_{\srm}$;
\item \label{point:vorder}$\wt{\mu}_\srm > 1 - \delta$ on $\widetilde \Gamma_{\srm\srm}$ and $\wt{\mu}_\urm > 1 - \delta$ on $\widetilde \Gamma_{\urm\urm}$;
\item \label{point:viorder}$\widetilde{\mathfrak X}\wt{\mu}_\srm \leqslant 0$ and $\widetilde{\mathfrak X}\wt{\mu}_\urm \geqslant 0$ everywhere.
\end{enumerate}
\end{lemm}

In the proof below, we will identify $\mrm T^*_{M_1}M$ with $\R_\rho \times \mrm T^*M_1$ so that a point $(1, z_1; \rho, \zeta_1) \in \mrm T^*_{M_1}M$ is identified with $(z_1; \rho, \zeta_1) \in \R_\rho \times \mrm T^*M_1$.

\begin{proof}[Proof of Lemma~\ref{lem:order}]
We introduce the strongly stable cone and the weakly unstable cone of aperture $\theta>0$ as
$$
\begin{aligned}
C^{\srm\srm}(\theta)&=\left\{(z_1;\rho,\zeta_1)\in (\R\times{\mathrm T}^*M_1)\setminus\underline{0}:\ \theta |\zeta_\srm|\geqslant |\rho \dd r+\zeta_\urm+\zeta_0\alpha|\right\},\\
C^{\urm}(\theta)&=\left\{(z_1;\rho,\zeta_1)\in (\R\times{\mathrm T}^*M_1)\setminus\underline{0}:\ \theta |\rho \dd r+{\zeta_\urm}+\zeta_0\alpha|\geqslant |{\zeta_\srm}|\right\}.
\end{aligned}
$$
Here, $\underline{0}$ denotes the set $\{(z_1;0)\}$ inside $\R_\rho \times{\mathrm T}^*M_1$. Both sets are disjoint for $0<\theta<1$ by construction and they can be identified with subsets of ${\mathrm T}^*M^\times\setminus\underline{0}$. Moreover we denote by $\wt C^{\srm\srm}(\theta)$ and $\wt C^{\urm}(\theta)$ their images in $\mathrm S^*_{M_1} M.$

Before going further, observe that, from the Anosov assumption and  \eqref{eq:symplectic-flow-coordinates}, the limit sets $\wt C^{\srm\srm}(0)$ and $\wt C^{\urm}(0)$ are respectively repulsive and attractive for the flow $\widetilde{\Phi}_t$ on $\mathrm S^*_{M_1}M$; more precisely, for each $\theta \in \left]0, 1 \right[$, there exists $\tau > 0$ such that  for each $t \geqslant \tau$, it holds
\begin{equation}\label{eq:localuniform}
\wt \Phi_t\left(\complement \wt C^{\srm\srm}(\theta)\right) \subset \wt C^\urm(\theta / 2)
\quad
\text{and}
\quad
\wt \Phi_{-t}\left(\complement \wt C^{\urm}(\theta)\right) \subset \wt C^{\srm\srm}(\theta / 2).
\end{equation}
Now let $\delta > 0$ small and $T > \tau$ (to be determined in a few lines in terms of $\delta$). Let $\wt{\mu}_0$ be a function in $\mathscr{C}^\infty({\mrm S}_{M_1}^*M,[0,1])$ that is equal to $1$ (resp. $0$) on $\wt C^{\srm\srm}(\theta)$ (resp. $\wt C^{\urm}(\theta)$). Now observe that for $T > \tau$ large enough, the sets
$$
\widetilde{\Gamma}^{\srm\srm}= S_{M_1}^*M  \setminus \left(  \bigcup_{t \in [0, T]} \widetilde{\Phi}_{-t}(\wt C^{\urm}(\theta))\right) \quad \text{and} \quad \widetilde{\Gamma}^{\urm}=S_{M_1}^*M \setminus \left( \bigcup_{t \in [0, T]} \widetilde{\Phi}_{t}(\wt C^{\srm\srm}(\theta))\right)$$
are respectively included in $\wt C^{\srm\srm}(\theta/2)$ and $\wt C^{\urm}(\theta / 2)$. Let us define
$$
\widetilde{\mu}_\srm = \frac{1}{2(T+\tau)}\int_{-T-\tau}^{T+\tau}\wt{\mu}_0\circ\widetilde{\Phi}_s\, \dd s.
$$

First, note that \emph{\ref{point:iorder}} and \emph{\ref{point:iiorder}} are automatically satisfied for $\wt{\mu}_\srm$. Indeed, set 
$$
N_\urm = \bigcap_{|t|\leqslant T + \tau} \Phi_t(\wt C^\urm(\theta)) \quad \text{and}  \quad N_{\srm\srm} = \bigcap_{|t|\leqslant T + \tau} \Phi_t(\wt C^{\srm\srm}(\theta)).
$$
Then $N_\urm$ and $N_{\srm\srm}$ are neighborhoods of $\wt {E_\urm^* \oplus \wt E_{0, \mrm r}^*}$ and $\wt E_\srm^*$ respectively and we have $\wt{\mu}_\srm([\zeta]) = 0$ for $[\zeta] \in N_\urm$ and $\wt{\mu}_\srm([\zeta]) = 1$ for $[\zeta] \in N_{\srm\srm}$.

Next, we write
\begin{equation}\label{eq:derivms}
\wt{\mathfrak X} \widetilde{\mu}_\srm = \frac{1}{2(T+\tau)}\left(\wt{\mu}_0 \circ \wt 
\Phi_{T+\tau} - \wt{\mu}_0 \circ \wt \Phi_{-T-\tau}\right).
\end{equation}
Then by \eqref{eq:localuniform} we have, for each $[\zeta] \notin (\widetilde \Gamma^\urm \cup \widetilde \Gamma^{\srm\srm})$,
$$
\widetilde\Phi_{T+\tau}([\zeta]) \in \widetilde C^\urm(\theta) \quad \text{and} \quad \widetilde\Phi_{-T-\tau}([\zeta]) \in \widetilde C^{\srm\srm}(\theta),
$$
as it follows from the definitions of $\widetilde \Gamma^\urm$ and $\widetilde \Gamma^{\srm\srm}$. Hence we get 
$$
\widetilde{\mathfrak X} \widetilde {\mu}_\srm = -\frac{1}{2(T+\tau)} \quad \text{ on } \complement (\widetilde \Gamma^\urm \cup \widetilde \Gamma^{\srm\srm}),
$$
and \emph{\ref{point:iiiorder}} is satisfied for $\widetilde {\mu}_\srm$.

Next we turn to \emph{\ref{point:ivorder}} and \emph{\ref{point:vorder}}. If $[\zeta]\in \widetilde{\Gamma}^{\urm}$, then $\widetilde{\Phi}_{-t}([\zeta])$ does not belong to $\wt C^{\srm\srm}(\theta)$  for $t \in [0, T]$. Hence $\widetilde{\Phi}_{-T}([\zeta]) \notin \wt C^{\srm\srm}(\theta)$ and thus by \eqref{eq:localuniform} we have
$
\wt \Phi_t([\zeta]) \in \wt C^\urm(\theta / 2)
$
for each $t \geqslant -T + \tau$. This gives $\widetilde {\mu}_\srm([\zeta]) \leqslant {\tau}/{2(T+\tau)}$. Similarly if $[\zeta]\in \widetilde{\Gamma}^{\srm\srm}$ one can show that $\widetilde {\mu}_\urm([\zeta]) \geqslant 1 - \tau / 2(\tau + T)$. Therefore $\wt{\mu}_\srm$ verifies \emph{\ref{point:ivorder}} and \emph{\ref{point:vorder}} by picking $T>0$ large enough.

Finally, point \emph{\ref{point:viorder}} is a consequence of \eqref{eq:localuniform} and \eqref{eq:derivms}. Indeed, if $[\zeta] \in \wt \Gamma_{\srm\srm}$, then $\wt \Phi_{-T-\tau}([\zeta]) \in \wt C^{\srm\srm}(\theta / 2)$ by \eqref{eq:localuniform} and thus $(\wt {\mu}_0 \circ \wt \Phi_{-T-\tau})([\zeta]) =1$. In particular \eqref{eq:derivms} yields $\wt{\mathfrak X} \wt {\mu}_\srm([\zeta]) \leqslant 0$. We prove similarly that $\wt{\mathfrak X} \wt {\mu}_\srm \leqslant 0$ on $\wt \Gamma_\urm$.

Considering the cones $\wt C^{\urm\urm}(\theta)$ and $\wt C^{\srm}(\theta)$ (by taking the flow in backward times), one can construct a function $\widetilde{\mu}_\urm$ satisfying \emph{\ref{point:iorder}---\ref{point:viorder}}. 

Finally note that all the neighborhoods constructed above can be made arbitrarily small, by taking $\theta$ arbitrarily small. Also $\eta$ depends only on $T$ and $\tau$, and thus on $\delta$ and $\theta$.
\end{proof}

With Lemma~\ref{lem:order}, Proposition~\ref{lem:order2} becomes at hand.

\begin{proof}[Proof of Proposition~\ref{lem:order2}]
Consider $\delta > 0$ small and $\eta, \wt{\mu}_{\mrm b}, \wt \Gamma_{\mrm b}$ and $\wt \Gamma_{\mrm b \mrm b}$ as in Lemma \ref{lem:order} for $\mrm b = \mrm u, \mrm s$. Then, set
\begin{equation}\label{eq:Upsilondef}
\wt \Upsilon_{\mrm u} = \wt \Gamma_{\mrm u} \cap \wt \Gamma_{\mrm u \mrm u}, \quad \wt \Upsilon_{\mrm s} = \wt \Gamma_{\mrm s} \cap \wt \Gamma_{\mrm s \mrm s}  \quad \text{and} \quad \wt \Upsilon_{0, \rrm} = \wt \Gamma_\urm \cap \wt \Gamma_\srm.
\end{equation}
Define the function $\widetilde{\mu}\in \mscr{C}^\infty(\mrm S^*_{M_1}M, [-1,1])$ by
\begin{equation}\label{eq:mtilde}
\wt{\mu} = \wt{\mu}_\srm - \wt{\mu}_\urm
\end{equation}
and let $\chi \in \mscr C^\infty(\R_+, [0,1])$ such that $\chi(t) = 0$ for $t \leqslant 1/2$ and $\chi(t) = 1$ for $t \geqslant 1$. Finally, define the function ${\mu}$ by \eqref{eq:defglobord}. That ${\mu}$ satisfies points \emph{\ref{point:iorder}---\ref{point:vorder}} of Proposition \ref{lem:order} is an immediate consequence of Lemma \ref{lem:order}~; indeed it suffices to take conical sets $\Upsilon_{\urm}$, $\Upsilon_\srm$ and $\Upsilon_{0,\rrm}$ in $\mrm T^*M^\times$ that project on $\wt \Upsilon_{\urm}$, $\wt \Upsilon_\srm$ and $\wt \Upsilon_{0,\rrm}$ (the latter are defined in \eqref{eq:Upsilondef}).
\end{proof}

\begin{rema}\label{rem:formglobal}
Let us record that the proof of Proposition~\ref{lem:order2} shows that any function $\mu : \mrm T^*M^\times \to [-1,1]$ given by
\begin{equation}
\mu(r, z_1; \rho, \zeta_1) =  \wt{\mu}([z_1; \rho,\zeta_1]) \chi(|\zeta|), \quad \wt{\mu} = \wt {\mu}_\srm - \wt {\mu}_\urm,
\end{equation}
where $\wt{\mu}_\srm$ and $\wt{\mu}_\urm$ satisfy the conditions of Lemma \ref{lem:order} and $\chi \in \mscr C^\infty(\R_+, [0,1])$ is such that $\chi(t) = 0$ for $t \leqslant 1/2$ and $\chi(t) = 1$ for $t \geqslant 1$, is a good global order function.
\end{rema}

\subsection{A sliced order function}\label{subsec:sliced}
In view of applying the exponential mixing properties from \cite{nonnenmacher2015decay}, we now proceed to the construction of an order function which is closer to the one used in this reference. Namely, we want to find a fonction $\nu : \mrm T^*M^\times \to [-1,1]$ that is decreasing along the \textit{sliced flow} $\Phi_{t, 1} : \mrm T^*M \to \mrm T^*M$. The latter is given by
$$
\Phi_{t,1}(z;\zeta)=\left(r,\varphi_{rt}(z_1);\,\rho,\, \dd (\varphi_{rt}|_{M_1})(z_1)^{-\top} \zeta_1\right), \quad (z;\zeta) = (r, z_1; \rho, \zeta_1) \in \mrm T^*M^\times.
$$
We denote by $\mathfrak X_1 \in \mathscr C^\infty(\mrm T^*M^\times, \mrm T \mrm T^*M^\times)$ its generator --- note that $\mathfrak X_1 = \mathfrak X + \rho\, \partial_{\zeta_0}$ in the coordinates $(r, z_1; \rho, \zeta_\perp, \zeta_0).$

\begin{defi}[Sliced order function]\label{def:slicedorder}
A \textit{sliced order function} is a map $\nu : \mrm T^*M^\times \to [-1,1]$ which is a linear combination of functions of the form
$$
(r,z_1;\rho,\zeta_1) \mapsto \widetilde {\mu}_1([z_1;\zeta_1])\,\chi\left(\frac{|\zeta_1|}{\langle \rho \rangle}\right)
$$
for some functions $\widetilde{\mu}_1 \in \mathscr C^\infty(\mrm S^*M_1, [0,1])$ and $\chi \in \mathscr C^\infty(\R, [0,1])$ such that $\chi(t) = 0$ for $t$ small.
\end{defi}

Given a sliced order function $\nu$, we can define the corresponding function on $\mrm T^*M_1$ by letting 
\begin{equation}\label{eq:renormalizedsliced}
\nu_1(z_1;\zeta_1)=\nu\left(z_1,0,\zeta_1\right).
\end{equation}
As a consequence of Proposition \ref{lem:order2}, one has the following result.

\begin{corr}\label{lem:order3}
There exist $\eta > 0$ and a sliced order function $\nu : \mrm T^*M^\times \to [-1,1]$ such that the function $\nu_1 = \nu|_{\mrm T^*M_1}: \mrm T^*M_1 \to [-1,1]$ satisfies the following properties, in the region $\{|\zeta_1|\geqslant 1\}$.
\begin{enumerate}[label = \emph{(\roman*)}]
\item $\nu_1$ is equal to $1$ (resp. $0, -1$) in a conical neighborhood of $E_\srm^*$ (resp. $E_0^*, E_\urm^*$)\,; 
\item $\mathfrak X_1 \nu_1 \leqslant 0$\,;
\item $\mathfrak X_1 \nu_1 < 0$ outside a conical neighborhood of $E_0^* \cup E_\urm^* \cup E_\srm^*$.
\end{enumerate}
A sliced order function satisfying the properties above will be labeled as \emph{good}.
\end{corr}

\begin{proof}
Take $\mu : \mrm T^*M \to [-1,1]$ a good global order function given by Proposition \ref{lem:order2} and any $\chi \in \mathscr C^\infty(\R, [0,1])$ such that $\chi(t) = 0$ for $t$ small and $\chi(t) = 1$ for $t \geqslant 1$. Then set
$$
\nu(r, z_1; \rho, \zeta_1) = \mu|_{\mrm T^*M_1}(z_1; \zeta_1) \chi\left(\frac{|\zeta_1|}{\langle \rho \rangle}\right).
$$
Then $\nu$ is a sliced order function. Moreover, since $\mathfrak X_1 = \mathfrak X$ on $\mrm T^*M_1 = \{\rho = 0\}$, it is straightforward to check that $\nu$ is a good one.
\end{proof}

As a matter of fact, we will need to find good global and sliced order functions which are comparable. This is the purpose of the next subsection.

\subsection{Comparison of order functions}

\begin{lemm}\label{lem:comparisonorder} Let $0<\delta<1/2$. Then, one can find a global order function $\mu=\mu_{\srm}-\mu_{\urm}$ as in Remark~\ref{rem:formglobal} and $T_1,\beta>1$ such that the following holds.

Letting
$
\wt{\mu}_{\srm, 1} = \mu_{\srm}|_{\mrm S^*M_1}$,  $\wt{\mu}_{\urm, 1} = \mu_{\urm}|_{\mrm S^*M_1},$
\begin{equation}\label{eq:mus}
\nu_\srm(z_1; \rho, \zeta_1) = (\wt{\mu}_{\srm, 1}\circ\wt \Phi_{T_1})([z_1; \zeta_1])\,\chi\left(\frac{|\zeta_1|}{\beta\langle\rho\rangle}\right),
\end{equation}
and
\begin{equation}\label{eq:muu}
\nu_\urm(z_1; \rho, \zeta_1) = (\wt{\mu}_{\urm, 1}\circ\wt \Phi_{T_1})([z_1; \zeta_1])\,\chi\left(\frac{\beta|\zeta_1|}{\langle\rho\rangle}\right),
\end{equation}
one has the following properties:
\begin{enumerate}[label = \emph{(\roman*)}]
 \item $\nu=\nu_\srm-\nu_\urm$ is a good sliced order function;
 \item $\nu_\srm\leqslant  \mu_\srm$ and $\mu_\urm\leqslant \nu_\urm$ for $|\zeta|\geqslant 1$;
 \item $\nu_\srm>0$ implies that $|\zeta_\perp|^2\geqslant \rho^2+\zeta_0^2.$
\end{enumerate}

\end{lemm}

\begin{proof}
By items \emph{\ref{point:iorder} \emph{and} \ref{point:iiorder}} of Lemma \ref{lem:order}, for $T_1 > 0$ large enough it holds
$$
\supp(\wt{\mu}_{\srm, 1}\circ\wt \Phi_{T_1}) \subset \operatorname{int} \{\wt{\mu}_{\srm, 1} = 1\}.
$$
Hence for $T_1$ large, $\supp(\wt{\mu}_{\srm, 1}\circ\wt \Phi_{T_1})$ is contained in the interior of $\{\wt{\mu}_{\srm} = 1\}$. In particular taking $\beta > 0$ large enough, we have for $|\zeta|\geqslant 1$
$$
\nu_\srm(z_1; \rho, \zeta_1) > 0 \quad \implies \quad {\mu}_\srm(z_1;\rho, \zeta_1) = 1.
$$
In particular this yields $\nu_\srm \leqslant \mu_\srm$ for $|\zeta|\geqslant 1$. 

Let $\pi : \mrm S^*_{M_1}M \setminus \wt {E_{\mrm r}^*} \to \mrm S^*M_1$ be the map $[(1, z_1; \rho, \zeta_1)] \mapsto (z_1;  \zeta_1 / |\zeta_1|)$, and let $K = \overline{\pi(\supp(\wt{\mu}_{\urm}))} \subset S^*M_1$, which is well defined since $\supp \wt{\mu}_\urm$ does not intersect $\wt{E_\mrm{r}^*}$. Then $ {K \cap \wt{E_\srm^*\oplus E_0^*} = \emptyset}$ by \emph{\ref{point:iiorder}} of Lemma \ref{lem:order2} and thus for $T_1$ large enough, we have
$$
\{\wt{\mu}_{\urm, 1}\circ\wt \Phi_{T_1} = 1\} \supset K.
$$
Therefore, if $\beta > 0$ is chosen large enough, we have $\nu_{\urm} \geqslant {\mu}_{\urm}$ on $\{|\zeta|\geqslant 1\}$. This completes the proof of the first two items. For the third item, this follows by construction (up to increasing the values of $\beta$ and $T_1$). 
\end{proof}


\subsection{Global and sliced escape functions}\label{subsec:globalescape}

\begin{defi}
A \textit{good global escape function} is a smooth map $F : \mrm T^*M^\times \to \R$ independent of $r$, which is a symbol in $\overline{S}^{0+}(\mrm T^*M^\times)$, and such that
$$
\mathfrak X F \leqslant 0 
$$
in the region $\{|\zeta|\geqslant 1\}.$
\end{defi}
Since $F$ does not depend on $r$, the condition $F \in \overline{S}^{0+}(\mrm T^*M^\times)$ means that for every $\varepsilon > 0$ and $\beta, \gamma \in \mathbb N^n$, there exists $C_{\beta, \gamma, \varepsilon} >0$ such that
$$
|\partial^\beta_z \partial^\gamma_\zeta F(z, \zeta)|\leqslant C_{\beta, \gamma, \varepsilon} \langle \zeta \rangle^{\varepsilon - |\gamma|}, \quad (z, \zeta) \in \mrm T^*M^\times.
$$
Again these symbols can be viewed as symbols on $\mrm{T}^*(\R\times M_1)$. Note that we only impose a nonincreasing condition (and not a decreasing condition) along the flow lines of $\Phi_t$ because it will be sufficient to show that the solution of \eqref{eq:vlasov} grows at most polynomially in time in a space defined by a good global escape function $F$ (see the proof of Lemma~\ref{l:polynomialbound} below). 





We can now state the main result of this section. 

\begin{prop}\label{prop:comparisonresult}
For each $\delta \in \left]0, 1\right[$, there are good global and sliced order functions $\mu$ and $\nu$, as well as positive symbols $f \in \overline{S}^2(\mrm T^*M^\times)$ (independent of $r$) and $g_1 \in \overline{S}^2(\mrm T^*M_1)$, such that the following holds. 

Let $g : \mrm T^*M^\times \to \R$ be defined by
$$
g(r,z_1;\rho, \zeta_1) = g_1(z_1; \zeta_1/\langle \rho \rangle).
$$
Let also $\mathfrak S > 0$ and $\sigma \in \R$ such that $|\sigma| < (1-\delta)\mathfrak S$ and denote
$$
F\Ss = (\mfS \mu + \sigma) \log f \quad \text{and} \quad G\Ss = (\mfS \nu + \sigma) \log g.
$$
Then we have the following properties:
\begin{enumerate}[label=\emph{(\roman*)}]
\item there exists $c>0$ such that $g_1(z_1;\zeta_1)\geqslant c\langle\zeta_1\rangle^2$;
\item $F\Ss$ is a global escape function\,;
\item for $|\zeta_1|\geqslant 1$, we have 
$$
\mathfrak S\nu(z_1,0,\zeta_1) = \mathfrak S \quad \text{near }E_\srm^* \qquad \text{and} \qquad \mathfrak S\nu(z_1,0,\zeta_1) = -\mathfrak S \quad \text{near }E_\urm^*;
$$
\item for $|\zeta_1|\geqslant 1$, it holds $\mathfrak{X}_1(\nu)(z_1,0,\zeta_1)\leqslant 0$;
\item whenever $\sigma \leqslant 0$, we have
$$
G_{\mathfrak S, \sigma} \leqslant F_{\mathfrak S, \sigma}  + 2(\mathfrak S + |\sigma|) \log \langle \rho \rangle
$$
in the region $\bigl\{|\zeta|\geqslant 1\bigr\}.$
\end{enumerate}
\end{prop}
By analogy with the definition of a good global escape function, we will say in the following that $G_{\mathfrak S,\sigma}$ is a \emph{good sliced escape function} even if it is maybe not decreasing along the flow lines. The rest of this section is devoted to the proof of Proposition \ref{prop:comparisonresult}.

\subsection{Construction of a global pseudo-norm}

Before going further, let us introduce some additional norms on $\mrm{T}^*M_1$. Given $T>1$, we set
$$
N_T(z_1; \zeta_1)=\int_0^{T} |\Phi_t(z_1; \zeta_\perp)|^2\dd t.
$$
Then note that we have
$$
\mathfrak{X}_1N_T(z_1; \zeta_1)=|\Phi_T(z_1;\zeta_\perp)|^2-|\zeta_\perp|^2.
$$
Hence, thanks to the Anosov property, we can fix $T>1$ large enough to ensure the existence of two conic neighborhoods $\Gamma_{\urm\urm}^1$ and $\Gamma_{\srm\srm}^1$ of $E_\urm^*$ and $E_\srm^*$ respectively such that
\begin{equation}\label{e:lyapunov-norm}
\mathfrak{X}_1 N_T \geqslant c_{\Sigma,\mathrm{g}}N_T  \quad \text{in } \Gamma_{\urm\urm}^1 \qquad \text{and} \qquad \mathfrak{X}_1 N_T \leqslant -c_{\Sigma,\mathrm{g}}N_T \quad \text{in } \Gamma_{\srm\srm}^1
\end{equation}
where $c_{\Sigma,\mathrm{g}}>0$ is a constant depending only on the Riemannian manifold. 
The neighborhood and parameter $T>1$ appearing in the above definition are fixed once and for all in this rest of this section.

Let $ Q \geqslant  1$ be some parameter (to be determined later) and set
$$
\widetilde{g}_1(z_1;\zeta_1)=\chi\left(\frac{|\zeta_0|^2}{N_T(\zeta_1)}\right)\zeta_0^2  +\left(1-\chi\left(\frac{|\zeta_0|^2}{N_T(\zeta_1)}\right)\right)N_T(\zeta_1),
$$
where we take $\chi(t)=1$ for $t\geqslant 1$ and $\chi(t)=0$ for $|t|\leqslant 1/2$. We also define
\begin{equation}\label{eq:deff}
f(r, z_1; \rho, \zeta_1) = Q + \chi\left(\frac{|\zeta_\perp|^2}{\zeta_0^2 + \rho^2}\right)\widetilde{g}_1(z_1;\zeta_1), \quad (r, z_1; \rho, \zeta_1) \in \mathrm T^*M^\times.
\end{equation}
One can check that this function belongs to $\overline{S}^2(\mrm{T}^* M^\times)$; we also note that it is independent of $r$. The next results says that $F_{\mathfrak S,\sigma}$ is indeed a good global escape function for an appropriate choice of $Q$. 
\begin{lemm}\label{lem:comparison} Let $0<\delta<1/2$. Then, one can find a global order function $\mu$ as in Lemma~\ref{lem:comparisonorder} and $Q_0> 1$ such that for all $Q\geqslant Q_0$, and all $\mathfrak S\geqslant 1$ and $\sigma\in\R$ verifying $|\sigma|\leqslant (1-2\delta)\mathfrak S$, one has
\begin{equation}\label{eq:decreasing}
\mathfrak X\bigl((\mathfrak S\mathrm m+\sigma)\log f\bigr) \leqslant 0
\end{equation}
in the region $\{|\zeta|\geqslant 1\}$.
\end{lemm}

\begin{proof} Let  $\mu$ be a global order function as obtained from Lemma~\ref{lem:comparisonorder}. By definition, for $|\zeta|\geqslant 1$, one has
$$
\mu(r,z_1;\rho,\zeta_1)=\widetilde{\mu}\left([z_1;\rho,\zeta_1]\right)=\widetilde{\mu}_\srm\left([z_1;\rho,\zeta_1]\right)-\widetilde{\mu}_\urm\left([z_1;\rho,\zeta_1]\right).
$$
Hence, we have
$$
\mathfrak X\bigl((\mathfrak{S}{\mu}+\sigma) \log f\bigr) 
= \mathfrak S(\mathfrak X \widetilde{\mu}) \log f + (\mathfrak S \widetilde{\mu} +\sigma)\frac{\mathfrak X f}{f}.
$$
Recall that by construction we have, in the region $\{|\zeta|\geqslant 1\}$, 
\begin{equation}\label{e:trivialbound}
 \mathfrak{X}\widetilde{\mu}\leqslant 0\quad \text{and}\quad |\mathfrak{X}f/f|\leqslant r\,C_{T},
\end{equation}
for some constant $C_T>0$ depending only on $T$ and $(\Sigma,\mathrm{g})$. In particular, this constant does not depend on $Q$ and $r$. By picking the neighborhoods $\widetilde{\Gamma}_{\urm\urm}$, $\widetilde{\Gamma}_{\srm\srm}$, $\widetilde{\Gamma}_\urm$ and $\widetilde{\Gamma}_\srm$ in Lemma~\ref{lem:order} small enough, we can verify that, there exist small conic neighborhoods $Y_{\urm\urm}, Y_{\srm\srm}$ and $Y_{0,\rrm}$ of $E_\urm^*$, $E_\srm^*$ and $E_0^*\oplus E_\rrm^*$ such that
\begin{itemize}[label=\tiny\textbullet]
 \item in $\{|\zeta|\geqslant 1\}\cap \complement(Y_{\urm\urm} \cup Y_{\srm\srm} \cup Y_{0, \rrm})$, one has $\mathfrak{X} \mu \leqslant -\eta r$;
 \item the quantity $\displaystyle\chi\left(\frac{|\zeta_\perp|^2}{\zeta_0^2 + \rho^2}\right)$ is equal to $0$ in $Y_{0, \rrm}$ and equal to $1$ in $Y_{\urm\urm} \cup Y_{\srm\srm}$;
 \item the quantity $\displaystyle\chi\left(\frac{|\zeta_0|^2}{N_T(\zeta_1)}\right)$ is equal to $0$ in $Y_{\urm\urm} \cup Y_{\srm\srm}$;
 \item there holds $\mu \geqslant 1 - 2\delta$ in $Y_{\srm\srm}$ and $\mu \leqslant 2\delta$ in $Y_{\urm\urm}$.
\end{itemize}
In particular, thanks to~\eqref{e:trivialbound}, we can deduce on the one hand that $\mathfrak X(\mathfrak{S}{\mu}+\sigma)\leqslant 0$ in $Y_{0,\rrm}$. On the other hand one has, in $\complement(Y_{\urm\urm} \cup Y_{\srm\srm} \cup Y_{0, \rrm})$,
$$
\mathfrak X(\mathfrak{S}{\mu}+\sigma)\leqslant -\mathfrak S \eta r\log Q+C_T r(\mathfrak S+|\sigma|)\leqslant r \mathfrak S\left(-\eta \log Q+ 2C_{T}\right).
$$
Hence, picking $Q>1$ large enough, we obtain \eqref{eq:decreasing} outside $Y_{\urm\urm} \cup Y_{\srm\srm}$. In these two regions, we use that $\mathfrak{X}\widetilde{\mu}\leqslant 0$ to obtain
$$
\mathfrak X\bigl((\mathfrak{S}{\mu}+\sigma) \log f\bigr) \leqslant (\mathfrak S \widetilde{\mu} +\sigma)\frac{\mathfrak X f}{f}.
$$
However by Lemma~\ref{lem:order}, one has $\mathfrak S {\mu} +\sigma\geqslant \mathfrak S(1-2\delta)+\sigma$ in $Y_{\srm\srm}$ while, in $Y_{\urm\urm}$, one has $\mathfrak S \widetilde{\mu} +\sigma\leqslant -\mathfrak S(1-2\delta)+\sigma$. Finally, by~\eqref{e:lyapunov-norm}, one finds that, in $Y_{\urm\urm} \cup Y_{\srm\srm},$
$$
\mathfrak X((\mathfrak{S}{\mu}+\sigma) \log f) \leqslant -rc_{\Sigma,\mathrm{g}}\left(\mathfrak S(1-2\delta)-|\sigma|\right).
$$
The right-hand side of the above inequality is indeed nonpositive as soon as the condition $|\sigma|\leqslant \mathfrak S(1-2\delta)$ holds.
\end{proof}

\subsection{Proof of Proposition~\ref{prop:comparisonresult}}\label{subsec:comparison}

We are now in position to prove Proposition~\ref{prop:comparisonresult}. We let $\mu$ be the order function appearing in Lemma~\ref{lem:comparison} and $Q>1$ be the parameter appearing in that same Lemma. The second item of Proposition~\ref{prop:comparisonresult} follows from this Lemma. Now letting $g_1=Q_0+\widetilde{g}_1$ and $\nu$ be the sliced order function from Lemma~\ref{lem:comparisonorder}, we directly obtain all the remaining items except for the last one (that is the comparison between $G_{\mathfrak S,\sigma}$ and $F_{\mathfrak S,\sigma}$). To deal with this point, note that we have
\begin{equation}\label{eq:borneinf}
\log f - 2 \log \langle \rho \rangle \leqslant \log g
\end{equation}
everywhere, since $f \leqslant\langle\rho \rangle^2g$ by \eqref{eq:deff} and the definition of $g_1$. In particular, one has, for $\sigma\leqslant 0$,
$$
\sigma\log g \leqslant \sigma\log f +2|\sigma|\log\langle\rho\rangle,
$$
and, in the region in the region $\{\nu \leqslant 0\}$, 
$$
\mathfrak S \nu \log g \leqslant \mathfrak S \nu\log f +2|\mathfrak S|\log\langle\rho\rangle\leqslant \mathfrak S \mu\log f +2|\mathfrak S|\log\langle\rho\rangle,
$$
where we used Lemma~\ref{lem:comparisonorder} to write the last inequality. Recall that $\nu = \nu_\srm - \nu_\urm$ where $\nu_\srm$ and $\nu_\urm$ are defined in \eqref{eq:mus} and \eqref{eq:muu}. According to the third item in Lemma~\ref{lem:comparisonorder}, one has
$$
\{\nu_\srm > 0\}\subset \left\{(z;\zeta)~:~\chi\left({|\zeta_\perp|^2}/(\rho^2+\zeta_0^2)\right) = 1\right\}.
$$
As $\{\nu > 0\} \subset \{\nu_{\srm} > 0\}$, we obtain that, in $\{\nu > 0\} \cap\{|\zeta|\geqslant 1\}$,
$$
\mathfrak S\nu\log g\leqslant \mathfrak S \nu\log f\leqslant \mathfrak S \mu\log f,
$$
simply because $\nu\leqslant \mu$ and $g\leqslant f$ in that region. This completes the proof the last item and thus the proof of the Proposition.

\section{Anisotropic Sobolev norms adapted to the free dynamics}
\label{sec:anisotropic}

The global and sliced escape functions from Section~\ref{sec:escape} will now be used to design anisotropic Sobolev norms that will be at the heart of the proof of Theorem~\ref{thm:main}. In this Section, we define the associated anisotropic norms, which are adapted to the free dynamics and establish some of their main properties. 

In Subsection~\ref{sec:4.1}, we start by explaining how the global (resp. sliced) escape functions give rise to the class of pseudo-differential operators $\mathbf{A}_{\mathfrak S_0,\sigma_0}$ (resp. $\mathbf{B}_{\mathfrak S_1,\sigma_1}^{(m_1)}$), which in turn allow to define the \emph{global} (resp. \emph{sliced}) anisotropic norms, see Definition~\ref{def:global-anisotropic-sobolev} (resp. Definition~\ref{def:sliced-anisotropic-sobolev}). In Subsection~\ref{sec:4.2}, we state the key result of this section, that is the bilinear estimate of Proposition~\ref{p:bilinear}. This estimate is tailored to handle the nonlinear part of the Vlasov equation~\eqref{eq:vlasov}; loosely speaking, it shows an exponential decay (obtained from the strong mixing of the geodesic flow) which shows that the nonlinear part might be treated in a perturbative way. It also features an estimate of the action of the operators $\mathbf{B}_{\mathfrak S_1,\sigma_1}^{(m_1)}$ by the anisotropic Sobolev norm associated with $\mathbf{A}_{\mathfrak S_0,\sigma_0}$ (that is, a control of a sliced anistropic Sobolev norm by a global anistropic Sobolev norm).

The remainder of the section is then mostly dedicated to the proof of Proposition~\ref{p:bilinear}; along the way, we also prove some estimates that will turn out handy for the upcoming bootstrap analysis.
The purpose of Subsections~\ref{sec:4.2} and~\ref{sec:4.3} is to show that our sliced escape functions (and their associated pseudo-differential operators) are well adapted to the abstract framework of \cite{nonnenmacher2015decay}, in view of obtaining a refined microlocal version of the decay of correlations of Theorem~\ref{t:dolgopyatliverani}. This verification is first thoroughly performed on $\mathrm S^\star \Sigma$ in Subsection~\ref{sec:4.2} and then lifted on the cotangent space (minus the null section) in Subsection~\ref{sec:4.3}.

Eventually, the proof of  Proposition~\ref{p:bilinear} is concluded in Subsections~\ref{sec:4.4} and~\ref{sec:4.6}. In particular we need to establish a connection between the sliced and global anisotropic Sobolev norms, using in a crucial way the comparison properties between the global and sliced escape functions, that were obtained in Proposition~\ref{prop:comparisonresult}.

\subsection{Definitions} 
\label{sec:4.1}
Let $\mathfrak S_0>0$ and $|\sigma_0|<(1-\delta) \mathfrak S_0$. The function $e^{F_{\mathfrak S_0,\sigma_0}}$ built from our global escape function $F_{\mathfrak S_0, \sigma_0}$ given by Proposition \ref{prop:comparisonresult} is a function on $\R\times{\mathrm T}^*M_1$ that we identified with a function on 
$$
{\mathrm T}^*M^\times\simeq{\mathrm T}^*(\R_{>0}\times M_1)\subset{\mathrm T}^*\R\times{\mathrm T}^*M_1.
$$ 
It defines a symbol in the standard class $\overline{S}^{2\mathfrak S_0}({\mathrm T}^*(\R\times M_1))$ of functions amenable to pseudo-differential calculus. See Appendix~\ref{aa:mfd} for the precise definition. We can then define
\begin{equation}\label{eq:anisotropicoperator}
\mathbf{A}_{\mathfrak S_0,\sigma_0}=\mathbf{Op}\left(e^{F_{\mathfrak S_0,\sigma_0}}\right),
\end{equation}
where $\mathbf{Op}$ is a quantization procedure on ${\mathrm T}^*(\R\times M_1)$ \cite[Ch.~14]{zworski}. In fact, due to the product structure of our space, we can set 
$$
\mathbf{A}_{\mathfrak S_0,\sigma_0}(u)(r,z_1)=\frac{1}{2\pi }\int_{\R^2}e^{i(r-s)\rho}{\oM}\left(\exp {F_{\mathfrak S_0,\sigma_0}(\rho,\, \cdot)}\right)u(s,z_1)\dd s\dd\rho,
$$
where ${\oM}$ is a quantization procedure on $M_1$. Again refer to Appendix~\ref{aa:mfd} for more details. These operators are well defined for every $u\in\mathscr{C}^\infty_c(\R\times M_1)$ (hence compactly supported functions on $M^\times$) and they map the set $\mathscr{S}(\R\times M_1)$ of smooth functions that are rapidly decaying in $r$ to itself. 

\begin{defi}
\label{def:global-anisotropic-sobolev} The \emph{global anisotropic Sobolev norms} are, for $\mathfrak S_0,\sigma_0 \in \R$,  the norms
$$
\left\|\mathbf{A}_{\mathfrak S_0,\sigma_0}u\right\|_{L^2(\R\times M_1)}, \quad u\in\mathscr{C}^{\infty}_c(\R\times M_1),
$$
where the $L^2$-norm is taken with respect to the measure $\dd r\,\dd L_1(z_1)$ rather than the Liouville measure $r^{n-1}\dd r\,\dd L_1(z_1)$ that we would extend to $\R\times M_1$. 

The associated global anisotropic Sobolev spaces are defined as the completion of $\mathscr{C}^{\infty}_c(\R \times M_1)$ for these norms in the space of tempered distributions.

\end{defi}
One key aspect of the analysis of~\eqref{eq:vlasov} will be to use these global anisotropic Sobolev norms, which are adapted to the dynamical properties of the free flow, to perform microlocal energy estimates.


\begin{rema}
This ``renormalized'' measure is still invariant by the vector field $X=rX_1$ and we will always pick it for the $L^2$-scalar product. The choice of this measure slightly simplifies the use of a quantization procedure adapted to $\R\times M_1$. Yet, at some points, it will add extra constants in our upper bounds depending on the support of the test function $u$. 
\end{rema}

In order to make the connection with the exponential decay results from~\cite{nonnenmacher2015decay}, we also introduce a family of slightly more exotic pseudo-differential operators built from our sliced escape function. Namely, let $\mathfrak S_1>0$ and $|\sigma_1|<(1-2\delta)\mathfrak S_1$. The escape function given by 
$$
G_{\mathfrak S_1,\sigma_1,1}(z_1;\zeta_1)=G_{\mathfrak S_1,\sigma_1}(r,z_1;0,\zeta_1)
$$
is independent from $r$ and belongs to the class of symbols $\overline{S}^{0+}(\mrm{T}^*M_1)=S^{0+}(\mrm{T}^*M_1)$ amenable to pseudo-differential calculus on the \emph{compact} manifold $M_1$. In particular, for every $s\in\mathbb{R}$, the operator $\exp {\oM}(G_{\mathfrak S_1,\sigma_1,1})$ defines a bounded operator from $H^s(M_1)$ to $H^{s-2\mathfrak S_1}(M_1)$. See Appendix~\ref{aaa:exponential} for a brief reminder and references. This allows to define the operator-valued pseudo-differential operator
\begin{equation}\label{eq:operatorvaluedanisotropic}
\forall m_1\in\R,\quad\mathbf{B}_{\mathfrak S_1,\sigma_1}^{(m_1)}(u)(r,z_1)=\frac{1}{2\pi }\int_{\R^2}\frac{e^{i(r-s)\rho}}{(1+\rho^2)^{\frac{m_1}{2}}}\exp\left({\oM}\left(G_{\mathfrak S_1,\sigma_1,1}\right)\right)u(s,z_1)\dd s\,\dd \rho.
\end{equation}
See Appendix~\ref{aa:operatorvalued} for a brief reminder for the pseudo-differential calculus of operator-valued symbols. This operator admits an inverse which is given by
$$
\mathbf{B}_{-\mathfrak S_1,-\sigma_1}^{(-m_1)}(u)(r,z_1)=\frac{1}{2\pi }\int_{\R^2}\frac{e^{i(r-s)\rho}}{(1+\rho^2)^{-\frac{m_1}{2}}}\exp\left(-{\oM}\left(G_{\mathfrak S_1,\sigma_1,1}\right)\right)u(s,z_1)dsd\rho.
$$
Recall from~\cite[Lemma~3.3]{faure2011upper} that we can pick $C_{1}>0$ large enough (depending on $\mathfrak S_1$) to ensure that, for every $\sigma_1\in[-2,2]$ and for every $\operatorname{Re}(z)\geqslant C_{1}$,
\begin{equation}\label{eq:FSresolventbound}
\left\|e^{{\oM}\left(G_{\mathfrak S_1,\sigma_1,1}\right)}\left(X_1+z\right)^{-1}e^{-{\oM}\left(G_{\mathfrak S_1,\sigma_1,1}\right)}\right\|_{L^2(M_1)\rightarrow L^2(M_1)}<\infty.
\end{equation}
Hence, for every integer $N_1\geqslant 1$, we can set
\begin{equation}\label{eq:operatorvaluedanisotropicresolvent}
\mathbf{B}_{\mathfrak S_1,\sigma_1}^{(m_1,N_1)}=\mathbf{B}_{\mathfrak S_1,\sigma_1}^{(m_1)}(X_1+C_1)^{-N_1}.
\end{equation}
Here the operator $X_1$ is identified with an operator acting on $\R\times M_1$. 
\begin{defi}
\label{def:sliced-anisotropic-sobolev}
The \emph{sliced anisotropic Sobolev norms} are, for $m_1, \mathfrak S_1,\sigma_1 \in \R$, $N_1 \geqslant1$,  the norms
$$
\left\| \mathbf{B}_{\mathfrak S_1,\sigma_1}^{(m_1,N_1)} u\right\|_{L^2}, \quad u\in\mathscr{C}^{\infty}_c(\R\times M_1).
$$
The associated sliced anisotropic Sobolev spaces are defined as the completion of the space $\mathscr{C}^{\infty}_c(\R \times M_1)$ for these norms.
\end{defi}

\begin{rema}\label{r:faure-sjostrand}
Recall that $X=rX_1$ where $X=\frac{\dd}{\dd t}\varphi_t^*|_{t=0}$. Theorem 1.4 in~\cite{faure2011upper} implies the meromorphic extension of the resolvent in~\eqref{eq:FSresolventbound} up to $
\text{Re}(z)\geqslant -c_0\mathfrak S,$ 
where $c_0>0$ is a geometric constant depending only on $(\Sigma,\mathrm{g})$ according to Proposition~\ref{prop:comparisonresult} and $\mathfrak S$ is the regularity parameter in their anisotropic spaces. In fact, there is an additional constant $C$ (still independent of $(\nu,g)$ defining the escape function) in the statement of this reference which comes from the fact the vector field may not be volume preserving. Here, due to the particular form of our quantization procedure (with isochore charts) and to the fact that $X_1$ is volume preserving, it holds in fact $$
X_1={\oM}\bigl(\langle \zeta_1, X_1\rangle\bigr)
$$
(see Remark~\ref{r:explicitsymbol}). Therefore the constant $C$ from~\cite[Th.~1.4]{faure2011upper} can be taken equal to $0$ (see Equation~$(3.8)$ in this reference). It should be noted that the escape function in~\cite{faure2011upper} satisfies slightly stronger requirements that the ones satisfied by $G_{\mathfrak S_1,\sigma_1,1}$. Yet, using an approach through radial type estimates, these requirements were simplified by Dyatlov and Zworski in~\cite[\S 3]{DyatlovZworski2016} and one only needs that $G_{\mathfrak S_1,\sigma_1,1}$ satisfies the weaker properties stated in Proposition~\ref{prop:comparisonresult} (namely the first, third and fourth item of this statement). In that case, one can verify that the resolvent can be extended to the region
\begin{equation}\label{eq:strip-meromorphic-continuation}
\text{Re}(z)> -c_0(\mathfrak S_1-|\sigma_1|),
\end{equation}
where $c_0>0$ is some geometric constant depending only on $(\Sigma,\mathrm{g})$.
We refer to~\cite[\S3.3.1, Th.~2]{Dyatlov2023} for more details on the value of the constant $c_0$ in terms of the Lyapunov exponents of the flow.
\end{rema}

\subsection{A bilinear estimate}
\label{sec:4.2}

The main result of this section is the following bilinear estimate, which will be crucial in the upcoming nonlinear analysis.  

\begin{prop}
\label{p:bilinear} There exists $\vartheta_0>0$ such that the following holds.

 Let $\mathfrak S_0, m_1, N_1 \geqslant 0$ such that
$$
\mathfrak{S}_0\geqslant 3, \quad m_1\geqslant 50 \mathfrak{S}_0 \quad \text{and} \quad N_1\geqslant 50\mathfrak S_0+m_1.
$$
Let also $I \subset \R$ be an interval with $\overline{I}\subset\R_+^*$ and $\chi\in\mathscr{C}^\infty_c(I)$. 

Then there exists a constant $C>0$ such that, for all $\Phi\in\mathscr{C}^\infty(\Sigma)$, and for all $u\in \mathscr{C}^\infty_c(\R\times M_1)$,
\begin{equation}\label{eq:bilinearprop}
\left\|\mathbf{B}_{\mathfrak S_0,-2}^{(m_1,N_1)}\varphi_{-t}^* \mathbf{P}\{\Phi, \chi u\} \right\|_{L^2}\leqslant C  e^{-\vartheta_0 t\min I}
\| \Phi \|_{\mathscr{C}^N} \left( \left\| u\right\|_{L^2} +  \left\|\mathbf{A}_{\mathfrak S_0,0}\chi u\right\|_{L^2}\right), \quad t \geqslant 0.
\end{equation}
Here $\mathbf P: \mathscr C^\infty(M^\times) \to  \mathscr C^\infty(M^\times)$ is given by
\begin{equation}\label{eq:spectral-projector}
\mathbf{P}v(x_1, r\xi_1)=v(x_1, r\xi_1) - \int_{M_1} v(y_1,r\eta_1)\dd\mrm {\mrm L}_1(y_1,\eta_1), \quad v \in \mathscr C^\infty(M^\times),
\end{equation}
where $\mrm {\mrm L}_1$ is the normalized Liouville measure on the unit cotangent bundle $M_1=\mrm S^*\Sigma$.
\end{prop}

This proposition will be relevant to ensure the convergence of the integral remainder term in the Duhamel formula for the Vlasov equation~\eqref{eq:vlasov}. More precisely, the estimate~\eqref{eq:bilinearprop} is intended to be applied with $\Phi=\Phi(u(t))$ and $u=u(t)$ where $u(t)$ is the solution to the nonlinear Vlasov equation~\eqref{eq:vlasov}. 
  The right-hand side of~\eqref{eq:bilinearprop} achieves two important goals:
\begin{itemize}
\item exponential decay in time is obtained;
\item the sliced anisotropic norm $\left\| \mathbf{B}_{\mathfrak S_0,-2}^{(m_1,N_1)}(\cdots)\right\|_{L^2}$ in the left-hand side is controlled by the global anisotropic norm $\left\|\mathbf{A}_{\mathfrak S_0,0}\chi u \right\|_{L^2}$ on the right-hand side.
\end{itemize}

The exponential decay will help the convergence of the integral remainder \eqref{eq:integralremainder} provided that the  norms in the right-hand side of~\eqref{eq:bilinearprop} do not grow too fast with $t$. We can first observe that all $L^p$ norms of $u(t)$ are preserved under the evolution by~\eqref{eq:vlasov} so that $\left\| u(t)\right\|_{L^2}$ will not contribute to any growth in time. Similarly, the fact that $\Phi(u(t))$ is expressed in terms of a smooth interaction kernel implies that the term $\| \Phi(u(t)) \|_{\mathscr{C}^N}$ remains bounded with time. Hence, the only difficulty is to find an a priori upper bound on the global anisotropic Sobolev norm $\|\mathbf{A}_{\mathfrak S_0,0}\chi u(t)\|_{L^2}$. This is exactly where the choice of the global escape function will turn out to be crucial in view of getting a polynomial upper bound. In fact, as we shall use it in the proof of Lemma~\ref{l:polynomialbound} below, the principal symbol of $\mathbf{A}_{\mathfrak S_0,0}$ belongs to a nice enough class of symbols on $\R\times M_1$, which is the key to performing an energy estimate with the help of microlocal methods.

The rest of this section is devoted to the proof of Proposition~\ref{p:bilinear} and, along the way, we shall also obtain intermediate corollaries that will be helpful for the upcoming nonlinear analysis. The proof is organized as follows. First, we recall the exponential decay result for the geodesic flow restricted to $M_1$. Then, we lift these results to $\R\times M_1$ and finally combine them with the properties of our escape functions (together with microlocal arguments) to get Proposition~\ref{p:bilinear}.

\subsection{Exponential decay on $M_1$}  
\label{sec:4.3}

Unfortunately, the decay estimate given by Theorem~\ref{t:dolgopyatliverani} does not seem sufficient for our needs and we will rather rely on a microlocal refinement of this result following from \cite{nonnenmacher2015decay}:

\begin{theo}[Nonnenmacher-Zworski]\label{t:nonnenmacherzworski} There is $\widetilde{\vartheta}_0>0$ such that, for each $\mathfrak S_1\geqslant 3$, one can find a constant\footnote{As we shall see in the proof, this is in fact the same constant as in~\eqref{eq:FSresolventbound}.} $C_{1}>0$ so that the following property holds.

 Let $N>10\mathfrak S_1$ and $\sigma_1\in[-2,2]$. Then, there exists a constant $C>0$ such that, for every $t\geqslant 0$,
 
 $$ 
 \left\|e^{{\oM}\left(G_{\mathfrak S_1,\sigma_1,1}\right)}\left(X_1+C_{1}\right)^{-N}\varphi_{-t}^*\mathrm{P}_1e^{-{\oM}\left(G_{\mathfrak S_1,\sigma_1,1}\right)}\right\|_{L^2(M_1)\rightarrow L^2(M_1)}\leqslant Ce^{-t\widetilde{\vartheta}_0},
$$
where $\mathrm{P}_1v=v-\int_{M_1}v\,\dd \mrm {\mrm L}_1.$
 \end{theo}
Recall that the original decay of correlations result in~\eqref{eq:exp-mixing-classical} is rather written for smooth test functions on $M_1$, i.e. for every $u,v\in\mathscr{C}^{k_0}(M_1)$,
\begin{equation}\label{eq:liverani}
\left|\int_{M_1} (u\circ\varphi_{-t}) v \,\dd\mrm L-\int_{M_1} u\dd \mrm {\mrm L}_1\int_{M_1}v\dd \mrm {\mrm L}_1\right|\\
\leqslant C e^{-t\vartheta_0}\|u\|_{\mathscr{C}^{k_0}}\|v\|_{\mathscr{C}^{k_0}},
\end{equation}
which can directly be recovered from the refined estimate of Theorem~\ref{t:nonnenmacherzworski}.

\begin{rema}
For simplicity, we choose to work in Section~\ref{sec:mainproof} with a fixed regularity parameter $\mathfrak S_1$ independent of the geometry, more specifically $\mathfrak S_1= 3$ to ensure that $(1-2\delta)\mathfrak S_1>2\geqslant |\sigma_1|$ (provided $\delta>0$ is small enough). Modulo some slightly more tedious work, this could probably be optimized to pick any $\mathfrak S_1>1$ but we do not pursue this issue here.
\end{rema}

\begin{proof} The results in~\cite{nonnenmacher2015decay} are not exactly stated under such a precise form. Yet, this result is implicitely proved in $\S9$ of this reference as a consequence of their main result and, for the sake of completeness, we will briefly recall the analytical context of this result and the argument to go from resolvent estimates to exponential decay of the flow. 

The main result (Theorem $2$) in~\cite{nonnenmacher2015decay} is a result on semiclassical selfadjoint pseudo-differential operators $P(h)$ of order $m>0$ such that the following holds
\begin{enumerate}
 \item there exists a complex absorbing potential, i.e. a selfadjoint (semiclassical) pseudo-differential $W(h)$ of order $0\leqslant k\leqslant m $ whose principal symbol $W$ is nonnegative and satisfies certain growth assumptions (Eq.~(1.9) and (1.10) in that reference);
 \item the level set $p^{-1}(E)$ (where $p$ is the principal symbol of $P(h)$) is a smooth submanifold with $\dd p|_{p^{-1}(E)}\neq 0$;
 \item the trapped set at energy $p^{-1}((E-\delta,E+\delta))$ is a normally hyperbolic and symplectic manifold. Recall that the trapped set consists in all the Hamiltonian trajectories for the flow induced by $p$ that never enters the damping region in forward and backward times.
\end{enumerate}
Then, Theorem~$2$ in~\cite{nonnenmacher2015decay} is a resolvent estimate for $(P(h)-iW(h)-z)^{-1}$ when $z$ lies in a small neighbodhood of size $\lambda_0 h$ of $E$ (with $\lambda_0>0$ that can be expressed in terms of purely dynamical quantities). We also refer to~\cite{GerardSjostrand1988, WunschZworski2011, Dyatlov2015, Dyatlov2016} for related results when the incoming and outgoing trapped sets are regular enough. 

It is important to note that the role of $W(h)$ is purely auxiliary in this result and one has to find some concrete operators $\widetilde{P}(h)$ where such a $W(h)$ exists and to deduce some applications on the spectral gap for $\widetilde{P}(h)$. Among the several examples given in this reference, one has the case where $P(h)=-ihX_1$ where $X_1$ is a contact Anosov flow on a compact manifold (a geodesic flow is the main example of a contact flow). In fact, as explained in Example $2$ and $\S 9.1$ of this reference, one can design an operator $W$ with these requirements when $P(h)=-ihX_1$. This follows from the fact that the trapped set\footnote{Here trapped set refers to the Hamiltonian trajectories that stay in a bounded region of $\mrm{T}^*M_1$} in the energy window $(E-\delta,E+\delta)$ for this operator is given by $\{(z_1;\zeta_0,0): |\zeta_0-E|<\delta\}\subset E_0^*$. Recall that $p(z_1;\zeta_1)=\zeta_1(X_1(z_1))$ in this setting. The contact assumption is thus here to ensure that our trapped set $E_0^*$ has a symplectic structure and the Anosov assumption allows to verify the normally hyperbolic property. Hence, Theorem~2 from this reference can be applied and it yields a resolvent estimate for the corresponding modified operator, i.e. $(-ihX_1-iW(h)-z)^{-1}$. Now, the strategy in~\cite[\S 9]{nonnenmacher2015decay} is to convert this key resolvent estimate into a resolvent estimate for a more relevant operator regarding the dynamical properties (see Theorem $4$ in this reference).

The strategy is as follows. Define 
$$
\widetilde{P}(h) = e^{{\oM}_h\left(G_{\mathfrak S_1,\sigma_1,1}\right)}(-ihX_1)e^{-{\ohM}\left(G_{\mathfrak S_1,\sigma_1,1}\right)},
$$
where ${\ohM}$ is a semiclassical quantization procedure~\cite{zworski,dyatlov2017mathematical}. By construction, $\widetilde{P}(h)$ agrees microlocally with $P(h)-iW(h)$ in a compact region of $\mrm{T}^*M_1$. According to Remark~\ref{r:faure-sjostrand}, the geometric properties of our sliced escape function\footnote{This Remark was stated in a nonsemiclassical set-up but the proof in~\cite{DyatlovZworski2016} allows $0<h\leqslant 1.$} ensures that the resolvent $(\widetilde{P}(h)-z)^{-1}$ extends meromorphically to a strip 
$$
\{\text{Re}(z)\geqslant -(c_0((1-2\delta)\mathfrak S_1+|\sigma_1|)h\}.
$$
Moreover, the principal symbol of $\widetilde{P}(h)$ is given by $\zeta_1(X_1)+ih\{\zeta_1(X_1),G_{\mathfrak{S}_1,\sigma_1,1}\}$ modulo terms lying in $h^2S^{-1+}(\mathrm{T}^*M_1)$ (see Remark~\ref{r:explicitsymbol}). Then, thanks to Proposition~\ref{prop:comparisonresult} and to~\eqref{e:lyapunov-norm}, we can verify that the imaginary part is not greater than $-ch<0$ (for some positive $c>0$) near $E_u^*$ and $E_s^*$. In particular, using the Anosov property, only points lying in $E_0^*$ does not end up in the damping region of our nonselfadjoint operator (meaning in either backward or forward times).

\begin{rema} In~\cite{nonnenmacher2015decay}, the authors made the stronger assumptions of~\cite{faure2011upper} on the escape function used to define $\widetilde{P}(h)$ which ensures that the imaginary part is $\leqslant -ch$ outside a small neighborhood of $E_0^*$. Yet, the choice of this escape function was made to ensure the Fredholm property and the meromorphic continuation of the resolvent beyond the real axis~\cite[Lemma~9.2]{nonnenmacher2015decay} and they could as well have used the larger class of weights functions as introduced in~\cite{DyatlovZworski2016} for which these properties remain true (see Remark~\ref{r:faure-sjostrand}). See for instance~\cite[\S 6]{DyatlovZworski2015} for resolvent estimates with these weaker weights. As we have just explained it, the only difference is that the damping only occurs near $E_u^*$ and $E_s^*$ (but all points outside $E_0^*$ end up in these neighborhoods in either backward or forward times). 
\end{rema}

Hence, one can define a (selfadjoint) pseudo-differential operator $W_\infty(h)$ which is compactly supported (near $E_0^*$), whose wavefront set does not intersect the wavefront of $W(h)$ and such that the wavefront set of $\text{Id}-W_\infty(h)$ does not intersect the trapped set of the Hamiltonian flow near the energy $E$. See~\cite[\S 9.3]{nonnenmacher2015decay} for more precise statements. The key point is that $P_\infty(h)=\widetilde{P}(h)-iW_\infty(h)$ has now an empty undamped set and one can apply (nontrapping) resolvent estimate to this operator near $z=E$. See~\cite[Lemma 9.3]{nonnenmacher2015decay} and the references therein. Then, with these two resolvent estimates at hand and still following~\S9 from this reference, one can make use of the gluing method due to Datchev and Vasy~\cite{DatchevVasy2012} in view of getting a resolvent estimate for $(\widetilde{P}(h)-z)^{-1}$ near $z=E$ -- see~\cite[Th.~4]{nonnenmacher2015decay} for a precise statement of the semiclassical resolvent estimate.

Once we have these semiclassical resolvent estimates, we can turn to the decay estimate of Theorem~\ref{t:nonnenmacherzworski} for the induced flow. Recall that~\cite[Th.~4]{nonnenmacher2015decay} (when translated in a nonsemiclassical form, see Equation~(9.13) in this reference) states the existence of a parameter $\vartheta_1>0$ such that, for every $\mathfrak S_1\geqslant 3$ and for every $\sigma_1\in[-2,2]$, one can find a constant $C_{\mathfrak S_1,\sigma_1}>0$ with the following property
 \begin{equation}\label{eq:NZresolvent}
 \begin{aligned}
 &\left\|e^{{\oM}\left(G_{\mathfrak S_1,\sigma_1,1}\right)}\left(X_1+z\right)^{-1}e^{-{\oM}\left(G_{\mathfrak S_1,\sigma_1,1}\right)}\right\|_{L^2(M_1)\rightarrow L^2(M_1)}\\
 &\hspace{8cm}\leqslant C_{\mathfrak S_1,\sigma_1}(1+|\operatorname{Im}(z)|)^{2(\mathfrak S_1+|\sigma_1|)+1},
 \end{aligned}
 \end{equation}
whenever $z\in\mathbb{C}$ verifies $\operatorname{Re}(z)\in[-\vartheta_1,\vartheta_1]$ and $|\text{Im}(z)|\geqslant \vartheta_1$. 

Such a resolvent estimate can be transferred into our expected exponentially decaying upper bound. Let $v, w \in \mathscr C^\infty(M)$~; we want to estimate
\begin{equation}\label{eq:wantestnz}
\int_{M_1} (\mathrm{P}_1(v) \circ\varphi_{-\tau}) w\, \dd {\mrm L}_1 = \int_{M_1} (\mathrm{P}_1(v)\circ\varphi_{-\tau}) \mathrm{P}_1(w)\, \dd {\mrm L}_1.
\end{equation}
To do that, we proceed as \cite[\S9, Proof of Cor. 5]{nonnenmacher2015decay}. More precisely, we use Stone's formula to rewrite \eqref{eq:wantestnz} as 
\begin{equation}
\begin{aligned}
\frac{1}{2\pi}\lim_{\varepsilon\to 0}\int_{\R}e^{i\tau \lambda}\left\langle \left((X_1-i\lambda +\varepsilon)^{-1}-(X_1-i\lambda-\varepsilon)^{-1}\right) \mathrm{P}_1(v),\overline{\mathrm{P}_1(w)}\right\rangle_{L^2(M_1)}\dd\lambda.
\end{aligned}
\end{equation}
From this, we infer that, for every $N\geqslant 0$, \eqref{eq:wantestnz} writes
$$
\frac{1}{2\pi}\lim_{\varepsilon\to 0}\int_{\R}\frac{\dd \lambda \, e^{i\tau \lambda}}{(C_1+iz)^{N}} 
 \sum_{\pm} \pm  \left\langle  (X_1-iz \pm \varepsilon)^{-1} (X_1+C_{1})^{N}\mathrm{P}_1(v),\overline{\mathrm{P}_1(w)}\right\rangle_{L^2(M_1)},
$$
where $C_{1}$ is chosen large enough to ensure that~\eqref{eq:FSresolventbound} holds true. Then, we use the resolvent bound~\eqref{eq:NZresolvent} (and its analogue in backward time) and pick $N\geqslant 10\mathfrak S_1$ to ensure the convergence of the integral in $\lambda$. It allows us make a deformation of contour from $\lambda =\R$ to $\lambda=\R+i\vartheta_0$ where we choose $\vartheta_0$ in such a way that there is no Ruelle-Pollicott resonances (except for $0$) in the strip delimited by these two lines. As a consequence, one gets that \eqref{eq:wantestnz} is equal to
\begin{multline*}
\frac{1}{2\pi}\lim_{\varepsilon\to 0}\int_{\R+ i\vartheta_0}\frac{\dd \lambda \, e^{i\tau \lambda}}{(C_1+iz)^{N}} 
 \sum_{\pm} \pm  \left\langle  (X_1-iz \pm \varepsilon)^{-1} (X_1+C_{1})^{N}\mathrm{P}_1(v),\overline{\mathrm{P}_1(w)}\right\rangle_{L^2(M_1)}.
\end{multline*}
Finally by~\eqref{eq:NZresolvent} we obtain that for every $\tau\geqslant 0$,
\begin{multline*}
\left|\int_{M_1} (\mathrm{P}_1(v)\circ\varphi_{-\tau}) w\, \dd{\mrm L}_1\right|
\leqslant C e^{-\tau\vartheta_0}\left\|e^{{\oM}\left(G_{\mathfrak S_1,\sigma_1,1}\right)}(X_1+C_1)^{N}\mathrm{P}_1(v)\right\|_{L^2}\\
\times
\|e^{-{\oM}\left(G_{\mathfrak S_1,\sigma_1,1}\right)}\mathrm{P}_1(w)\|_{L^2},
\end{multline*}
which concludes the proof of the Theorem thanks to~\eqref{eq:FSresolventbound}. 
\end{proof}

\subsection{Lifting Theorem~\ref{t:nonnenmacherzworski} to $\R\times M_1$}
\label{sec:4.4}

Building on Theorem~\ref{t:nonnenmacherzworski}, we are able to derive the following lemma which lifts this result to $\R\times M_1$: 
\begin{lemm}\label{l:NZradial} There exists $\vartheta_0>0$ depending only on the choice of $(\nu,g_1)$ defined in Proposition~\ref{prop:comparisonresult} such that the following holds.

Let $\mathfrak{S}_1\geqslant 3$, $\sigma_1\in[-2,2]$, $N\in\mathbb{Z}_+^*$ and $m_1,m_2\in\mathbb{Z}$ be such that
$$
m_1m_2\leqslant 0,\quad m_1 + m_2 > 1,\quad \text{and}\quad N\geqslant 10\mathfrak S_1-\min\{m_1,m_2\}+2.
$$
Let also $I$ be an open and bounded interval such that $\overline{I}\subset\R_+^*$ and $\chi\in\mathscr{C}^\infty_c(I)$. Then, there exists a constant $C>0$ such that, for all $t\geqslant 0$,
$$
\left\|\mathbf{B}_{\mathfrak S_1,\sigma_1}^{(m_1,N)}\chi\varphi_{-t}^*\mathbf{P}\mathbf{B}_{-\mathfrak S_1,-\sigma_1}^{(m_2)}\right\|_{L^2\rightarrow L^2}\leqslant C e^{-t\vartheta_0\min I}.
$$
\end{lemm}

\begin{rema} By taking the flow in positive time, one obtains a similar statement by remplacing $\mathfrak S_1 $ by $-\mathfrak S_1 <0$.
\end{rema}

\begin{proof} Let $u$ be an element in $\mathscr{S}(\R\times M_1)$. From the definition of our sliced anisotropic norms~\eqref{eq:operatorvaluedanisotropic} and~\eqref{eq:operatorvaluedanisotropicresolvent}, we want to compute the norm of 
\begin{equation}\label{eq:wantcompute}
\left\|\mathbf{B}_{\mathfrak S_1,\sigma_1}^{(m_1)}\chi\varphi_t^*\mathbf{P}(X_1+C_1)^{-N}\mathbf{B}_{-\mathfrak S_1,-\sigma_1}^{(m_2)}u\right\|_{L^2},
\end{equation}
where $\chi\in\mathscr{C}^\infty_c(\R_+^*)$ is compactly supported inside the interval $I$. To this end, we use Plancherel's formula to write \eqref{eq:wantcompute} as 
$$
\begin{aligned}
\frac{1}{(2\pi)^4}\int_{\R}\frac{1}{(1+\rho^2)^{m_1}} \times\Biggl\|\int_{I\times\R^2} & \frac{e^{is(\widetilde{\rho}-\rho)}e^{-i\widetilde{r}\widetilde{\rho}}}{(1+\widetilde{\rho}^2)^{\frac{m_2}{2}}} e^{{\oM}\left(G_{\mathfrak S_1,\sigma_1,1}\right)}{e^{tsX_1}\mathrm{P}_1\chi(s)} \\
& \quad \times {(X_1+C_{1})^{-N}}e^{-{\oM}\left(G_{\mathfrak S_1,\sigma_1,1}\right)}u(\widetilde{r})\,\dd s\,\dd\widetilde{\rho}\,\dd\widetilde{r}\Biggr\|_{L^2(M_1)}^2\dd\rho,
\end{aligned}
$$
or equivalently
$$
\begin{aligned}
\frac{1}{(2\pi)^4}\int_{\R}\frac{1}{(1+\rho^2)^{m_1}} \times\Biggl\|\int_{I\times\R} & \frac{e^{is(\widetilde{\rho}-\rho)}}{(1+\widetilde{\rho}^2)^{\frac{m_2}{2}}} e^{{\oM}\left(G_{\mathfrak S_1,\sigma_1,1}\right)}{e^{tsX_1}\mathrm{P}_1\chi(s)} \\
& \quad \times {(X_1+C_{1})^{-N}}e^{-{\oM}\left(G_{\mathfrak S_1,\sigma_1,1}\right)}\widehat{u}(\widetilde{\rho})\,\dd s\,\dd\widetilde{\rho}\Biggr\|_{L^2(M_1)}^2\dd\rho,
\end{aligned}
$$
where $\widehat{u}$ is the Fourier transform of $u$ with respect to the radial variable. Let us introduce the operator
$$
P_{\rho,\widetilde{\rho}}=\frac{1+i(\widetilde{\rho}-\rho)\partial_s}{1+(\rho-\widetilde{\rho})^2}.
$$
Then it holds
$$
P_{\rho, \widetilde \rho}\left(e^{is(\widetilde \rho - \rho)}\right) = e^{is(\widetilde \rho - \rho)},
$$
and integrating by parts yields, for $N_2\geqslant 0$,
$$
\begin{aligned}
 &\left\|\mathbf{B}_{N_1,\sigma_1}^{(m_1)}\chi\varphi_t^*\mathbf{P}(X_1+C_{1})^{-N}\mathbf{B}_{-\mathfrak S_1,-\sigma_1}^{(m_2)}u\right\|_{L^2}^2 \\
&  \qquad = \frac{1}{(2\pi)^4}\int_{\R}\frac{1}{(1+\rho^2)^{m_1}}
\left\|\int_{I\times\R}e^{-is\rho}\langle\widetilde{\rho}\rangle^{-m_2} \langle\widetilde{\rho}-\rho\rangle^{-N_2}W_{s,\rho,\widetilde{\rho}}\left[\widehat{u}(\widetilde{\rho})\right]\,\dd s\,\dd\widetilde{\rho}\,\right\|_{L^2(M_1)}^2\dd\rho,
 \end{aligned}
$$
where the operator $W_{s,\rho,\widetilde{\rho}}$ is given by
$$
W_{s,\rho,\widetilde{\rho}}=e^{{\oM}\left(G_{\mathfrak S_1,\sigma_1,1}\right)}\left(\frac{e^{is\widetilde{\rho}}(P_{\rho,\widetilde{\rho}})^{N_2}(e^{tsX_1}\mathrm{P}_1\chi(s))(X_1+C_{1})^{-N}}{\langle\widetilde{\rho}-\rho\rangle^{-N_2}}\right)e^{-{\oM}\left(G_{\mathfrak S_1,\sigma_1,1}\right)}.
$$
By Cauchy-Schwarz inequality, one has
$$
\left\|\mathbf{B}_{N_1,\sigma_1}^{(m_1)}\chi\varphi_t^*\mathbf{P}(X_1+C_{1})^{-N}\mathbf{B}_{-\mathfrak S_1,-\sigma_1}^{(m_2)}u\right\|_{L^2}^2 \leqslant I_{m_1,m_2,N_2}\sup_{s\in I,\,\rho\in\R}
\int_{\R}\left\|W_{s,\rho,\widetilde{\rho}}\left[\widehat{u}(\widetilde{\rho})\right]\right\|_{L^2(M_1)}^2\,\dd\widetilde{\rho},
$$
where
$$
 I_{m_1,m_2,N_2}=\frac{|I|^2}{(2\pi)^4}\int_{\R^2 }\frac{1}{(1+\rho^2)^{m_1}}\frac{1}{(1+\widetilde{\rho}^2)^{m_2}} \frac{1}{(1+(\rho-\widetilde{\rho})^2)^{N_2 }}\dd\rho\,\dd\widetilde{\rho}
$$
is finite provided $N_2=-\min\{m_1,m_2\}+1$ and $m_1+ m_2 > 1$.
Hence, recalling that $\|u\|_{L^2}=\int_\R\|\widehat{u}(\widetilde{\rho})\|_{L^2(M_1)}^2\dd\widetilde{\rho}$, the proof boils down to estimating the operator norm of $W_{s,\rho,\widetilde{\rho}}$, which can be bounded by a sum of terms of the form
 $$
 t^{\ell_1}\left\|e^{{\oM}\left(G_{\mathfrak S_1,\sigma_1,1}\right)}\chi^{(\ell_2)}(s)e^{tsX_1}\mathrm{P}_1X_1^{\ell_1}(X_1+C_{1})^{-N}e^{-{\oM}\left(G_{\mathfrak S_1,\sigma_1,1}\right)}\right\|_{L^2(M_1)\rightarrow L^2(M_1)}
 $$
 with $\ell_1+\ell_2\leqslant N_2$ and $s\in I$. According to Theorem~\ref{t:nonnenmacherzworski}, the latter are bounded by
 $
C  t^{\ell_1 } e^{-t \vartheta_0 \min I },
 $ 
 provided that $N-N_2>10\mathfrak S_1$. The result follows.
\end{proof}

As a direct consequence of Lemma~\ref{l:NZradial}, we deduce the following decay estimate for smooth test functions which will also be used in our nonlinear analysis:
\begin{coro}\label{c:smoothNZ} There exists $\vartheta_0>0$ depending only on the choice of $(\nu,g_1)$ defined in Proposition~\ref{prop:comparisonresult} such that the following holds.

Let $\mathfrak S_1\geqslant 3$, $\sigma_1\in[-2,2]$, $m_1\in\Z$ and $I$ be an open and bounded interval such that $\overline{I}\subset\R_+^*$. Then, there exist a constant $C>0$ and an integer $N \in \mathbb{N}$ such that, for all $t\geqslant 0$ and for every $u\in\mathscr{C}^N_c(I\times M_1)$

$$
\left\|\mathbf{B}_{\mathfrak S_1,\sigma_1}^{(m_1)}\varphi_{-t}^*\mathbf{P}u\right\|_{L^2}\leqslant C e^{-t\vartheta_0\min I}\|u\|_{\mathscr{C}^N}.
$$ 
\end{coro}
Again, we can get a similar statement for the flow in positive time by replacing $\mathfrak S_1$ by $-\mathfrak S_1$. We also record that we implicitely recover Theorem~\ref{t:dolgopyatliverani} from the introduction (away from the zero section):
\begin{coro}\label{c:smoothcorrelation} There exist $\vartheta_0,k_0>0$ depending only on Riemannian manifold $(\Sigma,\mathrm{g})$ such that, for every open and bounded interval $I$ such that $\overline{I}\subset\R_+^*$, one can find a constant $C>0$ such that, for all $t\geqslant 0$ and for every $u,v\in\mathscr{C}^{k_0}_c(I\times M_1)$,
\begin{multline*}
\left|\int_{M^\times} (u\circ\varphi_{-t}) v\, \dd {\mrm L}-\int_{M^\times}\left(\int_{M_1} u(r,z_1')\dd {\mrm L}_1(z_1')\right) v(r,z_1)\dd {\mrm L}(r,z_1)\right|\\
\leqslant C e^{-t\vartheta_0\min I}\|u\|_{\mathscr{C}^{k_0}}\|v\|_{\mathscr{C}^{k_0}}.
\end{multline*}
\end{coro}


\subsection{Comparison between the different norms}
\label{sec:4.5}

We now discuss how the different anisotropic norms we have introduced compare to each other in various senses that are relevant to the proof of Proposition~\ref{p:bilinear}. We begin with a comparison of sliced anisotropic Sobolev norms.
\begin{lemm}\label{l:comparisonnormdr} 
Let $\mathfrak S_1,\sigma_1,\sigma_2\in\mathbb{R}$ and $m_1,m_2 \in\Z$ verifying
$$
\mathfrak{S}_1\geqslant 1,\quad m_1m_2\leqslant 0,\quad m_1 + m_2 >2,\quad\text{and}\quad\sigma_1+\sigma_2\leqslant 0.
$$
Let also $I$ be an open and bounded interval such that $\overline{I}\subset\R_+^*$ and $\chi\in\mathscr{C}^\infty_c(I)$. Then, there exist a constant $C>0$ and an integer $N\geqslant 1$ such that

$$
\left\|\mathbf{B}_{\mathfrak S_1,\sigma_1}^{(m_1)} (a\partial_r) \chi\mathbf{B}_{-\mathfrak S_1,\sigma_2}^{(m_2)}\right\|_{L^2\rightarrow L^2}\leqslant C\left\|a\right\|_{\mathscr{C}^N}, \qquad a \in \mathscr C^\infty(M_1).
 $$
\end{lemm}

\begin{proof} We proceed as in the proof of Lemma~\ref{l:NZradial} and we write that
\begin{multline*}
 \left\|\mathbf{B}_{\mathfrak S_1,\sigma_1}^{(m_1)} (a(z_1)\partial_r) \chi\mathbf{B}_{-\mathfrak S_1,\sigma_2}^{(m_2)}u\right\|_{L^2}^2=\frac{1}{(2\pi)^4}\int_{\R}\frac{\rho^2}{(1+\rho^2)^{m_1}}\\
 \left\|\int_{I\times \R^2}\frac{ e^{is(\widetilde{\rho}-\rho)}\chi(s)}{(1+\widetilde{\rho}^2)^{\frac{m_2}{2}}}e^{{\oM}(G_{\mathfrak S_1,\sigma_1})}ae^{{\oM}(G_{-\mathfrak S_1,\sigma_2})} \widehat{u}(\widetilde{\rho})\,\dd s\,\dd\widetilde{\rho}\right\|^2_{L^2(M_1)}\dd \rho.
\end{multline*}
 As in the proof of Lemma~\ref{l:NZradial}, we can integrate by parts with respect to the variable $s$ using the operator $P_{\rho,\widetilde{\rho}}$ in view of getting the missing integrability in the $\rho$ (or $\widetilde{\rho}$) variable. Hence, under the assumption that $m_1 + m_2 >2$, one can find some uniform constant $C>0$ such that
 \begin{multline*}
 \left\|\mathbf{B}_{\mathfrak S_1,\sigma_1}^{(m_1)} (a\partial_r) \chi\mathbf{B}_{-\mathfrak S_1,\sigma_2}^{(m_2)}u\right\|_{L^2}^2\leqslant C\left\|e^{{\oM}(G_{\mathfrak S_1,\sigma_1,1})}ae^{{\oM}(G_{-\mathfrak S_1,\sigma_2,1})}\right\|_{L^2(M_1)\rightarrow L^2(M_1)}^2\\
 \times \int_{\R}\left\|
 \widehat{u}(\widetilde{\rho})\right\|^2_{L^2(M_1)}\dd\widetilde{\rho}.
 \end{multline*}
 Hence, applying Plancherel's formula, everything boils down to an estimate on
 $$
 \left\|e^{{\oM}(G_{\mathfrak S_1,\sigma_1,1})}ae^{{\oM}(G_{-\mathfrak S_1,\sigma_2,1})}\right\|_{L^2(M_1)\rightarrow L^2(M_1)}.
 $$
 Using the composition rule~\eqref{eq:compositionmfd2} in $S^{2\mathfrak S_1}(\mrm{T}^*M_1)$ together with the approximate expressions~\eqref{eq:symbolexpleft} and~\eqref{eq:symbolexpright} for the exponential\footnote{Note that this may require to rescale the symbol of $\mathbf{B}$ by a parameter $h>0$ that depends only on $N_1$ and on the choice of the escape function. As our arguments only care about what happens outside a compact subset of $\mrm{T}^*M_1$, this rescaling of the escape function has no importance.} of a pseudo-differential operator, one finds that this quantity is uniformly bounded as soon as $\sigma_1+\sigma_2\leqslant 0$. Moreover, the upper bound (given by the composition rule and the Calder\'on-Vaillancourt Theorem~\eqref{eq:calderonmfd}) is linear in terms of $\|a\|_{\mathscr{C}^N}$ for some large enough $N\geqslant1$. 
\end{proof}

Arguing similarly, one obtains
\begin{lemm}\label{l:comparisonnormtgt}
Let $\mathfrak S_1,\sigma_1,\sigma_2\in\mathbb{R}$ and $m_1,m_2 \in\Z$ verifying
$$
\mathfrak{S}_1\geqslant 1,\quad m_1m_2\leqslant 0,\quad m_1 + m_2 >1,\quad\text{and}\quad\sigma_1+\sigma_2+1\leqslant 0.
$$
Let also $I$ be an open and bounded interval such that $\overline{I}\subset\R_+^*$ and $\chi\in\mathscr{C}^\infty_c(I)$. Then, there exist a constant $C>0$ and an integer $N\geqslant 1$ such that, for every smooth vector field $Y$ on $M_1$, one has
 $$
 \left\|\mathbf{B}_{\mathfrak S_1,\sigma_1}^{(m_1)} Y\chi\mathbf{B}_{-\mathfrak S_1,\sigma_2}^{(m_2)}\right\|_{L^2\rightarrow L^2}\leqslant C\left\|Y\right\|_{\mathscr{C}^N}.
 $$ 
\end{lemm}
From the two lemmas above and Corollary~\ref{c:smoothNZ}, we deduce the following result on the semigroup $(\varphi_t)_{t\geqslant0}$:
\begin{coro}\label{c:smoothNZ2} There exists $\vartheta_0>0$ depending only on the choice of $(\nu,g_1)$ defined in Proposition~\ref{prop:comparisonresult} such that the following holds.

Let $\mathfrak S_1\geqslant 3$, $\ell\geqslant 0$, $\sigma_1\in[-1,1]$ and $m_1\in\Z$. Let also $I$ be an open and bounded interval such that $\overline{I}\subset\R_+^*$. Then, there exist  $C>0$, $N \in \mathbb{N}$ such that, for all $t\geqslant 0$, for every $u\in\mathscr{C}^N_c(I\times M_1)$, for every smooth function $a$ in $M_1$ and for every smooth vector field $Y$ on $M_1$,

\begin{equation}\label{eq:lll}
\left\|\mathbf{B}_{\mathfrak S_1,\sigma_1}^{(m_1)}X_1^\ell\left(a\partial_r+Y\right)\varphi_{-t}^*\mathbf{P}u\right\|_{L^2}\leqslant C e^{-t\vartheta_0\min I}\|u\|_{\mathscr{C}^N}\left(\|a\|_{\mathscr{C}^N}+\|Y\|_{\mathscr{C}^N}\right).
\end{equation}
\end{coro}
Again, the same result holds true for positive time provided that we replace $\mathfrak S_1$ by $-\mathfrak S_1$ in the definition of $\mathbf{B}_{\mathfrak S_1,\sigma_1}^{(m_1)}$. 

\begin{proof}
First note that one has
 $$
 X_1^\ell(a\partial_r+Y)=X_1[X_1^{\ell-1},a\partial_r+Y]+X_1(a\partial_r+Y)X_1^{\ell-1}.
 $$
 Hence, by induction, we get
 $$
 X_1^\ell(a\partial_r+Y)=\sum_{k=0}^\ell (a_{\ell,k}\partial_r+Y_{\ell,k})X_1^k,
 $$
 where $Y_{\ell,k}$ (resp. $a_{\ell,k}$) is a vector field (resp. a function) on $M_1$ (whose semi-norms depend on a finite number of derivatives of $Y$). In particular, to get the south result, it suffices to bound the quantity
 $$
\max_{k = 0, \dots, \ell} \left\|\mathbf{B}_{N_1,\sigma_1}^{(m_1)}\left(a_{\ell,k}\partial_r+Y_{\ell,k}\right)\varphi_{-t}^*\mathbf{P}X_1^ku\right\|_{L^2}.
 $$
Such a bound is obtained as a direct consequence of Corollary~\ref{c:smoothNZ} and Lemmas~\ref{l:comparisonnormdr} and~\ref{l:comparisonnormtgt}.
\end{proof}

We end this subsection by establishing the connection between the sliced and the global anisotropic Sobolev norms. 
To this purpose, we recall that $\mathbf{A}_{\mathfrak S_0,\sigma_0}$ admits an almost inverse thanks to the composition formula. More precisely, according to Lemma~\ref{l:pseudoinverse1}, Remark~\ref{r:comparisonsobolevnorms} and given any $N'\geqslant 1$, one can find $b_{N'}\in \overline{S}^0(\mrm{T}^*(\R\times M_1))$ which is independent of $r$ and equal to $1$ modulo $\overline{S}^{-1+}(\mrm{T}^*(\R\times M_1))$ such that, for every $\chi\in\mathscr{C}^\infty_c(\R)$, it holds
\begin{equation}\label{eq:pseudoinverseA}
\left\|\left(\mathbf{Op}\left(e^{-F_{\mathfrak S_0,\sigma_0}}b_{N'}\right)\mathbf{A}_{\mathfrak S_0,\sigma_0}-\text{Id}\right)\chi\right\|_{H^{-N'}(\R\times M_1)\rightarrow H^{N'}(\R\times M_1)}<\infty.
\end{equation}
With this convention at hand, one has the following result.
\begin{lemm}\label{l:comparisonnormNZ-ours}
For every $\mathfrak S_0\geqslant 1$, $\sigma_1 < \sigma_0 \leqslant 0$,  $m_1 \geqslant 5(\mathfrak S_0+|\sigma_1|)$ and $N'\geqslant 0$, one has
 $$
 \left\|\mathbf{B}_{\mathfrak S_0,\sigma_1}^{(m_1)}\mathbf{Op}\left(e^{-F_{\mathfrak S_0,\sigma_0}}b_{N'}\right)\right\|_{L^2\rightarrow L^2} < \infty.
 $$

\end{lemm}

\begin{proof}
Arguing as in the proof of Lemma~\ref{l:NZradial}, owing to Plancherel's identity, it is sufficient to study, for all $\rho \in \R$, the quantity
\begin{equation}
\left\|e^{{\oM}(G_{\mathfrak S_0,\sigma_1,1})}{\oM}(e^{-F_{\mathfrak S_0,\sigma_0}(\rho)}b_{N'}(\rho))\right\|_{L^2(M_1)\rightarrow L^2(M_1)}.
\end{equation}
The latter is bounded from above by
 \begin{multline*}
 \left\|e^{{\oM}(G_{\mathfrak S_0,\sigma_1,1})} e^{-{\oM}(G_{\mathfrak S_0,\sigma_1,1}(z_1, \langle \rho \rangle^{-1}\zeta_1 ))}  \right\|_{L^2(M_1)\rightarrow L^2(M_1)} \\
 \times \left\|e^{{\oM}(G_{\mathfrak S_0,\sigma_1,1}(z_1, \langle \rho \rangle^{-1}\zeta_1 ))}{\oM}(e^{-F_{\mathfrak S_0,\sigma_0}}(\rho)b_{N'}(\rho))\right\|_{L^2(M_1)\rightarrow L^2(M_1)} .
 \end{multline*}
Using the composition rule~\eqref{eq:compositionmfd2} and relying on the approximate expressions~\eqref{eq:symbolexpleft} and~\eqref{eq:symbolexpright} for the exponential of a pseudo-differential operator, 
we have on the one hand
$$
 \left\|e^{{\oM}(G_{\mathfrak S_0,\sigma_1,1})} e^{-{\oM}(G_{\mathfrak S_0,\sigma_1,1}(z_1, \langle \rho \rangle^{-1}\zeta_1 ))}  \right\|_{L^2\rightarrow L^2}  \leqslant C \langle  \rho \rangle^{2(\mathfrak S_0 + |\sigma_1|)}.
 $$
On the other hand, we recall that by construction of the sliced and global escape functions, we have $G_{\mathfrak S_0,\sigma_1,1}(z_1, \langle \rho \rangle^{-1}\zeta_1 )=G_{\mathfrak S_0,\sigma_1}(z_1,\rho,\zeta_1 )$ and Proposition~\ref{prop:comparisonresult} gives
\begin{equation}\label{eq:comparison-escape-fct}
G_{\mathfrak S_0, \sigma_0}(z_1,\rho,\zeta_1) \leqslant F_{\mathfrak S_0, \sigma_0}(z_1,\rho,\zeta_1)  + 2(\mathfrak S_0 + |\sigma_0|) \log \langle \rho \rangle
\end{equation}
in the region $\{|\zeta|\geqslant 1\}$. By the composition rule in $S^{\mathfrak S_0+|\sigma_1|}(\mrm{T}^*M_1)$ (see Appendix~\ref{aa:compactmfd}) and the approximate expressions~\eqref{eq:symbolexpleft} and~\eqref{eq:symbolexpright} for the exponential of a pseudo-differential operator, we can, modulo a bounded operator on $L^2$ (of norm $\mathcal{O}(\langle\rho\rangle^{2(\mathfrak{S}_0+|\sigma_1|)})$), focus on estimating the following norm:
$$
    \langle\rho\rangle^{\sigma_0-\sigma_1}\left\|{\oM}\left(b(\rho)e^{(\sigma_1-\sigma_0) \log g} e^{G_{\mathfrak S_0, \sigma_0}(\rho) - F_{\mathfrak S_0, \sigma_0}(\rho)}\right)\right\|_{L^2\rightarrow L^2} 
$$
 where by assumption $\sigma_1 - \sigma_0 <0$ and where $b\in \overline{S}^0(\mrm{T}^*(\R\times M_1))$ is independent of $r$ and equal to $1$ modulo $\overline{S}^{-1+}(\mrm{T}^*(\R\times M_1))$. By the Calder\'on-Vaillancourt theorem (see e.g. \eqref{eq:calderonmfd}) and by~\eqref{eq:comparison-escape-fct},
 we get the estimate
 $$
  \left\|{\oM}\left(b(\rho)e^{(\sigma_1-\sigma_0) \log g} e^{G_{\mathfrak S_0, \sigma_0} (\rho)- F_{\mathfrak S_0, \sigma_0}(\rho)}\right)\right\|_{L^2\rightarrow L^2} \leqslant C \langle \rho \rangle^{ 2(\mathfrak S_0 + |\sigma_0|) }
  $$
  We finally obtain the bound
 $$
 \left\|e^{{\oM}\left(G_{\mathfrak S_0,\sigma_1}\right)}{\oM}\left(e^{-F_{\mathfrak S_0,\sigma_0}(\rho)}b_{N'}(\rho)\right)\right\|_{L^2\rightarrow L^2} \leqslant C \langle \rho \rangle^{4(\mathfrak S_0+ |\sigma_1|)},
 $$
 and we can conclude, relying on the fact that $m_1 \geqslant 5(\mathfrak S_0+|\sigma_1|)$.
 \end{proof}

\subsection{Proof of Proposition~\ref{p:bilinear}}
\label{sec:4.6}


We are eventually in position to prove Proposition~\ref{p:bilinear}. Before getting to this matter, we first give a description of the Hamiltonian vector field $X_\Phi$ generated by a smooth function $\Phi$ on $\Sigma$.

\subsubsection{Hamiltonian vector fields}

In what follows we will work with the coordinates $(r,z_1)=(r,x,\xi_1)\in\R_+^*\times M_1$. Recall that the connection map associated with the Riemannian metric $\mathrm{g}$ on $\Sigma$ is defined as follows. Let $z=(x,\xi)\in M=\mrm{T}^*\Sigma$ and let $z(t)=(x(t),\xi(t))$ be a curve in $M$ such that $z(0)=z$ and $z'(0)=Z$. Then, one sets
$$
\mathcal{C}_z:Z\in \mrm T_z M\mapsto \nabla_{z'(0)}\xi(0)\in \mrm T_x^*\Sigma,
$$
where $\nabla$ is the Levi-Civita connection associated to $(\Sigma,\mathrm{g})$. One can verify that $\mathcal{C}_z$ induces an isomorphism between the horizontal space $\text{Ker}(D\pi(z)),$ with $\pi:(x,\xi)\in M\mapsto x\in\Sigma$ the canonical projection~\cite[Ch.~1]{Ruggiero2007}. Then, one has that $X_\Phi(z)=\mathcal{C}_z^{-1}(\dd\Phi(x))$ with $\dd\Phi(x)$ that can be decomposed as follows
$$
\dd\Phi(x)=\mathrm \mathrm{g}_x^*\left(\dd\Phi(x),\xi_1\right)\xi_1+ D^{\perp}_{\xi_1}\Phi(x).
$$
In the coordinates $(r,z_1)=(r,x,\xi_1)\in\R_+^*\times M_1$, it finally yields the following decomposition for the vector field $X_{\Phi}$:
\begin{equation}\label{eq:potentialvectorfield}
 X_{\Phi}=\frac{1}{r}\mathcal{C}_{z_1}^{-1}\left(D^{\perp}_{\xi_1}\Phi(x)\right) + \mathrm{g}_x^*\left(\dd\Phi(x),\xi_1\right) \partial_r,
\end{equation}
where the vector field
$$
X_{\Phi}^1(z_1)=\mathcal{C}_{z_1}^{-1}\left(D^{\perp}_{\xi_1}\Phi(x)\right)
$$
can be identified with a vector field on $M_1$. This can be reexpressed in terms of pseudo-differential operators on $\R\times M_1$ as defined in Appendix~\ref{aa:operatorvalued}. Namely, one has
\begin{equation}\label{eq:potentialvectorfieldpseudo}
 X_{\Phi}=\frac{1}{r}\mathbf{Op}\left(b_1(z_1;\zeta_1)\right)+\mathbf{Op}\left(\rho \Phi_1(z_1)\right),
\end{equation}
where $b_1(z_1;\zeta_1)$ is a polynomial of degree $1$ in the $\zeta_1$ variable. Namely, one has
\begin{equation}\label{eq:semi-normsymbolpotentialvf}
 |\beta|\geqslant 2\quad \Longrightarrow \quad p_{1,\alpha,\beta}(\rho \Phi_1(z_1))=p_{1,\alpha,\beta}(b_1(z_1;\zeta_1))=0,
\end{equation}
and
 \begin{equation}\label{eq:semi-normsymbolpotentialvf2}
 |\beta|=0, 1 \quad \Longrightarrow \quad  p_{1,\alpha,\beta}(\rho \Phi_1(z_1))+p_{1,\alpha,\beta}(b_1(z_1;\zeta_1))\leqslant C_{\alpha,\beta}\|\Phi\|_{\mathscr{C}^{|\alpha|+1}(\Sigma)},
\end{equation}
 where $p_{1,\alpha,\beta}$ are the semi-norms associated to the class of symbols $\overline{S}^{1}(\mrm{T}^*(\R\times M_1))$ (see Appendix~\ref{aa:mfd}).

 \begin{rema}\label{r:potential-kernel}
For later purposes, we record a few useful relations when $\Phi=\Phi(u)=\mathrm{K}\pi_* u$, for $\mathrm{K}f(x)=\int_{\Sigma}K(x,y)f(y)\,\dd\text{vol}_{\mathrm{g}}(y)$ the smoothing interaction operator in the Vlasov equation~\eqref{eq:vlasov}. Namely, we can set
\begin{equation}\label{eq:radialcomponent-vf-potential}
K_1(x,\xi_1,y)=\mathrm{g}_x^*(\dd_xK(x,y),\xi_1)\in\mathscr{C}^\infty (M_1\times\Sigma,\R),
\end{equation}
and 
\begin{equation}\label{eq:tangentialcomponent-vf-potential}
\mathcal{X}_1(x,\xi_1,y)=\mathcal{C}_{z_1}^{-1}D_{\xi_1}^\perp K(x,y)\in \mathscr{C}^\infty(M_1\times\Sigma,TM_1).
\end{equation}
This allows to rewrite $X_{\Phi(u)}$ as
\begin{equation}\label{eq:potentialvectorfield2}
X_{\Phi(u)}(r,z_1)=\frac{1}{r}X_{\Phi(u)}^1(z_1)+\Phi_1(u)(z_1)\partial_r,
\end{equation}
where
$$
\Phi_1(u)(z_1)=\int_{M}K_1(z_1,y)u(y,\eta)\dd {\mrm L}(y,\eta)
$$
and
$$
X_{\Phi(u)}^1(z_1)=\int_{M}\mathcal{X}_1(z_1,y)u(y,\eta)\dd {\mrm L}(y,\eta),
$$
where $L$ is the Liouville measure on $M$. Hence, the semi-norms of the corresponding symbols $b_1(z_1;\zeta_1,u)$ and $\Phi_1(u)(z_1)$ can be controlled in terms of certain integrals of $u$ over $M$.
\end{rema}

 \subsubsection{Proof of Proposition~\ref{p:bilinear}}

Recall that we aim at estimating
$$
\left\|\mathbf{B}_{\mathfrak S_0,\sigma_1}^{(m_1,N_1)}\varphi_{-\tau}^*\mathbf{P}\{\Phi,\chi u\} \right\|_{L^2},
$$
which, thanks to~\eqref{eq:potentialvectorfield}, can be rewritten as
$$
\left\|\mathbf{B}_{\mathfrak S_0,\sigma_1}^{(m_1,N_1)}\varphi_{-\tau}^*\mathbf{P}\left(a_\Phi \partial_r+r^{-1}Y_\Phi\right) \chi u\right\|_{L^2},
$$
where $a_\Phi$ (resp. $Y_\Phi$) is a smooth function (resp. vector field) on $M_1$. Recall from~\eqref{eq:semi-normsymbolpotentialvf} and~\eqref{eq:semi-normsymbolpotentialvf2} that their $\mathscr{C}^N$ norms are controlled by the $\mathscr{C}^{N+1}$ norm of $\Phi$. 

We now fix $\chi_1$ which is compactly supported in $I$ and which is identically equal to $1$ on the support of $\chi$. Thanks to~\eqref{eq:pseudoinverseA}, one can then write, for every $N'\geqslant 1$,
\begin{multline*}
\left\|\mathbf{B}_{\mathfrak S_0,\sigma_1}^{(m_1,N_1)}\varphi_{-\tau}^*\mathbf{P}\{\Phi,\chi u\}\right\|_{L^2}\leqslant C_{N'}\left\|\mathbf{B}_{\mathfrak S_0,\sigma_1}^{(m_1,N_1)}\mathbf{P}\varphi_{-\tau}^*\chi_1\left(a_\Phi \partial_r+r^{-1}Y_\Phi\right)\chi_1 \right\|_{H^{N'}\rightarrow L^2} \|u\|_{L^2}\\
  +\left\|\mathbf{B}_{\mathfrak S_0,\sigma_1}^{(m_1,N_1)}\varphi_{-\tau}^*\mathbf{P}\chi_1 \left(a_\Phi \partial_r+r^{-1}Y_\Phi\right)\chi_1\mathbf{Op}\left(b_{N'}e^{-F_{\mathfrak S_0,\sigma_0}}\right)\mathbf{A}_{\mathfrak S_0,\sigma_0}\chi u\right\|_{L^2}.
\end{multline*}
Thanks to Lemmas~\ref{l:NZradial},~\ref{l:comparisonnormdr} and~\ref{l:comparisonnormtgt}, there exists and $\vartheta_0>0$ (that is independent of the choice of the interval $I$) such that, for every $\mathfrak S_0\geqslant 3$, $m_1\geqslant 2$, $N_1\geqslant 10\mathfrak S_0+m_1+4$ and $-2\leqslant \sigma_1\leqslant0$, one can find $N', N\gg 1$ and $C>0$ such that
\begin{align*}
&\left\|\mathbf{B}_{\mathfrak S_0,\sigma_1}^{(m_1,N_1)}\mathbf{P}\varphi_{-\tau}^*\chi_1\left(a_\Phi \partial_r+r^{-1}Y_\Phi\right)\chi_1 \right\|_{H^{N'}\rightarrow L^2} \\
&\qquad \qquad \leqslant \left\|\mathbf{B}_{\mathfrak S_0,\sigma_1}^{(m_1,N_1)}\mathbf{P}\varphi_{-\tau}^*\chi_1 \mathbf{B}_{-\mathfrak S_0,-\sigma_1}^{(-m_1+2)} \|_{L^2 \rightarrow L^2} \| \mathbf{B}_{\mathfrak S_0,\sigma_1}^{(m_1-2)} \left(a_\Phi \partial_r+r^{-1}Y_\Phi\right)\chi_1 \right\|_{H^{N'}\rightarrow L^2}  \\
&\qquad \qquad \leqslant
Ce^{-\vartheta_0 \tau \min I }\|\Phi\|_{\mathscr{C}^N}.
\end{align*}
Hence, we are left with estimating the term
\begin{align*}
&\left\|\mathbf{B}_{\mathfrak S_0,\sigma_1}^{(m_1,N_1)}\varphi_{-\tau}^*\mathbf{P}\chi_1 \left(a_\Phi \partial_r+r^{-1}Y_\Phi\right)\chi_1\mathbf{Op}\left(b_{N'}e^{-F_{\mathfrak S_0,\sigma_0}}\right)\right\|_{L^2\rightarrow L^2} \\
&\qquad \leqslant \left\|\mathbf{B}_{\mathfrak S_0,\sigma_1}^{(m_1,N_1)}\varphi_{-\tau}^*\mathbf{P}\chi_1\mathbf{B}_{-\mathfrak S_0,-\sigma_1}^{(-m_1+2)} \right\|_{L^2\rightarrow L^2} \left\| \mathbf{B}_{\mathfrak S_0,\sigma_1}^{(m_1-2)} \left(a_\Phi \partial_r+r^{-1}Y_\Phi\right)\chi_1 \mathbf{B}_{-\mathfrak S_0,-\sigma_1-1}^{(-m_1+5)} \right\|_{L^2\rightarrow L^2}  \\
&\qquad  \qquad \times \left\| \mathbf{B}_{\mathfrak S_0,\sigma_1+1}^{(m_1-5)} \mathbf{Op}\left(b_{N'}e^{-F_{\mathfrak S_0,\sigma_0}}\right)\right\|_{L^2\rightarrow L^2} . 
\end{align*}
We can therefore again apply Lemmas~\ref{l:NZradial},~\ref{l:comparisonnormdr} and~\ref{l:comparisonnormtgt} in addition with Lemma~\ref{l:comparisonnormNZ-ours}. This yields the existence of $\vartheta_0>0$ (that is still independent of the choice of the interval $I$) such that, for every $\mathfrak S_0\geqslant 3$, $m_1\geqslant 50\mathfrak S_0$, $N_1\geqslant 50 \mathfrak S_0+m_1$, and $\sigma_1+1<\sigma_0\leqslant 0$ with $-2\leqslant \sigma_1\leqslant -1$, one can find $N,C>0$ such that
$$
\left\|\mathbf{B}_{\mathfrak S_0,\sigma_1}^{(m_1,N_1)}\varphi_{\tau}^*\mathbf{P}\chi_1 \left(a_\Phi \partial_r+r^{-1}Y_\Phi\right)\chi_1\mathbf{Op}\left(b_{N'}e^{-F_{\mathfrak S_0,\sigma_0}}\right)\right\|_{L^2\rightarrow L^2}\leqslant Ce^{-\vartheta_0 \tau \min I }\|\Phi\|_{\mathscr{C}^N}.
$$
The proof of  Proposition~\ref{p:bilinear} is finally complete.

\section{Global existence}
\label{sec:global}

The goal of this short section is to prove the following result.

\begin{theo}
\label{theo:globalexistence}
Let $N_0 \in \mathbb{N}$.
Let $u_0 \in \mscr C_c^{N_0}(M).$ Then there exists a unique $u \in \mscr C^1\left(\R, L^1(M,\dd {\mrm L})\right)$ satisfying \eqref{eq:vlasov} on $\mathbb{R}$. In fact $u \in \mscr C^k\left(\R,  \mscr C^{N_0-k}_c(M)\right)$ for all $k=0,\ldots,N_0$.\end{theo}

All along this section, we will use the Liouville measure $\dd {\mrm L}(r,z_1)=r^{n-1}\dd r \dd {\mrm L}_1(z_1)$ rather than $\dd r \dd {\mrm L}_1(z_1)$. We start with a local existence result.






\begin{lemm}
\label{lem:localexistence}
Let $N_0 \in \mathbb{N}$.
Let $u_0 \in \mscr C_c^{N_0}(M).$ Then there exist $\varepsilon > 0$ and  a unique $u \in \mscr C^1\left([-\varepsilon, \varepsilon], L^1(M,\dd {\mrm L})\right)$ satisfying \eqref{eq:vlasov} on $[-\varepsilon, \varepsilon]$. In fact $u \in \mscr C^k\left([-\varepsilon, \varepsilon],  \mscr C^{N_0-k}_c(M)\right)$ for all $k=0,\ldots,N_0$.
Moreover, $\varepsilon$ depends only on $\| u_0\|_{L^1(M)}$, $\| \nabla u_0\|_{L^\infty(M)}$ and on the support of $u_0$.
\end{lemm}

\begin{proof}
The proof is based on a classical argument involving the Banach theorem. Note that $u \in \mathscr{C}^1\left([-\varepsilon, \varepsilon], L^1(M)\right)$ satisfies \eqref{eq:vlasov} if and only if one has
\begin{equation}\label{eq:fixedpoint}
u(t) = \varphi_{-t}^*u_0 - \int_0^t \varphi_{s-t}^* \{\Phi(u(s)), u(s)\} \dd s, \quad |t|\leqslant \varepsilon.
\end{equation}
If $v \in \mscr C^0([-\varepsilon, \varepsilon], L^1(M))$, we let $u = \Psi(v) \in \mscr C^1([-\varepsilon, \varepsilon], L^1(M))$ be the unique solution of the equation
$$
\partial_t u + (X + X_{\Phi(v)}) u=0, \quad u|_{t=0} = u_0.
$$
where $X_{\Phi(v)} \in \mscr C^\infty(M, TM)$ is the Hamiltonian vector field generated by $\Phi(v)$, that is $X_{\Phi(v)} f = \{\Phi(v), f\}$ for $f \in \mscr C^1(M)$. We shall explain in a few lines why it is indeed defined on $[-\varepsilon,\varepsilon]$. Equivalently, one can write
$$
\Psi(v)(t) = u_0 \circ \psi^v_{t, 0}
$$
where $\psi^v_{t,s} : M \to M$ is the non-autonomous flow associated to the vector field $X + X_{\Phi(v)}(t)$, defined by
\begin{equation}\label{eq:cauchylip}
\partial_t \psi^v_{t,s}(z) + [X + X_{\Phi(v)(t)}](\psi^v_{t,s}(z))=0, \quad \psi^v_{s,s}(z) = z, \quad z \in M.
\end{equation}
Recall from the Cauchy--Lipschitz Theorem (see \cite[Theorem V.4.1]{hartman1982ordinary} for example) that the map $(t,z)\in[-\varepsilon,\varepsilon]\times M\mapsto \psi_{t,0}^v(z)\in M$ is smooth with respect to the $z$-variable and of class $\mathscr{C}^1$ with respect to the time variable (as $\Phi(v)$ is $\mathscr{C}^\infty$).  In~\eqref{eq:potentialvectorfield}, we already gave an exact expression of this vector field. In particular, given an orbit $(r(t),z_1(t))$ of this vector field, one has
$$
r(t)=r(0)+\int_0^t\Phi_1(v)(s)\,\dd s
,$$
so that every orbit remains in a compact set of $M$. In particular, if $u_0$ is supported in the region $\{|\xi|\leqslant R_0\}$, then the solution $u$ is indeed defined on $[-\varepsilon,\varepsilon]$ and supported in $\{|\xi|\leqslant R_0+C_K\varepsilon |v|_{L^1(M)}\}$, where $C_K>0$ depends only on the interaction kernel $K$.

We want to find a fixed point of $\Psi$. Let $B$ be the ball of radius $|u_0|_{L^1}$ in the space $\mscr C^0([-\varepsilon, \varepsilon], L^1(M))$, for the norm
$$
\|v\|_1 = \sup_{|t|\leqslant \varepsilon} \|v(t)\|_{L^1(M, \dd {\mrm L})}.
$$
Note that $\Psi : B \to B$. Let $v_1, v_2 \in B$ and set $u_j = \Psi(v_j)$. We write
$$
\partial_t(u_2-u_1)= -\left(X+X_{\Phi(v_2)}\right)(u_2-u_1) +X_{\Phi(v_2)-\Phi(v_1)}u_1,
$$
from which we infer that
$$
u_2(t)-u_1(t)=\int_0^t (\psi_{t,s}^{v_2})^*X_{\Phi(v_2)(s)-\Phi(v_1)(s)}u_1(s)\dd s.
$$
As $\psi_{t,s}^{v_2}$ is volume preserving, one finds that

$$
\begin{aligned}
\left\|u_2(t)- u_1(t)\right\|_{L^1(M,\dd {\mrm L})} &\leqslant \int_0^t \left|(\psi_{t,s}^{v_2})^* X_{\Phi(v_2-v_1)} u_2(s)\right|_{L^1(M,\dd {\mrm L})} \dd s \\
&\leqslant \varepsilon \sup_{|t|\leqslant\varepsilon}\|X_{\Phi(v_2-v_1)}(t)\|_\infty \|u_1(t)\|_{\mscr C^1(M)}.
\end{aligned}
$$
Here the norm on $\mscr C^1(M)$ is taken with respect to any smooth metric on $M$. One has $\|X_{\Phi(v_2 - v_1)}\|_\infty \leqslant C \|\Phi(v_2 - v_1)\|_{\mathscr{C}^1} \leqslant \widetilde{C} \|v_2 - v_1\|_{1}$, and 
$$
\|u_1\|_{\mscr C^1(M)} = \sup_{|t|\leqslant \varepsilon} \|u_0 \circ \psi^{v_1}_{t,0}\|_{\mscr C^1(M)} \leqslant \|\nabla u_0\|_{\infty} \sup_{|t|\leqslant \varepsilon}|\dd \psi_{t,0}^{v_1}|_{L^\infty(Q)}
$$
where $Q = \{|\xi|\leqslant R_0+C_K\|u_0\|_{L^1}\varepsilon\}$. Now, by differentiating~\eqref{eq:cauchylip}, we find that $\dd \psi_{t,s}^{v_1}$ solves a linear differential equation, namely
$$
\partial_t\dd \psi_{t,s}^{v_1}(z)=\dd \left(X+X_{\Phi(v_1)(t)}\right)\left(\psi_{t,s}^{v_1}(z)\right)\dd \psi_{t,s}^{v_2}(z),\quad \dd \psi_{s,s}^{v_2}(z)=\text{Id}.$$
Applying Gronwall's Lemma, one finds that $|\dd \psi_{t,0}^{v_1}|_{L^\infty(Q)}$ is bounded by 
$$Ce^{C\varepsilon(R_0+C_K(\varepsilon+1)|v|_{L^1})}\leqslant Ce^{\widetilde{C}\varepsilon(R_0+|u_0|_{L^1})},$$ 
where $C,\widetilde{C}>0$ are constants that depend only on the interaction kernel. This implies the existence of a constant $C_{\mathrm{g},K}$ (depending only on the interaction kernel $K$ and on the metric $\mathrm{g}$) such that, for all $|t|\leqslant\varepsilon$,
$$
\|u_2(t)-u_1(t)\|_{L^1(M,\dd {\mrm L})}\leqslant C_{\mathrm{g},K}\varepsilon e^{\varepsilon C_{\mathrm{g},K}(R_0+\|u_0\|_{L^1})}\|\nabla u_0\|_{L^\infty}\|v_2-v_1\|_{L^1(M,\dd {\mrm L})}.
$$

In particular, $\Psi$ is a contraction as soon as $\varepsilon$ is small enough. Now the Banach theorem yields the existence of a unique fixed point $u \in B$ of $\Psi$. Next, we write that
\begin{equation}\label{eq:ufp}
u(t) = u_0 \circ \psi^{\urm}_{t, 0}, \quad |t|\leqslant\varepsilon.
\end{equation}
We already saw that $(t,z)\in[-\varepsilon, \varepsilon] \times M  \mapsto \psi^{\urm}_{t,0}(z)\in M$ is smooth with respect to $z$, $\mathscr{C}^1$ with respect to $t$ and proper (so that $u(t)$ has compact support). Applying \eqref{eq:ufp} inductively, one obtains the expected regularity of $u$.
\end{proof}

On the other hand, exploiting the Hamiltonian structure of the Vlasov equation, we have

\begin{lemm}
\label{lem:conservation-Lp}
Let $J$ be an interval containing $0$. Let $u$ be a solution to~\eqref{eq:vlasov} on $J$. Then for all $p \geqslant 1$ and all $t \in J$,
$$
\| u(t) \|_{L^p(M,\dd {\mrm L})} = \| u_0 \|_{L^p(M,\dd {\mrm L})},
$$
where the $L^p$ norm is taken with respect to the Liouville measure.
\end{lemm}

Finally let us proceed with the proof of Theorem~\ref{theo:globalexistence}.

\begin{proof}
Let $T_{\mathrm{max}}$ be the maximal (forward) time of existence associated to the initial condition $u_0$, that is 
$$
T_{\mathrm{max}} = \sup \left\{T>0, \exists u \in \mathscr{C}^1([0,T], L^1(M,\dd {\mrm L})) \text{ solution to  } \eqref{eq:vlasov} \text{  with initial condition  } u_0 \right\}.
$$
Thanks to Lemma~\ref{lem:localexistence}, $T_{\mathrm{max}}>0$. If $T_{\mathrm{max}} = + \infty$, then we may conclude as in the end of the proof of Lemma~\ref{lem:localexistence}. Let us show that the case $T_{\mathrm{max}} <+  \infty$ is impossible. Would it be the case, we could proceed as in the proof of Lemma~\ref{lem:localexistence} to control $\| \nabla u(t) \|_{L^\infty(M)}$ in the following way. First, using Lemma~\ref{lem:conservation-Lp} with $p=1$, if $u_0$ is supported in $\{|\xi|\leqslant R_0\}$, then $u(t)$ is supported in $Q_{\mathrm{max}}=\{|\xi|\leqslant R_0+C_K T_{\mathrm{max}}\|u_0\|_{L^1}\}$. Then, we can write
$$
\|\nabla u(t)\|_{L^\infty} = \|\nabla\left(u_0 \circ \psi^{u}_{t,0}\right)\|_{L^\infty} \leqslant \|\nabla u_0\|_{L^\infty} \|\dd \psi_{t,0}^{u}\|_{L^\infty(Q_{\mathrm{max}})}.
$$
One thus finds that $|\dd \psi_{t,0}^{u}|_{L^\infty}$ is bounded by $Ce^{C T_{\mathrm{max}} (R_0+(1+T_{\mathrm{max}})|u_0|_{L^1})}$, where $C>0$ is a constant that depends only on the interaction kernel $K$ and on the metric $\mathrm{g}$. 
In particular, there exists a constant $C_\mathrm{max}$ with a dependence with respect to the solution only through  $\| u_0 \|_{L^1(M,\dd {\mrm L})}$, $ \|\nabla u_0 \|_{L^\infty(M)}$, the support of $u_0$ and $T_{\mathrm{max}}$, such that for all $t \in [0,T_\mathrm{max})$, we have $\| \nabla u(t) \|_{L^\infty(M)} \leqslant C_\mathrm{max}$. 


Now let $\eta>0$ to be fixed later and consider the Cauchy problem
\begin{equation}\label{eq:vlasovTmax}
 \partial_t \widetilde{u}+ \left\{\mathrm H + \Phi(\widetilde{u}), \, \widetilde{u} \right\}=0,\quad \widetilde{u}|_{t=T_{\mathrm{max}} - \eta }=u(T_{\mathrm{max}} - \eta ). 
 \end{equation}
Owing to Lemma~\ref{lem:localexistence} (using the invariance by translation in time of the equation), there exists $\varepsilon>0$ depending only on $\| u(T_{\mathrm{max}} - \eta )  \|_{L^1(M,\dd {\mrm L})}$, $\| \nabla u(T_{\mathrm{max}} - \eta ) \|_{L^\infty(M)}$, and the support of $u(T_{\mathrm{max}} - \eta )$, and therefore only on $\| u_0 \|_{L^1(M,\dd {\mrm L})}$, $ \|\nabla u_0 \|_{L^\infty(M)}$, the support of $u_0$ and $T_{\mathrm{max}}$, such that \eqref{eq:vlasovTmax}
is uniquely solvable on $[T_{\mathrm{max}} - \eta - \varepsilon, T_{\mathrm{max}} - \eta  + \varepsilon]$. We can therefore pick $\eta = \varepsilon/2$, which shows by a uniqueness argument that the solution to~\eqref{eq:vlasov} can be continued beyond $T_{\mathrm{max}}$. It leads to the expected contradiction.
We argue identically for the backward times of existence, which concludes the proof of the theorem.
\end{proof}

\section{Nonlinear exponential stability}
\label{sec:mainproof}

This final Section is dedicated to the proof of Theorem~\ref{thm:main}.
To this purpose, we argue with a bootstrap argument, which we initiate in Subsection~\ref{sec:6.1}. Loosely speaking, the idea is to work on the largest interval of nonnegative times $\mathcal{J}_{k}(\varepsilon)$ on which the potential $\Phi(u)$ decays exponentially fast for a well-chosen rate (with a control by a small enough constant $\varepsilon$), with the aim to show that $\mathcal{J}_{k}(\varepsilon)$ is actually $\mathbb{R}_+$: that this is the case corresponds precisely to the main result of this Section, namely Theorem~\ref{theo:connected}. Theorem~\ref{thm:main} then follows.

Subsection~\ref{sec:6.2} corresponds to a preliminary step in which we show that on  $\mathcal{J}_{k}(\varepsilon)$, the support of the solution to the Vlasov equation~\eqref{eq:vlasov} remains far from the null section. In Subsection~\ref{sec:6.3}, we reduce the proof of Theorem~\ref{thm:main} to two statements, Lemmas~\ref{l:mainPhi} and~\ref{l:boundednorm}. Lemma~\ref{l:mainPhi} states that a refined control of $\Phi(u)$ (with exponential decay) can be obtained on $\mathcal{J}_{k}(\varepsilon)$, modulo the uniform control of a sliced anisotropic Sobolev norm. This uniform control correspond precisely to Lemma~\ref{l:boundednorm}.
The proof of Lemma~\ref{l:mainPhi} is performed in Subsection~\ref{sec:6.4} and relies crucially on many results from Section~\ref{sec:anisotropic}.
Subsection~\ref{sec:6.5} is then about the proof of Lemma~\ref{l:boundednorm}. We again reduce it to the proof of two other statements, Lemmas~\ref{l:BtoA} and~\ref{l:polynomialbound}.  Lemma~\ref{l:BtoA} is a refined estimate of the sliced anisotropic Sobolev norm, which shows that it is sufficient to prove an estimate of a global anisotropic Sobolev norm with a polynomial growth in time. To prove Lemma~\ref{l:BtoA}, we use crucially the bilinear estimate of Proposition~\ref{p:bilinear} from Section~\ref{sec:anisotropic}. This polynomial growth is precisely the purpose of Lemma~\ref{l:polynomialbound}.
Everything therefore finally boils down to an estimate of the global anisotropic Sobolev norm, that is associated with our global escape function.
This is obtained thanks to a microlocal energy estimate for the Vlasov equation~\eqref{eq:vlasov}, which we perform in Subsection~\ref{ss:proof-microlocal}.
We conclude with Subsection~\ref{sec:weak} in which we finally prove the weak convergence result (with exponential speed) of Theorem~\ref{thm:main}.

\subsection{The bootstrap argument}
\label{sec:6.1}
We are now in position to prove Theorem~\ref{thm:main}. Let $u_0$ be a smooth function in $\mathscr{C}^{k_0}_c(\R_+^*\times M_1)$ (with $k_0\gg 1$) and denote by $u(t,r,z_1)$ the solution to~\eqref{eq:vlasov} given by Theorem~\ref{theo:globalexistence}. Let $r_0\in\left]0,1\right[$ be such that
$$
\supp u_0 \subset \{(x,v) \in \mathrm{T}^*\Sigma~:|v| \in[ r_0,r_0^{-1}]\}.
$$
Let $0<\varepsilon\leqslant 1$ be a (small) parameter to be fixed later on and take also $\vartheta=\vartheta_0 r_0/4$, where $\vartheta_0>0$ is the rate of convergence appearing in Section~\ref{sec:anisotropic}, namely the one given by Proposition~\ref{p:bilinear}, Lemma~\ref{l:NZradial}, Corollaries~\ref{c:smoothNZ}, and~\ref{c:smoothcorrelation}, and~\eqref{eq:liverani}. This parameter $\vartheta$ is fixed once and for all. For every $k\geqslant1$, we introduce the following subinterval of $\R_+$:
\begin{equation}
\label{def:Tstar-boostrap}
\mathcal{J}_{k}(\varepsilon,u_0) = \Bigg\{ T \geqslant 0:\ \forall t \in [0,T],\ 
 e^{\vartheta t} \| \Phi(u (t))  \|_{\mathscr{C}^k(\Sigma)}\leqslant \varepsilon \Bigg\}.
\end{equation}
Imposing $\|u_0\|_{L^1(M)}$ small enough with respect to $\varepsilon$ (in terms of $\|K\|_{\mathscr{C}^k}$), this defines by continuity a nonempty interval which is closed by continuity of the solution with respect to the time variable $t$. Our goal is to show that this interval is also open for appropriate choices of the parameters $\varepsilon$ and $k$. This will imply that this interval is $\R_+$. 
We will then be able to complete the proof of Theorem~\ref{thm:main} in Section~\ref{sec:weak}. 
Hence, the main statement of this Section reads:
\begin{theo}\label{theo:connected} 
There exist $k,k_0\geqslant 0$ such that, for all  $r_0\in(0,1)$, one can find $\varepsilon \in (0,1)$ and $c_0>0$ such that, for all $u_0\in\mathscr{C}^{k_0}_c(M^\times)$ that is supported in $[r_0,r_0^{-1}]\times M_1$ and satisfying $\|u_0\|_{\mathscr{C}^{k_0}}\leqslant c_0\varepsilon$, one has
$$
\mathcal{J}_{k}(\varepsilon,u_0)=\R_+,
$$
where $\mathcal{J}_{k}(\varepsilon,u_0)$ is defined by~\eqref{def:Tstar-boostrap}.
\end{theo}
In order to alleviate notations, we will drop the dependence in $u_0$ and write $\mathcal{J}_{k}(\varepsilon)$ instead of $\mathcal{J}_{k}(\varepsilon,u_0)$.

\subsection{Uniform control of the support}
\label{sec:6.2}
We begin our proof with the following observation.
\begin{lemm}\label{l:support} There exists a constant $c_{K,\mathrm{g}}>0$ depending only on the metric $\mathrm{g}$ and on the interaction kernel $K$ such that, if $\varepsilon<c_{K,\mathrm{g}}r_0\vartheta$, then, for every $t\in\mathcal{J}_{k}(\varepsilon)$,
\begin{equation}
\supp u(t) \subset \{(x,v) \in \mathrm{T}^*\Sigma~:~ |v| \in[r_0/2,2r_0]\}.
\end{equation}
\end{lemm}

\begin{rema} In the proof of Theorem~\ref{theo:connected}, $\varepsilon$ will first of all be chosen small enough so that Lemma~\ref{l:support} holds; another smallness constraint of similar nature will also appear in the subsequent Lemma~\ref{l:boundednorm}.

\end{rema}

\begin{proof} Let $T$ be an element of $\mathcal{J}_k(\varepsilon)$. Denote by $\psi_{t,s} : M \to M$ the flow associated to the vector field $X + X_{\Phi(u)}(t)$, defined by
$$
\partial_t \psi_{t,s}(z) = [X + X_{\Phi(u)}(t)](\psi_{t,s}(z)), \quad \psi_{s,s}(z) = z, \quad z \in M.
$$
Recall that $u(t)$ is the solution to~\eqref{eq:vlasov} with initial data $u_0$ compactly supported in $[r_0,r_0^{-1}]\times M_1$. Thanks to~\eqref{eq:potentialvectorfield2}, one has, for every $t\in[0,T]$ and for every initial data $z=(r(0),z_1(0))$ lying in $[r_0,r_0^{-1}]\times M_1$, one has
$$
r(t)=r(0)+\int_0^t\Phi_1(u(s))(z_1(s))\dd s,
$$
where we set $(r(t),z_1(t))=\psi_{t,0}(z)$. According to the expression of $\Phi_1$, one gets that there exists a constant depending only the interaction kernel $K$ and the metric $\mathrm{g}$ such that, for every $t\in[0,T]$,
$$
|r(t)-r(0)|\leqslant C_{K,g}\int_0^t\|\Phi(u(s))\|_{\mathscr{C}^1}ds\leqslant \varepsilon C_{K,\mathrm{g}}\vartheta^{-1},
$$
where we use that $t\in\mathcal{J}_k(\varepsilon)$ in the last inequality. Hence, if we pick $\displaystyle\varepsilon<\frac{r_0\vartheta}{2C_{K,\mathrm{g}}}$, we get the expected conclusion.
\end{proof}

\subsection{Proof of Theorem~\ref{theo:connected}}
\label{sec:6.3}
 We now explain how to prove Theorem~\ref{theo:connected}. 
We start with the following key statement. Let $\mathbf{B}_{3, -2}^{(m_1,N_1)}$ be the sliced anisotropic operator defined in~\eqref{eq:operatorvaluedanisotropicresolvent} with the parameters $(m_1,N_1)$ chosen so that Proposition~\ref{p:bilinear}, Lemma~\ref{l:NZradial} and Corollaries~\ref{c:smoothNZ} and~\ref{c:smoothNZ2} hold true.

\begin{lemm} 
\label{l:mainPhi}There exist $k,k_0\geqslant 1$ and $C_{K,r_0}>0$ (independent of $\varepsilon$ and $u_0$) such that, for all $t \in\mathcal{J}_k(\varepsilon)$,
\begin{multline*}
\|\Phi(u(t))\|_{\mathscr{C}^k} \leqslant C_{K,r_0}\|u_0\|_{\mathscr{C}^{k_0}}e^{-\vartheta t} \\
+ C_{K,r_0}\int_0^t e^{-\frac{r_0\vartheta_0}{2} (t-s)}\|\Phi(u(s))\|_{\mathscr{C}^{k}(\Sigma)}\|\mathbf{B}_{3, -2}^{(m_1,N_1)}\mathbf{P}(u(s))\|_{L^2(\R\times M_1)}\dd s.
\end{multline*}
\end{lemm}

Observe that a control of $\|\mathbf{B}_{3, -2}^{(m_1,N_1)}\mathbf{P}(u(s))\|_{L^2(\R\times M_1)}$ for all $s \in \mathcal{J}_k(\varepsilon)$  is needed for Lemma~\ref{l:mainPhi} to be relevant. This is the purpose of the next statement.

\begin{lemm}\label{l:boundednorm} There exist a constant $\widetilde{c}_{K,\mathrm{g}}>0$ depending only on $\mathrm{g}$ and $K$, as well as $k_0 \in \mathbb{N}$ and $C>0$ (independent of $\varepsilon$ and $u_0$) such that, for $\varepsilon<\widetilde{c}_{K,\mathrm{g}}r_0\vartheta$, for all $s\in \mathcal{J}_k(\varepsilon)$,
 $$
 \|\mathbf{B}_{3, -2}^{(m_1,N_1)}\mathbf{P}(u(s))\|_{L^2(\R\times M_1)}\leqslant C \|u_0\|_{\mathscr{C}^{k_0}}.
 $$
\end{lemm}

The reason for taking $\varepsilon$ small enough will come from the subsequent Lemma~\ref{l:polynomialbound}.
Hence, recalling that $\vartheta=\vartheta_0 r_0/4$ and up to increasing the value of $k_0$ (for the regularity of the initial data), one gets that, for every $t\in\mathcal{J}_k(\varepsilon)$, for every $u_0$ (supported in $[r_0,r_0^{-1}]\times M_1$) with $\|u_0\|_{\mathscr{C}^{k_0}}$ small enough,
\begin{equation}\label{eq:boundnonlinear8}
C_{K,r_0}\int_0^t e^{-\frac{r_0\vartheta_0}{2} (t-s)}\|\Phi(u(s))\|_{\mathscr{C}^{k}(\Sigma)}\|\mathbf{B}_{3, -2}^{(m_1,N_1)}\mathbf{P}(u(s))\|_{L^2(\R\times M_1)}\dd s\leqslant\frac{\varepsilon}{4}e^{-\vartheta t}.
\end{equation}
Gathering 
Lemmas~\ref{l:mainPhi} and~\ref{l:boundednorm}, we find that $\|\Phi(u(t))\|_{\mathscr{C}^k(\Sigma)}\leqslant\frac{\varepsilon}{2}e^{-\vartheta t}$ from which we infer that $\mathcal{J}_{k}(\varepsilon)$ is open. As already explained, this allows to conclude the proof of Theorem~\ref{theo:connected}. 

We are therefore left with the proofs of Lemmas~\ref{l:mainPhi} and~\ref{l:boundednorm}, which is precisely the purpose of the next two subsections.

\subsection{Proof of Lemma~\ref{l:mainPhi}}
\label{sec:6.4}
In this section, we prove Lemma~\ref{l:mainPhi}. We first use Duhamel formula to write
\begin{equation}\label{eq:duhamel}
u(t)=\varphi_{-t}^*(u_0)-\int_0^t\varphi_{s-t}^*\{ \Phi(u(s)),u(s)\}\dd s,
\end{equation}
which yields
\begin{equation}\label{eq:duhamelkernel}
\Phi(u(t))=\Phi\left(\varphi_{-t}^*(u_0)\right)-\int_0^t\Phi\left(\varphi_{s-t}^*\{ \Phi(u(s)),u(s)\}\right)\dd s.
\end{equation}
We pick $t\in\mathcal{J}_k(\varepsilon)$ and our goal is to estimate $\|\Phi(u(t))\|_{\mathscr{C}^k(\Sigma)}$. We begin with the linear term in~\eqref{eq:duhamelkernel}. One has
$$
\partial_x^\alpha\left(\Phi\left(\varphi_{-t}^*(u_0)\right)\right)=\int_{\mrm{T}^*\Sigma}\partial_x^\alpha K(x,y) u_0(r,\varphi_{-tr}(y,\eta_1))\dd {\mrm L}(y,\eta_1).
$$
Thanks to Corollary~\ref{c:smoothcorrelation} and as $\int_\Sigma K(x,y)\dd \text{vol}_{\mathrm{g}}(y)=0$ for every $x\in\Sigma$, there exists some $k_0\geqslant 1$ (depending only on $(\Sigma,\mathrm{g})$) and a constant $C_{k,r_0}>0$ (depending also on $k$ and $r_0$) such that
\begin{equation}\label{eq:boundlinearterm}
 \|\Phi(\varphi_{-t}^*(u_0))\|_{\mathscr{C}^k(\Sigma)}\leqslant C_{k,r_0} e^{-\vartheta t} \|K\|_{\mathscr{C}^{k_0+k}}\|u_0\|_{\mathscr{C}^{k_0}}.
\end{equation}
Note that this step did not rely on the fact that $t$ belongs to $\mathcal{J}_k(\varepsilon)$. We will now deal with the nonlinear term for which this assumption will be crucially used. 
We now write, for $\chi\in\mathscr{C}^{\infty}_c(\R_+^*)$ which is equal to $1$ on $[r_0/2,2r_0^{-1}]$ (where $u$ is supported thanks to Lemma~\ref{l:support}),
\begin{multline}\label{eq:boundnonlinear1}
 \int_0^t\partial_x^\alpha\Phi\left(\varphi_{s-t}^*\{ \Phi(u(s)),u(s)\}\right)\dd s\\
 =\int_0^t\left\langle \partial_x^\alpha K(x,\pi_1\circ\varphi_{(t-s)r}(z_1))\chi(r)r^{n-1} ,\chi(r)\{\Phi(u(s)),u(s)\}(r,z_1)\right\rangle_{L^2(\R\times M_1)}\dd s,
\end{multline}
where $\pi_1:(y,\eta_1)\in M_1=S^*\Sigma\mapsto y\in\Sigma$ and where the $L^2$ scalar product is taken with respect to the measure $\dd r\dd {\mrm L}_1(z_1)$. Using again $\int_{\Sigma}K(x,y)\dd \text{Vol}_\mathrm g(y)=0$, one can rewrite~\eqref{eq:boundnonlinear1} as
\begin{multline}\label{eq:boundnonlinear2}
 \int_0^t\partial_x^\alpha\Phi\left(\varphi_{s-t}^*\{ \Phi(u(s)),u(s)\}\right)\dd s\\
 =\int_0^t\left\langle \varphi_{t-s}^*\mathbf{P}\left(\partial_x^\alpha K(x,\pi_1(z_1))\chi(r)r^{n-1}\right) ,\chi(r)\{\Phi(u(s)),u(s)\}(r,z_1)\right\rangle_{L^2(\R\times M_1)}\dd s.
\end{multline}
Next we define
$
\mathbf Q = \mathrm{Id} - \mathbf P
$
so that
$$
\mathbf{Q}(v)=\int_{M_1}v(z_1,r)\dd {\mrm L}_1(z_1), \quad v \in \mathscr C^\infty_c(\R_+^*\times M_1).
$$
We now decompose $u(s)=\mathbf{Q}(u(s))+\mathbf{P}(u(s))$ and use~\eqref{eq:potentialvectorfield2} to rewrite
$$\{\Phi(u(s)),\mathbf{Q}(u(s))\}(r,z_1)=
\Phi_1(u(s))(z_1)\partial_r \mathbf{Q}(u(s)).
$$
Recall now that $X=rX_1$ is the geodesic vector field generating $\varphi_\tau$ so that
$$\partial_r\varphi_\tau^*=\varphi_\tau^*\left(\partial_r +\tau X_1\right).
$$
Hence, since $\chi$ is identically equal to $1$ on the support of $u(s)$ (as $s\in\mathcal{J}_k(\varepsilon)$), one has
\begin{equation}\label{eq:boundnonlinear3}
\begin{aligned}
 &\int_0^t\partial_x^\alpha\Phi\left(\varphi_{s-t}^*\{ \Phi(u(s)),\mathbf{Q}(u(s))\}\right)\dd s = \\
 &-\int_0^t(t-s)\left\langle \varphi_{t-s}^*\mathbf{P}\left(\partial_x^\alpha K(x,\pi_1(z_1))\chi(r)\right) ,\chi(r)r^{n-1}\mathbf{Q}(u(s))(r) X_1\Phi_1(u(s))\right\rangle_{L^2(\R\times M_1)}\dd s\\
& -(n-1)\int_0^t\left\langle \varphi_{t-s}^*\mathbf{P}\left(\partial_x^\alpha K(x,\pi_1(z_1))\chi(r)\right) ,\chi(r)r^{n-2}\mathbf{Q}(u(s))(r)\Phi_1(u(s))\right\rangle_{L^2(\R\times M_1)}\dd s.
 \end{aligned}
\end{equation}
We will bound the first term on the right-hand side of~\eqref{eq:boundnonlinear3}, the second one can be handled similarly. It writes, after integrating by parts with respect to the $r$ variable,
 $$
 \int_{\R_+}\chi^2(r)\mathbf{Q}(u(s))(r)\left(\int_{M_1}\partial_x^{\alpha}K(x,\varphi_{(t-s)r}(z_1))X_1(\Phi_1(u(s)))(z_1)\dd {\mrm L}_1(z_1)\right)r^{n-1} \dd r.
 $$
Applying~\eqref{eq:liverani} together with~\eqref{eq:semi-normsymbolpotentialvf2}, there exist a constant $C_{K,\alpha,r_0}>0$ (depending on the convolution kernel, $r_0$ and $(\Sigma,\mathrm{g})$) such that, for every $t\in\mathcal{J}_k(\varepsilon)$, the modulus of the above integral is bounded from above by
\begin{equation}\label{eq:boundnonlinear4}
\begin{aligned}
 & 
 C_{K,\alpha,r_0}\|u_0\|_{L^1}\int_0^t(t-s)e^{-\frac{\vartheta_0r_0}{2}(t-s)}\|\Phi_1(u(s))\|_{\mathcal{C}^{k_0+1}}\dd s\\
 & \hspace{3cm} \leqslant C_{K,\alpha,r_0}\|u_0\|_{L^1}\int_0^t(t-s)e^{-\frac{\vartheta_0r_0}{2}(t-s)}\|\Phi(u(s))\|_{\mathcal{C}^{k_0+2}}\dd s \\
 & \hspace{3cm} \leqslant \widetilde{C}_{K,\alpha,r_0}\|u_0\|_{L^1}\varepsilon  e^{-\vartheta t},
\end{aligned}
\end{equation}
where we used that $\|u(s)\|_{L^1(r^{n-1}\dd r\dd {\mrm L}_1)}=\|u_0\|_{L^1(r^{n-1}\dd r\dd {\mrm L}_1)}$ for every $s\geqslant 0$ and that $\vartheta=\vartheta_0 r_0/4$. A similar estimate holds for the second term of the right-hand side of~\eqref{eq:boundnonlinear3} and thus 
we are left with analyzing
\begin{multline*}
 \int_0^t\partial_x^\alpha\Phi\left(\varphi_{s-t}^*\{ \Phi(u(s)),\mathbf{P}(u(s))\}\right)\dd s\\
 =-\int_0^t\left\langle \varphi_{t-s}^*\mathbf{P}\left(\partial_x^\alpha K(x,\pi_1(z_1))r^{n-1}\chi(r)\right) ,\{\Phi(u(s)),\mathbf{P}(u(s))\}(r,z_1)\right\rangle_{L^2(\R\times M_1)}\dd s.
\end{multline*}
After integration by parts and recalling that $\chi$ is identically equal to $1$ on the support of $u$, this is equal to
\begin{equation}\label{eq:boundnonlinear6}
\begin{aligned}
 \int_0^t\left\langle r^{n-1}\chi(r) X_{\Phi(u(s))}\varphi_{t-s}^*\mathbf{P}\chi(r)\left(\partial_x^\alpha K(x,\pi_1(z_1))\right) ,\mathbf{P}(u(s))(r,z_1)\right\rangle_{L^2(\R\times M_1)}\dd s \\
  =\int_0^t\left\langle  T_\alpha(x) ,\mathbf{B}_{3, -2}^{(m_1,N_1)}\mathbf{P}(u(s))(r,z_1)\right\rangle_{L^2(\R\times M_1)}\dd s,
  \end{aligned}
\end{equation}
where we set
\begin{align*}
T_\alpha(x) &= \mathbf{B}_{-3, 2}^{(-m_1)}(X_1+C_1)^{N_1} X_{\Phi(u(s))}\varphi_{t-s}^*\mathbf{P}\left(\chi(r)r^{n-1}\left(\partial_x^\alpha K(x,\pi_1(z_1))\right)\right).\end{align*}
Thanks to Corollary~\ref{c:smoothNZ2} applied in positive time, there exist $N'\geqslant 1$ (depending only on the choice of the operator $\mathbf{B}$) and a constant $C_{K,k,r_0}>0$ such that, for all $|\alpha|\leqslant k$ and for all $x\in\Sigma$
$$
\|T_\alpha(x)\|_{L^2(\R\times M_1)} \leqslant C_{K,k,r_0} e^{-\frac{r_0\vartheta_0}{2} (t-s)}\|\Phi(u(s))\|_{\mathscr{C}^{N'}(\Sigma)}.
$$
As a result (up to modifying $C_{K,k,r_0}>0$), ~\eqref{eq:boundnonlinear6} can be controlled by
\begin{multline}\label{eq:boundnonlinear7}
 \left|\int_0^t\partial_x^\alpha\Phi\left(\varphi_{s-t}^*\{ \Phi(u(s)),\mathbf{P}(u(s))\}\right)\dd s\right|\\
 \leqslant C_{K,k,r_0}\int_0^t e^{-\frac{r_0\vartheta_0}{2} (t-s)}\|\Phi(u(s))\|_{\mathscr{C}^{N'}(\Sigma)}\|\mathbf{B}_{3, -2}^{(m_1,N_1)} \mathbf{P}(u(s))\|_{L^2(\R\times M_1)}\dd s.
\end{multline}
The proof of  Lemma~\ref{l:mainPhi} is complete
if we pick $k\geqslant N'$ in the definition of $\mathcal{J}_k(\varepsilon)$.

\subsection{Proof of Lemma~\ref{l:boundednorm}}
\label{sec:6.5}
In order to ease readability, we reduce Lemma~\ref{l:boundednorm} to the following two statements.

\begin{lemm}
\label{l:BtoA}
There exist $k_0 \in \mathbb{N}$, $C>0$ (independent of $\varepsilon$ and $u_0$) such that, for all $t \in\mathcal{J}_k(\varepsilon)$, 
$$
 \left\|\mathbf{B}_{3, -2}^{(m_1,N_1)} \mathbf{P}(u(t))\right\|_{L^2(\R\times M_1)}\leqslant C  \| u_0 \|_{\mathscr{C}^{k_0}} 
 + C\varepsilon e^{-2\vartheta t}\int_0^te^{\vartheta s}\left\|\mathbf{A}_{3,0}\chi u(s)\right\|_{L^2}ds.
$$
\end{lemm}

\begin{lemm}\label{l:polynomialbound}There exist a constant $\widetilde{c}_{K,\mathrm{g}}>0$ depending only on the metric $\mathrm{g}$ and on the interaction kernel $K$, $p \in \mathbb{N}$ and $C>0$ (independent of $\varepsilon$ and $u_0$) such that, for $\varepsilon<\widetilde{c}_{K,\mathrm{g}}r_0\vartheta$,  for every $s\in \mathcal{J}_k(\varepsilon)$,
 $$
 \|\mathbf{A}_{3,0} \chi u(s)\|_{L^2(\R\times M_1)}\leqslant C \langle s\rangle^p\|u_0\|_{L^2}.
 $$
 
\end{lemm}

\begin{rema} In Lemma~\ref{l:polynomialbound}, $\varepsilon$ must be small as follows. We set $C(r_0,\mu)>0$ to be a positive constant such that, for all $u\in L^1([r_0/4,4r_0]\times M_1)$ and for all $(r,z_1;\rho,\zeta_1)$ in $\mrm{T}^*([r_0/4,4r_0]\times M_1)$,
$$
\left|\{\zeta(X_{\Phi(u)}),\mu(z_1;\rho,\zeta_1)\}\right|\leqslant C(r_0,\mu)\|\Phi(u)\|_{\mathscr{C}^2}.
$$
Recall from Remark~\ref{r:potential-kernel} that $\zeta(X_{\Phi(u)})$ is in fact of the form
$\frac{1}{r}\zeta_1(X_{\Phi(u)}^1(z_1))+ \Phi_1(u)\rho,$
with $\Phi_1(u)$ and $X_{\Phi(u)}^1$ defined through linear combination of derivatives of $\Phi(u)$. We  shall fix $\varepsilon>0$ small enough to ensure that
\begin{equation}\label{eq:value-vareps}
 3C(r_0,\mu)\varepsilon \leqslant \frac{\vartheta}{8}.
\end{equation}
\end{rema}

Assuming Lemmas~\ref{l:BtoA} and \ref{l:polynomialbound}, we obtain Lemma~\ref{l:boundednorm}, as 
the possible polynomial growth of  $\|\mathbf{A}_{3,0} \chi u(s)\|_{L^2(\R\times M_1)}$ is compensated by the exponential factor $e^{-2\vartheta t}$. We provide the proofs of Lemmas~\ref{l:BtoA} and ~\ref{l:polynomialbound} in the next subsections.

\subsection{Proof of Lemma~\ref{l:BtoA}}


We start with the Duhamel formula to write
\begin{equation}\label{eq:duhamelanisotropic}
\mathbf{B}_{3, -2}^{(m_1,N_1)} \mathbf{P}(u(t))
=\mathbf{B}_{3, -2}^{(m_1,N_1)}\mathbf{P}(\varphi_{-t}^*(u_0))
-\int_0^t\mathbf{B}_{3, -2}^{(m_1,N_1)} \varphi_{s-t}^*\mathbf{P}\{\Phi(u(s)),u(s)\}\dd s.
\end{equation}

Thanks to Lemma~\ref{l:NZradial}, there exists a constant $C>0$ and an integer $N'\geqslant 1$ 
such that, for all $t\geqslant 0$,
$$
\|\mathbf{B}_{3, -2}^{(m_1,N_1)} \mathbf{P}(\varphi_{-t}^*(u_0))\|_{L^2(\R\times M_1)}\leqslant C e^{-\vartheta t}\|u_0\|_{\mathscr{C}^{N'}}.
$$
Hence, in order to prove Lemma~\ref{l:BtoA}, we just have to obtain the expected upper bound on the integral remainder in the right-hand side of~\eqref{eq:duhamelanisotropic}, which we denote by
$$
R(t)=\int_0^t\mathbf{B}_{3, -2}^{(m_1,N_1)}\varphi_{s-t}^*\mathbf{P}\{\Phi(u(s)),u(s)\}\dd s.
$$
Since $\supp u(t)\subset [r_0/2,2r_0]$ by  Lemma~\ref{l:support}, this is also equal to
$$
R(t)=\int_0^t\mathbf{B}_{3, -2}^{(m_1,N_1)}\varphi_{s-t}^* \chi\mathbf{P} \{\Phi(u(s)),u(s)\}\dd s,
$$
where $\chi\in\mathscr{C}^\infty_c(\R_+^*)$ is identically equal to $1$ on $[r_0/2,2r_0]$.

We are now in position to apply the bilinear estimate of Proposition~\ref{p:bilinear}. Indeed, since $s \in  \mathcal{J}_k(\varepsilon)$ and up to taking $k\geqslant N$ (with $N$ the regularity parameter appearing in this proposition), one has
\begin{align*}
\|{R}(t)\|_{L^2}&\leqslant C \int_0^t e^{-\frac{\vartheta_0 r_0}{2}(t-s)} \| \Phi(u(s))\|_{\mscr C^k} \left( \| u(s) \|_{L^2( \R \times M_1)} + \left\|\mathbf{A}_{3,0}\chi u(s)\right\|_{L^2} \right) \dd s\\
&\leqslant C \varepsilon e^{-2\vartheta t}\int_0^t e^{\vartheta s}\left( \| u(s) \|_{L^2( \R \times M_1)} + \left\|\mathbf{A}_{3,0}\chi u(s)\right\|_{L^2} \right) \dd s.
\end{align*}
By Lemma~\ref{lem:conservation-Lp}, one has\footnote{We only have an upper bound as we took the measure $\dd r \dd {\mrm L}_1$ for the $L^2$ scalar product.} $\left\| u(s)\right\|_{L^2}\leqslant C_{r_0}\|u_0\|_{L^2}$ for every $s\in\mathcal{J}_k(\varepsilon)$. Hence, we can ensure that the contribution of this term to the upper bound for $ \|\mathbf{B}_{3,-2}^{(m_1,N_1)}\mathbf{P}(u(t))\|_{L^2(\R\times M_1)}$ is controlled by $C  e^{-\vartheta t}\varepsilon\|u_0\|_{L^2}$ for every $t\in\mathcal{J}_k(\varepsilon)$.
This concludes the proof of Lemma~\ref{l:BtoA} provided that we pick $k_0\geqslant N'$. 

\subsection{Proof of Lemma~\ref{l:polynomialbound}}
\label{ss:proof-microlocal}

In this section, we use microlocal energy estimates to prove Lemma~\ref{l:polynomialbound}. This is precisely here that our construction of a global escape function is crucial.

Due to a logarithmic loss of derivatives that will occur in the energy estimates, we will not study directly $\|\mathbf{A}_{3,0} \chi u(s)\|_{L^2(\R\times M_1)}$ but rather
$$
\|\mathbf{A}_{3,\sigma_0(s)}  \chi u(s)\|_{L^2(\R\times M_1)}
$$
where $\sigma_0(s)=\frac{1}{8} e^{-\vartheta s} $ is a time-dependent weight. This will yield some extra dissipation in the energy estimates that will precisely compensate the aforementioned logarithmic loss of derivatives. 
This idea of considering a time-dependent regularity is somehow reminiscent of the time-dependent weights often used for Cauchy-Kowalevskaya type theorems in analytic regularity (see e.g. \cite{Caf} and the proofs of Landau damping in the torus of \cite{BMM} and \cite{GNR21}). The fact that $\sigma_0$ remains non-negative for all times is important and, as we shall see at the end of the proof, estimating $
\|\mathbf{A}_{3,\sigma_0(s)}  \chi u(s)\|_{L^2(\R\times M_1)}
$ will be sufficient to bound $\|\mathbf{A}_{3,0} \chi u(s)\|_{L^2(\R\times M_1)}$.

Recall that $\mathbf{A}_{3,\sigma_0(t)}$ is of the form
$$
\mathbf{A}_{3,\sigma_0(t)}v(r,z_1)=\frac{1}{2\pi}\int_{\R^2}e^{i(r-s)\rho}{\oM}\left(e^{(3\mu+\sigma_0(t)) \log f}\right)v(s,z_1)\dd\rho  \dd s,
$$
where $\mu$ and $f$ are given by Proposition~\ref{prop:comparisonresult}. Let $\chi$ be a smooth function that is compactly supported in $]r_0/4,4r_0[$ and that is identically equal to $1$ on $[r_0/2,2r_0]$. Thanks to Lemma~\ref{l:cutoff}, there exists a constant $C>0$  such that
$$
\left\|\mathbf{A}_{3,\sigma_0(t)} u(t)\right\|_{L^2}\leqslant \left\|\chi \mathbf{A}_{3,\sigma_0(t)}\chi u(t)\right\|_{L^2}+ C\|u_0\|_{L^2}.
$$
Hence, we can restrict our attention to estimating the norm of $\left\|\chi \mathbf{A}_{3,\sigma_0(t)}\chi u(t)\right\|_{L^2}$. To this purpose, we write
\begin{equation}
\label{eq:derivativeanisotropicnorm}
\begin{aligned}
\frac{\dd}{\dd t}\left\|\chi \mathbf{A}_{3,\sigma_0(t)}\chi u(t)\right\|_{L^2}^2 
=-2\operatorname{Re}&\left\langle \chi \mathbf{A}_{3,\sigma_0(t)}\chi Xu(t),\chi \mathbf{A}_{3,\sigma_0(t)}\chi u(t)\right\rangle\\
&-2\operatorname{Re}\left\langle \chi \mathbf{A}_{3,\sigma_0(t)}\chi X_{\Phi(u(t))} u(t),\chi \mathbf{A}_{3,\sigma_0(t)}\chi u(t)\right\rangle \\
&-\frac{1}{4} \vartheta e^{-t\vartheta} \operatorname{Re}\left\langle \chi \widetilde{\mathbf{A}}_{3,\sigma_0(t)}\chi u(t),\chi \mathbf{A}_{3,\sigma_0(t)}\chi u(t)\right\rangle,
\end{aligned}
\end{equation}
where
$$
\widetilde{\mathbf{A}}_{3,\sigma_0(t)}v(r,z_1)=\frac{1}{2\pi}\int_{\R^2}e^{i(r-s)\rho}{\oM}\left(\log f e^{(3\mu+\sigma_0(t)) \log f}\right)v(s,z_1)\dd\rho  \dd s.
$$
Now using the composition rule for pseudo-differential operators on $\R\times M_1$ (see Remark~\ref{r:pseudononcompact} for a brief reminder), one has
$$
\widetilde{\mathbf{A}}_{3,\sigma_0(t)}\chi^2 =\mathbf{Op}\left(\log f \right)\chi\mathbf{Op}\left( e^{(3\mu+\sigma_0(t)) \log f}\right)\chi\\
+\mathbf{Op}\left(b_{-1} e^{(3\mu+\sigma_0(t)) \log f}\right)+\mathbf{R}_N,
$$
where $b_{-1}$ is an element in $\overline{S}^{-1+0}(\mrm{T}^*(\R\times M_1))$ and the remainder $\mathbf{R}_N$ is a bounded operator from $H^{-N}_{\text{comp}}(\R\times M_1)$ to $H^N_{\text{loc}}(\R\times M_1)$ (with the semi-norm of $R_N$ that is uniformly bounded in terms of $t$). Applying the composition rule for pseudo-differential operators, one finds that the last term in the right-hand side of~\eqref{eq:derivativeanisotropicnorm} satisfies
\begin{equation}\label{eq:derivativevariableorder}
\begin{aligned}
\left\langle \chi \widetilde{\mathbf{A}}_{3,\sigma_0(t)}\chi u(t),\chi \mathbf{A}_{3,\sigma_0(t)}\chi u(t)\right\rangle = \left\langle \chi\mathbf{Op}(\log f)\chi \mathbf{A}_{3,\sigma_0(t)}\chi u(t),\chi \mathbf{A}_{3,\sigma_0(t)}\chi u(t)\right\rangle\\
+\mathcal{O}\left(\|\mathbf{A}_{3,\sigma_0(t)-\frac{1}{4}}u(t)\|_{L^2}^2+\|u_0\|_{L^2}^2\right),
\end{aligned}
\end{equation}
where the constant in the remainder depends on $\vartheta$, and on the choice of $f$ and $\mu$ (but not on $t$). Note that we used one more time that $\|u(t)\|_{L^2}\leqslant 2^nr_0^{-n+1} \|u_0\|_{L^2}$ for every $t\in\mathcal{J}_k(\varepsilon)$ as the $L^2$-scalar product is taken with respect to $\dd r \dd {\mrm L}_1(z_1)$ rather than $r^{n-1}\dd r \dd {\mrm L}_1(z_1)$ .

Next we deal with the first term in~\eqref{eq:derivativeanisotropicnorm}. As $X$ preserves the measure $\dd r \dd {\mrm L}_1$ on $\R\times M_1$, one has
$$
\left\langle \chi \mathbf{A}_{3,\sigma_0(t)}\chi Xu(t),\chi \mathbf{A}_{3,\sigma_0(t)}\chi u(t)\right\rangle=\left\langle \left[\chi \mathbf{A}_{3,\sigma_0(t)}\chi, X\right]u(t),\chi \mathbf{A}_{3,\sigma_0(t)}\chi u(t)\right\rangle
$$
We can now insert the almost inverse $\mathbf{W}_{3,\sigma_0(t)}^{(N)}= \mathbf{Op}\left(e^{-(3 \mu + \sigma_0(t)\log f)}b_{N}\right)$ of $\mathbf{A}_{3,\sigma_0(t)}$, as given by~\eqref{eq:pseudoinverseA},  with $N\geqslant 5$. Applying the composition rule together with Calder\'on-Vaillancourt one more time, it yields 
\begin{equation}\label{eq:bracketX}
\begin{aligned}
&\left\langle \chi \mathbf{A}_{3,\sigma_0(t)}\chi Xu(t),\chi \mathbf{A}_{3,\sigma_0(t)}\chi u(t)\right\rangle\\
&\hspace{0.62cm}= \left\langle \left[\chi \mathbf{A}_{3,\sigma_0(t)}\chi, X\right]{\mathbf{W}}_{3,\sigma_0(t)}^{(N)} \mathbf{A}_{3,\sigma_0(t)}u(t),\chi \mathbf{A}_{3,\sigma_0(t)}\chi u(t)\right\rangle\\
&\hspace{7cm}+\mathcal{O}\left(\|\mathbf{A}_{3,\sigma_0(t)-\frac{1}{4}}u(t)\|_{L^2}^2+\|u_0\|_{L^2}^2\right),
\end{aligned}
\end{equation}
where the constant in the remainder depends again on $\vartheta$ and the choice of $f$ and $\mu$ (but not on $t$). Using Lemmas~\ref{l:support} and~\ref{l:cutoff}, we can (up to increasing the size of the constant in the remainder) insert a cutoff $\chi$ so that~\eqref{eq:bracketX} becomes
\begin{multline*}
\left\langle \left(\left[\chi \mathbf{A}_{3,\sigma_0(t)}\chi, X\right]\mathbf{W}_{3,\sigma_0(t)}^{(N)}\chi\right)\chi \mathbf{A}_{3,\sigma_0(t)}\chi u(t),\chi \mathbf{A}_{3,\sigma_0(t)}\chi u(t)\right\rangle\\
+\mathcal{O}\left(\|\mathbf{A}_{3,\sigma_0(t)-\frac{1}{4}}u(t)\|_{L^2}^2+\|u_0\|_{L^2}^2\right).
\end{multline*}
Another application of the composition rule together with the Calder\'on-Vaillancourt Theorem allows us to write the above term as
\begin{multline*}
\left\langle \mathbf{Op}\left(\chi^3{\mathfrak{X}}((3\mu+\sigma_0(t))\log f)\right)\chi\mathbf{A}_{3,\sigma_0(t)}\chi u(t),\chi \mathbf{A}_{3,\sigma_0(t)}\chi u(t)\right\rangle\\
+\mathcal{O}\left(\|\mathbf{A}_{3,\sigma_0(t)-\frac{1}{4}}u(t)\|_{L^2}^2+\|u_0\|_{L^2}^2\right).
\end{multline*}
By Proposition~\ref{prop:comparisonresult}, we know that $\chi(r)\mathfrak{X}((3\mu+\sigma_0(t))\log f)(z,\zeta)\geqslant 0$ for $\langle\zeta\rangle \geqslant R_0$. Hence, up to adding a pseudo-differential that is compactly supported in the variable $\zeta$ (and thus smoothing), we are in position to apply the sharp G\aa{}rding inequality~\cite[Prop.~E.23]{dyatlov2017mathematical} (see 
Section~\ref{s:garding} and Remark~\ref{r:pseudononcompact} for a brief reminder). It allows to end up with
\begin{equation}\label{eq:bracketX4}
-2\operatorname{Re}\left\langle \chi \mathbf{A}_{3,\sigma_0(t)}\chi Xu(t),\chi \mathbf{A}_{3,\sigma_0(t)}\chi u(t)\right\rangle\leqslant
C\left(\|\mathbf{A}_{3,\sigma_0(t)-\frac{1}{4}}u(t)\|_{L^2}^2+\|u_0\|_{L^2}^2\right),
\end{equation}
where $C>0$ is a constant that is independent of $t$.

It remains to deal with the second term in the right-hand side of~\eqref{eq:derivativeanisotropicnorm}. Using that $X_{\Phi(u)}$ is an Hamiltonian vector field (and thus preserves the Liouville measure), we can write
\begin{multline*}
 \left\langle \chi \mathbf{A}_{3,\sigma_0(t)}\chi X_{\Phi(u(t))} u(t),\chi \mathbf{A}_{3,\sigma_0(t)}\chi u(t)\right\rangle\\
 =\left\langle \left[\chi \mathbf{A}_{3,\sigma_0(t)}\chi, X_{\Phi(u(t))}\right] u(t),\chi \mathbf{A}_{3,\sigma_0(t)}\chi u(t)\right\rangle\\
 +\left\langle X_{\Phi(u(t))}(r^{-n+1})\chi \mathbf{A}_{3,\sigma_0(t)}\chi  u(t),\chi \mathbf{A}_{3,\sigma_0(t)}\chi u(t)\right\rangle.
\end{multline*}
As $t\in\mathcal{J}_k(\varepsilon)$, we find
\begin{multline*}
 -2\operatorname{Re}\left\langle \chi \mathbf{A}_{3,\sigma_0(t)}\chi X_{\Phi(u(t))} u(t),\chi \mathbf{A}_{3,\sigma_0(t)}\chi u(t)\right\rangle\\
 \leqslant -2\operatorname{Re}\left\langle \left[\chi \mathbf{A}_{3,\sigma_0(t)}\chi, X_{\Phi(u(t))}\right] u(t),\chi \mathbf{A}_{3,\sigma_0(t)}\chi u(t)\right\rangle\\
 +n (2r_0)^{-n} \varepsilon e^{-\vartheta t}\left\|\chi \mathbf{A}_{3,\sigma_0(t)}\chi u(t)\right\|^2.
\end{multline*}
Proceeding as for the case of the vector field $X$ (i.e. through several applications of the composition rule together with the Calder\'on-Vaillancourt Theorem), we find that there exists a constant $C>0$ (depending on $\vartheta$, $r_0$ and on the choices of $f$ and $\mu$) such that
\begin{multline*}
 -2\operatorname{Re}\left\langle \chi \mathbf{A}_{3,\sigma_0(t)}\chi X_{\Phi(u(t))} u(t),\chi \mathbf{A}_{3,\sigma_0(t)}\chi u(t)\right\rangle\\
 \leqslant -2\operatorname{Re}\left\langle \mathbf{Op}\left(3\log f\{\zeta(X_{\Phi(u(t))}),\mu\}\right)\chi \mathbf{A}_{3,\sigma_0(t)}\chi u(t),\chi \mathbf{A}_{3,\sigma_0(t)}\chi u(t)\right\rangle\\
 +C \varepsilon e^{-\vartheta t}\left\|\chi \mathbf{A}_{3,\sigma_0(t)}\chi u(t)\right\|^2 +C\left(\|\mathbf{A}_{3,\sigma_0(t)-\frac{1}{4}}u(t)\|_{L^2}^2+\|u_0\|_{L^2}^2\right) .
\end{multline*}
Combining this inequality with~\eqref{eq:derivativeanisotropicnorm},~\eqref{eq:derivativevariableorder} and~\eqref{eq:bracketX4}, we are led to study the term
$$
-2\left\langle \chi\mathbf{Op}\left(\left[ \frac{1}{8} \vartheta e^{-t\vartheta}  + 3 \{\zeta(X_{\Phi(u(t))}) , \mu \} \right]\log f \right)\chi \mathbf{A}_{3,\sigma_0(t)}\chi u(t),\chi \mathbf{A}_{3,\sigma_0(t)}\chi u(t)\right\rangle.
$$
Thanks to~\eqref{eq:value-vareps}, for $\varepsilon$ small enough, one finds that, for all $t \in \mathcal J_k(\varepsilon)$,
$$
\frac{1}{8}  \vartheta e^{-t\vartheta}  + 3 \{\zeta(X_{\Phi(u(t))}) , \mu \} \geqslant 0.
$$
We can thus apply the sharp G\aa{}rding inequality one more time (together with the composition rule and the Calder\'on-Vaillancourt Theorem). This eventually yields the  upper bound
\begin{align*}
 \frac{\dd}{\dd t}\left\|\chi \mathbf{A}_{3,\sigma_0(t)}\chi u(t)\right\|_{L^2}^2
 &\leqslant C \varepsilon e^{-\vartheta t}\left\|\chi \mathbf{A}_{3,\sigma_0(t)}\chi u(t)\right\|^2 +C\left(\|\mathbf{A}_{3,\sigma_0(t)-\frac{1}{4}}u(t)\|_{L^2}^2+\|u_0\|_{L^2}^2\right)\\
 &\leqslant C \varepsilon e^{-\vartheta t}\left\|\chi \mathbf{A}_{3,\sigma_0(t)}\chi u(t)\right\|^2 +C\left(\|\mathbf{A}_{3,\sigma_0(t)-\frac{1}{8}}u(t)\|_{L^2}^2+\|u_0\|_{L^2}^2\right),
\end{align*}
which implies
\begin{equation}\label{eq:maininequalitybracket}
 \frac{\dd}{\dd t}\left(\left\|\chi \mathbf{A}_{3,\sigma_0(t)}\chi u(t)\right\|_{L^2}^2e^{-C\varepsilon\int_0^te^{-\vartheta s} \dd s}\right)\leqslant C\left(\|\mathbf{A}_{3,\sigma_0(t)-\frac{1}{8}}u(t)\|_{L^2}^2+\|u_0\|_{L^2}^2\right),
\end{equation}
where the constant $C>0$ depends on $r_0$, $\vartheta$ and the choice of the function $f$ and $\mu$. 

Iterating this procedure $p$ times (where $p$ is such that $0<3+\frac{1}{8}-\frac{p}{8}\leqslant\frac{1}{4}$), recalling that $3(1-2\delta) > 3-1/8$, so that Proposition~\ref{prop:comparisonresult} can be indeed applied at each step,
we find, according to~\eqref{1/8}, that for every $0\leqslant q\leqslant p$ and for every $t\in\mathcal{J}_k(\varepsilon)$,
\begin{align}
\label{1/4}
\frac{\dd}{\dd t}\left(\left\|\chi \mathbf{A}_{3,\sigma_0(t)-\frac{q}{8}}\chi u(t)\right\|_{L^2}^2e^{-C\varepsilon\int_0^te^{-\vartheta s} \dd s} \right)
&\leqslant C\left(\|\mathbf{A}_{3,\sigma_0(t)-\frac{q}{8}-\frac{1}{4}}u(t)\|_{L^2}^2+\|u_0\|_{L^2}^2\right)\\
\label{1/8}
&\leqslant \widetilde{C}\left(\|\mathbf{A}_{3,\sigma_0(t)-\frac{q+1}{8}}u(t)\|_{L^2}^2+\|u_0\|_{L^2}^2\right)
\end{align}
By definition of $p$, using the Calder\'on-Vaillancourt (together with the composition rule) one more time, there holds
$$
\|\mathbf{A}_{3,\sigma_0(t)-\frac{p}{8}-\frac{1}{4}} \chi u(t)\|_{L^2} \leqslant C \| u(t) \|_{L^2} \leqslant (r_0/2)^{1-n} \| u_0 \|_{L^2}.
$$
By induction, using~\eqref{1/4} for $q=p$ and~\eqref{1/8} for $q <p$, we deduce that for every $t\in\mathcal{J}_k(\varepsilon)$,
$$
\|\mathbf{A}_{3,\sigma_0(t)} \chi u(t)\|_{L^2} \leqslant C \langle t \rangle^p \| u_0 \|_{L^2},
$$
and by the composition rule and the Calder\'on-Vaillancourt Theorem, we conclude that, for every $t\in\mathcal{J}_k(\varepsilon)$,
$$
\|\mathbf{A}_{3,0} \chi u(t)\|_{L^2} \leqslant C \langle t \rangle^p \| u_0 \|_{L^2},
$$
where the constant depends only $\vartheta$ and the choice of the escape function $F$ used to define $\mathbf{A}$.
We have therefore proved the final Lemma~\ref{l:polynomialbound}.


\subsection{End of the proof}
\label{sec:weak}
We are now ready to complete the proof of Theorem~\ref{thm:main}. Let $\psi$ be a smooth test function. We have for all $t \in \R_+$,
\begin{equation}
\label{eq:uPsi}
\langle u(t),\psi\rangle = \langle\varphi_{-t}^*u_0,\psi\rangle + \int_0^t \langle \varphi_{-(t-s)}^*X_{\Phi(u(s))} u(s),\psi\rangle \dd s.
\end{equation}
By Lemma~\ref{l:support},  $u(t)$ is supported away from the null section, and since the geodesic flow preserves this property, we can assume without loss of generality that this is also the case for $\psi$. By Corollary~\ref{c:smoothcorrelation}, the first term in the right-hand side of~\eqref{eq:uPsi} converges exponentially fast to 
$$
\int_{M^\times}\left(\int_{M_1} \psi(r,z_1')\dd {\mrm L}_1(z_1')\right) u_0(r,z_1)\dd {\mrm L}(r,z_1) = \left\langle \left(\int_{M_1} u_0(r,z_1')\dd {\mrm L}_1(z_1')\right), \psi\right\rangle.
$$ 
For what concerns the second one, we write
\begin{multline}\label{eq:splitting-duhamel}
 \int_0^t \langle \varphi_{(t-s)}^*X_{\Phi(u(s))} u(s),\psi\rangle \dd s =  \int_0^t \langle X_{\Phi(u(s))} u(s),  \mathbf{P} \psi \circ \varphi_{-(t-s)} \rangle \dd s\\
 +   \int_0^t \langle X_{\Phi(u(s))} u(s),  \mathbf{Q} \psi\rangle \dd s.
 \end{multline}
 In order to analyze the two terms above, recall from Remark~\ref{r:potential-kernel} that $X_{\Phi(u)}$ is of the form
$\frac{1}{r}X_{\Phi(u)}^1(z_1)+ \Phi_1(u)\partial_r,$
with $\Phi_1(u)$ and $X_{\Phi(u)}^1$ defined through the smooth kernels given in~\eqref{eq:radialcomponent-vf-potential} and~\eqref{eq:tangentialcomponent-vf-potential}. From~\eqref{eq:semi-normsymbolpotentialvf} and~\eqref{eq:semi-normsymbolpotentialvf2}, we also have that their semi-norms can be controlled by $\|\Phi(u(t))\|_{\mathscr{C}^k}$. 

With that conventions at hand, we can rewrite~\eqref{eq:splitting-duhamel} as
\begin{multline*}
 \int_0^t \langle  X_{\Phi(u(s))} u(s),  \varphi_{(t-s)}^* \mathbf{P} \psi\rangle \dd s -   \int_0^t \left\langle \Phi_1(u(s)) \mathbf{P}u(s),  \left(\partial_r \mathbf{Q} \psi+\frac{n-1}{r}\mathbf{Q} \psi\right)\right\rangle \dd s\\
-   \int_0^t \left\langle \Phi_1(u(s)) \mathbf{Q}u(s),  \left(\partial_r \mathbf{Q} \psi+\frac{n-1}{r}\mathbf{Q} \psi\right)\right\rangle \dd s.
 \end{multline*}
For the fist term, we can argue as in the proof of Lemma~\ref{l:mainPhi} and use Theorem~\ref{theo:connected} and Lemma~\ref{l:boundednorm}
to show that it converges exponentially fast to $0$ as $t\to \infty$. For the second, we use Lemma~\ref{l:boundednorm} one more time together with the fact that 
$$
\left\|\mathbf{B}_{3,-2}^{(m_1,N_1)} \Phi_1(u(s))(\mathbf{B}_{3,-2}^{(m_1,N_1)})^{-1}\right\|_{L^2\rightarrow L^2}\leqslant C\|\Phi(u(s))\|_{\mathscr{C}^k}.
$$
This last property follows from the composition rule for pseudo-differential on $M_1$ together with the Calder\'on-Vaillancourt (up to increasing the value of $k$). Thanks to Theorem~\ref{theo:connected}, this shows that the second term converges exponentially fast to
$$
-   \left\langle \int_0^\infty\Phi_1(u(s)) \mathbf{P}u(s) \dd s,  \left(\partial_r \mathbf{Q} \psi+\frac{n-1}{r}\mathbf{Q} \psi\right)\right\rangle. 
$$
The third term also converges converges exponentially to some limit thanks to Theorem~\ref{theo:connected} together with the fact that the $L^p$ norms of $u(t)$ are preserved. In summary, it yields that the Duhamel remainder term converges exponentially fast to 
$$
\left\langle \underbrace{ \int_0^{+\infty}  X_{\Phi(u(s))} u(s)\dd s}_{= w_\infty}, \int_{M_1} \psi (r,z_1')\dd {\mrm L}_1(z_1') \right\rangle =  \left\langle \left(\int_{M_1} w_\infty(r,z_1')\dd {\mrm L}_1(z_1')\right), \psi\right\rangle.
$$
All in all, we deduce that $u(t)$ weakly converges exponentially fast to the distribution 
$$u_\infty = \int_{M_1} u_0(r,z_1')\dd {\mrm L}_1(z_1') + \int_{M_1} w_\infty(r,z_1')\dd {\mrm L}_1(z_1').$$ 
By Lemma~\ref{lem:conservation-Lp}, we deduce that $u_\infty \in L^1 \cap L^\infty(M)$ and that the support of $u_\infty$ avoids the null section. Hence we can write it as $h_\infty \circ \mrm H$ for some $h_\infty \in  L^\infty_{\operatorname{comp}}(\R_{>0})$.

To conclude, the exponential decay estimate of $\|\Phi(u(t))\|_{\mathscr{C}^{N}}$ for all $N\geqslant 0$ follows from the previous calculation by replacing $\psi$ by $\partial_x^\alpha K(x, \cdot)$ which has zero average.

\appendix

\section{Toolbox on microlocal analysis}
\label{a:pseudo}

In this appendix, we review material from microlocal analysis with a special emphasis on pseudo-differential operators on manifolds. These tools are used all along the article and we refer to~\cite[Ch.~18]{HormanderIII},~\cite[Ch.~4, 9, 14]{zworski} and~\cite[App.~E]{dyatlov2017mathematical} -- that we closely follow -- for more details and background on this theory. A notable difference with these references is the use of an isochore atlas and of the Weyl quantization which simplify some aspects of the exposition as it was pointed to us by Guedes Bonthonneau~\cite{bonthonneaubreviary}.

\subsection{Pseudodifferential operators on $\R^N$}\label{aa:pseudoRN} Let $a$ be a function in the Schwartz class $\mathscr{S}(\R^{2N})$. One can define the Weyl quantization of the symbol $a$ as
$$
\forall u\in\mathscr{S}(\R^N),\quad{\oW}(a)u(x)=\frac{1}{(2\pi)^N}\int_{\R^{2N}}e^{i\langle x-y,\xi\rangle}a\left(\frac{x+y}{2},\xi\right)u(y)dyd\xi.
$$
One can verify that ${\oW}(a)u$ belongs to $\mathscr{S}(\R^{2N})$. Let us observe that the Schwartz kernel of ${\oW}(a)$ is given by
$$
K_a(x,y)=\frac{1}{(2\pi)^N}\int_{\R^{N}}e^{i\langle x-y,\xi\rangle}a\left(\frac{x+y}{2},\xi\right)d\xi.
$$
One can also recover the symbol from the kernel through the formula
\begin{equation}\label{eq:symbol-kernel} 
 a(x,\xi)=\int_{\R^N}K_a\left(x+\frac{\eta}{2},x-\frac{\eta}{2}\right)e^{-i\langle \eta,\xi\rangle}d\eta,
\end{equation}
which is often referred to as the Wigner transform.

We can also introduce the class of Kohn-Nirenberg (or classical) symbols~\cite[\S 9.3]{zworski} $a$, for every $m\in\R$,
$$
S^m(\mrm{T}^*(\R^N))=\left\{a\in\mathscr{C}^\infty(\mrm{T}^*\R^N):\ \forall(\alpha,\beta),\ \sup_{x,\xi} \, \langle\xi\rangle^{|\beta|-m}|\partial^\alpha_x\partial_\xi^\beta a|<\infty\right\},
$$
where we used the standard convention $\langle\xi\rangle=\left(1+|\xi|^2\right)^{\frac12}$. One can verify that the definition of $\oW(a)$ extends to symbols lying in $S^m(\mrm{T}^*(\R^N))$ and that, for every $m\in\R$ and every $a\in S^m(\mrm{T}^*(\R^N))$,
$$
{\oW}(a):\mathscr{S}(\R^N)\rightarrow\mathscr{S}(\R^N)
$$
is a bounded operator~\cite[Th.~4.16]{zworski}. We denote by $\Psi^m(\R^n)$ the set of operators of the form ${\oW}(a)$ with $a$ belonging to $S^m(\mrm{T}^*(\R^N))$. This is what we will refer to the set of pseudo-differential operators of order $m$ on $\R^N$. The useful extra classes of operators are
$$
\Psi^{-\infty}(\R^N)=\bigcap_{m\in\R}\Psi^m(\R^N),\qquad \Psi^{m+}(\R^N)=\bigcap_{m'>m}\Psi^{m'}(\R^N)
$$
and their subset $\Psi_{\text{comp}}(\R^N)$ consisting of pseudo-differential operators whose symbol $a$ is compactly supported. Recall from~\cite[Th.~9.6]{zworski} that, if $a\in S^m(\mrm{T}^*(\R^N))$, then the Schwartz kernel $K_a$ of ${\oW}(a)$ belongs to $\mathscr{S}'(\R^{2N})$ and is smooth outside the diagonal 
$$
\Delta = \left\{(x,x)~:~x \in\R^N\right\}.
$$
Moreover, for every $N_0>|\alpha|+|\beta|+m+N$, one has 
$$
|\partial_x^\alpha\partial_y^\beta K_a(x,y)|\leqslant C_{N_0}|x-y|^{-N_0}, \quad (x,y) \in \R^{2N}.
$$
Conversely, if the Schwartz kernel $K \in\mathscr{C}^\infty(\R^{2N})$ of a smoothing operator $A$~\cite[Th.5.2.6]{hormander2015analysis} verifies, for all $\alpha,\beta$ and for all $N_0$, $\langle x-y\rangle ^{N_0}\partial_x^\alpha\partial_y^\beta K(x,y) \in L^{\infty}(\R^{2N})$ and one can show that $A={\oW}(a)$ for some $a\in S^{-\infty}(\mrm{T}^*\R^N)$~\cite[Th.~9.6]{zworski}.

\begin{rema}\label{r:invariantsymbol}
The key observation in view of defining pseudo-differential operators on manifolds is that, given a smooth diffeomorphism $\gamma:\R^N\rightarrow \R^N$ (with all the derivatives of $\gamma$ and $\gamma^{-1}$ being bounded on $\R^N$), one can define its symplectic lift
$$
\widetilde{\gamma}:\mrm{T}^*\R^N\rightarrow \mrm{T}^*\R^N,\quad (x,\xi)\mapsto \left(\gamma^{-1}(x),\dd\gamma(x)^\top\xi\right),
$$
and verify that $a\in S^m$ implies that $a\circ\widetilde{\gamma}\in S^m$~\cite[Th.~9.4]{zworski}.
\end{rema}
Let us now review a few standard properties of these operators.

\subsubsection{Adjoint} One can check the following, for every $a\in S^m(\mrm{T}^*(\R^N))$,
\begin{equation}\label{eq:adjointRN}
 \langle {\oW}(a)u,v\rangle_{L^2}=\langle u,{\oW}(\overline{a})v\rangle_{L^2}, \quad u, v\in\mathscr{S}(\R^N).
\end{equation}
In particular, if $a$ is real-valued, the operator is formally selfadjoint.

\subsubsection{Composition formula} Given $a\in S^{m_1}(\mrm{T}^*(\R^N))$ and $b\in S^{m_2}(\mrm{T}^*(\R^N))$, one has
\begin{equation}\label{eq:compositionRN}
 {\oW}(a)\circ{\oW}(b)={\oW}(a \operatorname{\sharp} b),
\end{equation}
where $a \operatorname{\sharp} b$ is the Moyal product of $a$ and $b$. It can be shown~\cite[Th.~9.5]{zworski} that $a \operatorname{\sharp} b$ belongs to $S^{m_1+m_2}(\mrm{T}^*(\R^N))$ and that, for every $N_0\geqslant 0$,
\begin{equation}\label{eq:moyalRN}
 (a \operatorname{\sharp} b)(x,\xi)-\sum_{k=0}^{N_0}\frac{\left(\langle\partial_\xi,\partial_y\rangle-\langle\partial_\eta,\partial_x\rangle\right)^k}{(2i)^kk!}\left(a(x,\xi)b(y,\eta)\right)|_{x=y,\xi=\eta} \in S^{m_1+m_2-N_0-1}(\mrm{T}^*(\R^N)) ,
\end{equation}
where one can verify that each term in the sum belongs to $S^{m_1+m_2-k}(\mrm{T}^*(\R^N))$. In particular, one can observe that the first term in the sum is equal to $ab$ while the second one is $\frac{1}{2i}\{a,b\}$. Going through the proof of this result, one can verify that the semi-norms of the remainder of order $N_0$ are controlled by the semi-norms of $\partial_x^{\alpha_1}\partial_\xi^{\beta_1}a \, \partial_x^{\alpha_2}\partial_\xi^{\beta_2}b$ with $|\alpha_1|+|\alpha_2|=|\beta_1|+|\beta_2|=N_0$.

These properties immediately lead to the formula for the bracket
\begin{equation}\label{eq:bracketRN}
 [{\oW}(a),{\oW}(b)]={\oW}(a \operatorname{\sharp} b-b\sharp a),
\end{equation}
and one can verify that the symbol $c=a \operatorname{\sharp} b-b\sharp a$ has the following asymptotic form:
\begin{equation}\label{eq:bracketsymbolRN}
 c(x,\xi)-2\sum_{k=0}^{N_0}\frac{\left(\langle\partial_\xi,\partial_y\rangle-\langle\partial_\eta,\partial_x\rangle\right)^{2k+1}}{(2i)^{2k+1}(2k+1)!}\left(a(x,\xi)b(y,\eta)\right)|_{x=y,\xi=\eta} \in S^{m_1+m_2-2N_0-3}(\mrm{T}^*(\R^N)).
\end{equation}
Again, one can check that each term in the sum belongs to $S^{m_1+m_2-2k-1}(\mrm{T}^*(\R^N))$.

\subsubsection{Action on Sobolev spaces} In order to make these formulas of interest, one needs to understand the action of these operators on standard Sobolev spaces
$$
H^s(\R^N)=\left\{ u\in\mathscr{S}'(\R^N): \langle\xi\rangle^s\widehat{u}(\xi) \in L^2(\R^N)\right\},\quad s\in\R,
$$
where $\widehat{u}(\xi)$ is the Fourier transform of $u$. The Calder\'on-Vaillancourt Theorem reads as follows\footnote{Technically speaking, this section deals with the $s=m=0$ but, using the composition rule, one can derive this more general version.}~\cite[\S 4.5]{zworski}. For every $s\in\R$, there exists $C_s>0$ and $N_s\in\Z_+$ such that, for every $a\in S^m(\mrm{T}^*\R^N)$,
\begin{equation}\label{eq:calderonRN}
\left\|{\oW}(a)\right\|_{H^{s}\rightarrow H^{s-m}}\leqslant C_s\sum_{|\alpha|+|\beta|\leqslant N_s}\left\|\langle\xi\rangle^{-m+|\beta|}\partial^\alpha_x\partial_\xi^\beta a\right\|_{\infty}. 
\end{equation}
Combined with~\eqref{eq:compositionRN}, this yields asymptotic formulas for ${\oW}(a)\circ{\oW}(b)$ with terms that are ``more and more smoothing''.

\subsubsection{G\aa{}rding inequality} 
\label{s:garding}
A key property is the sharp G\aa{}rding inequality~\cite[Th.~9.11]{zworski} which states that one can find $C_m,N_m>0$ such that, for every $a\in S^m(\mrm{T}^*\R^N)$ satisfying $a\geqslant 0$, one has
\begin{equation}\label{eq:gardingRN}
\langle {\oW}(a)u,u\rangle \geqslant-C_m\|u\|^2_{H^{m-\frac{1}{2}}}\sum_{|\alpha|+|\beta|\leqslant N_m}\|\langle\xi\rangle^{-m+|\beta|}\partial_x^\alpha\partial_\xi^\beta a\|_{\infty}, \quad u \in \mathscr S(\R^N).
\end{equation}

\subsubsection{Change of variables} Finally, if we let $\gamma:\R^N\rightarrow \R^N$ be a smooth diffeomorphism which is equal to the identity outside a compact set and $a$ be an element in $S^m(\mrm{T}^*\R^N)$, then there exists~\cite[Th.~9.10]{zworski} $a_\gamma$ in $S^m(\mrm{T}^*\R^N)$ such that
\begin{equation}\label{eq:changevariableRN}
(\gamma^{-1})^*{\oW}(a)\gamma^*={\oW}(a_\gamma).
\end{equation}
Following the proof in this reference, in the case where $a$ belongs to $\mathscr{S}(\R^{2N})$, the kernel of $K_{a_\gamma}$ writes
$$
K_{a_\gamma}(x,y)=\frac{1}{(2\pi)^N}\int_{\R^N}a\left(\frac{\gamma(x)+\gamma(y)}{2},\xi\right)e^{i\langle\gamma(x)-\gamma(y),\xi\rangle}d\xi.
$$
Thanks to~\eqref{eq:symbol-kernel}, we can write an asymptotic expression for $a_\gamma(x,\xi)$ following the lines of~\cite[Ch.~9]{zworski}. In fact, a remarkable feature of the Weyl quantization is that, in the case of a smooth \emph{volume preserving} diffeomorphism (which is the identity outside a compact set), one can find a sequence $(a_{\gamma,k})_{k\geqslant 2}$ such that, for every $k\geqslant 2$, $a_{\gamma,k}\in S^{m-k}(\mrm{T}^*\R^N)$ and for every $N_0\geqslant 2$,
\begin{equation}\label{eq:changevariableweyl}
 a_\gamma(x,\xi)-a\circ\widetilde{\gamma}(x,\xi)-\sum_{k=2}^{N_0}a_{\gamma,k}(x,\xi) \in S^{m-N_0-1}(\mrm{T}^*\R^N),
\end{equation}
where each $a_{\gamma,k}$ depends linearly on derivatives of order $k$ of $a$ with respect to $\xi$. See e.g.~\cite{bonthonneaubreviary} or~\cite[Th.~9.3]{zworski}.

\subsection{Operator-valued symbols}\label{aa:operatorvalued}

Due to the product structure of the phase space $\R\times M_1$, it is convenient to introduce pseudo-differential operators on $\R$ with operator-valued symbols. For the sake of simplicity, we restrict ourselves to $N=1$ but the presentation can be extended to any $N\geqslant 1$ as in the previous subsection. We fix two Hilbert spaces $\mathscr{H}_1$ and $\mathscr{H}_2$ and the space of continuous linear maps $\mathscr{L}(\mathscr{H}_1,\mathscr{H}_2)$ between these two spaces. As above, the space of (Kohn-Nirenberg) operator-valued symbols is given, for every $m\in\R$, by
\begin{multline*}
S^m(\mrm{T}^*(\R),\mathscr{L}(\mathscr{H}_1,\mathscr{H}_2))\\
=\left\{A\in\mathscr{C}^\infty(\mrm{T}^*\R,\mathscr{L}(\mathscr{H}_1,\mathscr{H}_2)):\ \forall(\alpha,\beta),\ \sup_{r,\rho}\langle\rho\rangle^{|\beta|-m}\|\partial^\alpha_r\partial_\rho^\beta A\|_{\mathscr{H}_1\rightarrow\mathscr{H}_2}<\infty\right\}.
\end{multline*}
As before, one can define the Weyl quantization of such a symbol as follows
$$
\forall u\in\mathscr{S}(\R,\mathscr{H}_1),\quad
{\oW}(A)u(r)=\frac{1}{2\pi}\int_{\R^{2}}e^{i\langle r-s,\rho\rangle}A\left(\frac{r+s}{2},\rho\right)u(s)dsd\rho,
$$
and, along the same lines, one can verify that
$$
{\oW}(A):\mathscr{S}(\R,\mathscr{H}_1)\rightarrow \mathscr{S}(\R,\mathscr{H}_2)
$$
is a continuous linear map. Again, one can introduce the space of pseudo-differential operators $\Psi^m(\R, \mathscr{H}_1,\mathscr{H}_2)$ of order $m$ with values in $\mathscr{L}(\mathscr{H}_1,\mathscr{H}_2)$ as the space of all operators of the previous form ${\oW}(A)$.

The formula for the adjoint remains true in this setting, i.e. for every $u$ in $\mathscr{S}(\R,\mathscr{H}_1)$ and every $v$ in $\mathscr{S}(\R,\mathscr{H}_2)$
\begin{equation}\label{eq:adjointoperatorvalued}
  \int_{\R} \langle {\oW}(A)u(r),v(r)\rangle_{\mathscr{H}_2}dr=\int_{\R} \langle u(r),{\oW}(A^*)v(r)\rangle_{\mathscr{H}_1}dr.
\end{equation}
Similarly, one can write a composition formula for this quantization:
\begin{equation}\label{eq:compositionoperatorvalued}
 {\oW}(A)\circ{\oW}(B)={\oW}(a \operatorname{\sharp} b),
\end{equation}
where $a \operatorname{\sharp} b$ is the Moyal product of the two operators  $A\in S^{m_1}(\mrm{T}^*(\R),\mathscr{L}(\mathscr{H}_2,\mathscr{H}_3))$ and $B\in S^{m_2}(\mrm{T}^*(\R),\mathscr{L}(\mathscr{H}_1,\mathscr{H}_2))$. Again, it can be shown that $a \operatorname{\sharp} b$ belongs to $S^{m_1+m_2}(\mrm{T}^*(\R),\mathscr{L}(\mathscr{H}_1,\mathscr{H}_3))$ and that, for every $N_0\geqslant 0$
\begin{equation}\label{eq:moyaloperatorvalued}
 a \operatorname{\sharp} b(r,\rho)-\sum_{k=0}^{N_0}\frac{\left(\partial_\rho\partial_s-\partial_\sigma\partial_r\right)^k}{(2i)^kk!}\left(A(r,\rho)B(s,\sigma)\right)|_{r=s,\rho=\sigma},
\end{equation}
belongs to $S^{m_1+m_2-N_0-1}(\mrm{T}^*(\R^N),\mathscr{L}(\mathscr{H}_1,\mathscr{H}_3))$ where one can verify that each term in the sum belongs to $S^{m_1+m_2-k}(\mrm{T}^*(\R),\mathscr{L}(\mathscr{H}_1,\mathscr{H}_3))$. A notable difference with the scalar case is that there is no particular simplification when writing $[{\oW}(A),{\oW}(B)]$. In particular, we do not have in general that $[{\oW}(A),{\oW}(B)]\in \Psi^{m_1+m_2-1}$. Finally, the analogue of the Calder\'on-Vaillancourt Theorem reads (on $H^s$ spaces): for every $s\in\R$, there exist $C_s>0$ and $N_s\in\Z_+$ such that, for every $a\in S^m(\mrm{T}^*\R,\mathscr{L}(\mathscr{H}_1,\mathscr{H}_2))$,
\begin{equation}\label{eq:calderonoperatorvalued}
\left\|{\oW}(A)\right\|_{H^s(\R,\mathscr{H}_1)\rightarrow H^{s-m}(\R,\mathscr{H}_2)}\leqslant C_0\sum_{|\alpha|+|\beta|\leqslant N_0}\sup_{r,\rho}\left\|\langle\rho\rangle^{-m+|\beta|}\partial^\alpha_r\partial_\rho^\beta A(r,\rho)\right\|_{\mathscr{L}(\mathscr{H}_1,\mathscr{H}_2)}. 
\end{equation}

\subsection{Pseudodifferential operators on $\R\times M_1$}\label{aa:mfd}

\subsubsection{The case of $M_1$}\label{aa:compactmfd} 

As in the case of $\R^N$, we can introduce the class of Kohn-Nirenberg symbols:
$$
S^m(\mrm{T}^*M_1)=\left\{a\in\mathcal{C}^\infty(\mrm{T}^*M_1):\forall(\alpha,\beta),\ \|\langle\zeta_1\rangle^{|\beta|-m}\partial_{z_1}^\alpha\partial_{\zeta_1}^\beta a\|_\infty<\infty\right\},
$$
where the derivatives are understood in local coordinate charts. Again, we use the convention 
$$
S^{m+}(\mrm{T}^*M_1)=\bigcap_{m'>m}S^{m'}(\mrm{T}^*M_1),
$$
We will now proceed as in~\cite[\S14.2.3]{zworski} to associate to each $a$ in $S^m(\mrm{T}^*M_1)$ a pseudo-differential operator in $\Psi^{m}(M_1)$ (see $\S14.2.2$ in this reference for a precise definition of pseudo-differential operators on a compact manifold). For later purposes, we introduce the semi-norms associated to this class of symbols
$$
\forall(\alpha,\beta)\in\Z_+^2,\quad p_{m,\alpha,\beta}(a)=\sum_{j\in J}\|\langle\zeta_1\rangle^{|\beta|-m}\partial_{z_1}^\alpha\partial_{\zeta_1}^\beta (\widetilde{\gamma}_j^{-1})^*(a)\|_\infty,
$$
where we have fixed a finite atlas $(U_j,\gamma_j)_{j\in J}$ on $M_1$ made of smooth diffeomorphism $\gamma_j:U_j\subset M_1\rightarrow V_j\subset \R^{2n-1}$ such that $\gamma_j^*(\text{Vol}_{g_S})=\text{Leb}_{\R^{2n-1}}$. Here $\text{Vol}_{g_S}$ denotes the volume form induced by the Sasaki metric $g_S$ on $M_1$. The existence of an isochore atlas follows from an argument due to Moser~\cite{Moser1965}. We then fix a smooth partition of unity $(\chi_j)_{j\in J}$, i.e.
$$
\forall z_1\in M_1,\quad\sum_{j\in J}\chi_j^2(z_1)=1,\quad\chi_j\in\mathscr{C}^\infty_c(U_j,[0,1]).
$$
For every $j\in J$, we also fix a smooth function $\widetilde{\chi}_j\in\mathscr{C}^\infty_c(U_j,[0,1])$ such that $\widetilde{\chi}_j\chi_j=\chi_j$. We can then set, for every $a\in S^m(\mrm{T}^*M_1)$,
\begin{equation}\label{eq:pseudomfd}
{\oM}(a)=\sum_{j\in J}\chi_j\gamma_j^*{\oW}\left((\widetilde{\gamma}_j^{-1})^*(\widetilde{\chi_j}a)\right)(\gamma_j^{-1})^*\chi_j:\mathscr{C}^\infty(M_1)\rightarrow\mathscr{C}^\infty(M_1).
\end{equation}
\begin{rema}\label{r:changecoord} For any $\ell\in J$, we fix two smooth functions $(\psi_{\ell,j})_{j=1,2}\in\mathscr{C}^\infty_c(U_\ell,[0,1])$ such that $\psi_{\ell,j}\widetilde{\chi}_\ell=\widetilde{\chi}_\ell$ and $\psi_{\ell,1}\psi_{\ell,2}=\psi_{\ell,2}$. One can then write
 \begin{multline*}
 (\gamma_\ell^{-1})^*\psi_{\ell,1}{\oM}(a)\psi_{\ell,2}\gamma_\ell^*\\
 =\sum_{j\in J}(\psi_{\ell,1}\chi_j)\circ\gamma_\ell^{-1}(\gamma_j\circ \gamma_\ell^{-1})^*{\oW}\left((\widetilde{\gamma}_j^{-1})^*(\widetilde{\chi_j}a)\right)(\chi_j\psi_{\ell,2})\circ\gamma_j^{-1}(\gamma_\ell\circ\gamma_j^{-1})^*
 \end{multline*}
 Thanks to~\eqref{eq:compositionRN} and to~\eqref{eq:changevariableRN}, there exists $\widetilde{a}_\ell\in S^m(\mrm{T}^*\R^{2n-1})$ such that 
 $$
 (\gamma_\ell^{-1})^*\psi_{\ell,1}{\oM}(a)\psi_{\ell,2}\gamma_\ell^*={\oW}(\widetilde{a}_\ell).
 $$
 Moreover, thanks to~\eqref{eq:changevariableweyl} and to the fact that we picked an isochore atlas, $a_\ell$ has an asymptotic expansion of the form
 $$
 \widetilde{a}_\ell=(\psi_{\ell,1}\psi_{\ell,2})\circ\gamma_\ell^{-1} a\circ\widetilde{\gamma}_\ell^{-1}+\sum_{k=2}^{N_0}a_{k,\ell}+\mathcal{O}_{S^{m-N_0-1}}(1),
 $$
 where each $a_{k,l}\in S^{m-k}(\mrm{T}^*\R^N)$ depends linearly on derivatives of order $k$ of $a$ with respect to $\xi$ and is supported in $ V_\ell\times\R^N$. Multiplying this expression by $\chi_\ell^2$ and making the sum over $\ell$, we define $\sigma({\oM}(a))$ the principal symbol of ${\oM}(a)$. Hence, one finds that the principal symbol of ${\oM}(a)$ is well defined modulo $S^{m-2}(\mrm{T}^*M_1)$. This is a specific feature of the fact that we picked an isochore atlas combined with the Weyl quantization.
\end{rema}

\begin{rema}\label{r:explicitsymbol} With our choice of quantization, we have the following simplified expressions. If $b(z_1)$ is a function that is independent of $\zeta_1$, then ${\oM}(b)u=bu$. If $Y$ is a smooth vector field on $M$, then $Y={\oM}(\zeta_1(Y(z_1))+r(z_1))$~\cite[Th.~4.5]{zworski}, where $r(z_1)$ depends on the choice of coordinate charts and linearly on $Y$. Indeed, one can write
$$
Yu=\sum_{j\in J}Y(\chi_j^2u)=\sum_{j\in J}\chi_jY(\chi_j u)=\sum_{j\in J}\chi_j\gamma_j^*\left((\gamma_j^{-1})^*Y\gamma_j^*\right)(\gamma_j^{-1})^*\chi_j u,
$$
and apply~\cite[Th.~4.5]{zworski} to express the vector field $(\gamma_j^{-1})^*Y\gamma_j^*$ as ${\oW}(\zeta(Y)+r_j)$. Note also that, for a volume preserving vector field, the remainder is equal to $0$ thanks to our choice of isochore charts.
\end{rema}

From our choice of picking an isochore atlas, we can immediately verify that the following properties hold, for every $a\in S^m(\mrm{T}^*M_1)$
\begin{equation}\label{eq:adjointmfd}
 \forall  u, v \in \mathscr{C}^\infty(M_1),\quad\langle {\oM}(a)u,v\rangle_{L^2(M_1)}=\langle u,{\oM}(\overline{a})v\rangle_{L^2(M_1)},
\end{equation}
and there exists $C_m,N_m>0$ such that, if $a\geqslant 0$, then
\begin{equation}\label{eq:gardingmfd}
 \forall  u \in \mathscr{C}^\infty(M_1),\ \langle {\oM}(a)u,u\rangle_{L^2(M_1)}\geqslant -C_m\|u\|^2_{H^{m-\frac{1}{2}}}\sum_{|\alpha|+|\beta|\leqslant N_m}\|\langle\zeta_1\rangle^{-m+|\beta|}\partial_{z_1}^\alpha\partial_{\zeta_1}^\beta a\|_{\infty}.
\end{equation}
We also find that the Calder\'on-Vaillancourt Theorem remains true for these operators. In other words, for every $s\in\R$, one can find $C_s,N_s>0$ such that 
\begin{equation}\label{eq:calderonmfd}
 \| {\oM}(a)\|_{H^s(M_1)\rightarrow H^{s-m}(M_1)}\leqslant C_s\sum_{|\alpha|+|\beta|\leqslant N_s}\|\langle\zeta_1\rangle^{-m+|\beta|}\partial_{z_1}^\alpha\partial_{\zeta_1}^\beta a\|_{\infty}.
\end{equation}
\begin{rema}
More generally, one can verify that every pseudo-differential operator $A\in \Psi^m(M_1)$ verifies
\begin{equation}\label{eq:calderonmfd2}
 \| A\|_{H^s(M_1)\rightarrow H^{s-m}(M_1)}<\infty.
\end{equation}
\end{rema}
Similarly, following the proof of~\cite[Th.~14.1]{zworski}, we can give asymptotic expansions for the operator ${\oM}(a){\oM}(b)$ when $a\in S^{m_1}(\mrm{T}^*M_1)$ and $b\in S^{m_2}(\mrm{T}^*M_1)$. In fact, for any $\ell\in J$, we fix a smooth function $\psi_\ell\in\mathscr{C}^\infty_c(U_\ell,[0,1])$ such that $\psi_\ell\widetilde{\chi}_\ell=\widetilde{\chi}_\ell$. Proceeding and keeping the same notations as in Remark~\ref{r:changecoord}, we can then determine the symbol 
$$
(\gamma_\ell^{-1})^*\psi_{\ell,1}{\oM}(a){\oM}(b)\psi_{\ell,2}\gamma_\ell^*
$$ in view of determining the principal symbol of ${\oM}(a){\oM}(b)$:
\begin{multline*}
(\gamma_\ell^{-1})^*\psi_{\ell,1}{\oM}(a){\oM}(b)\psi_{\ell,2}\gamma_\ell^*\\
=\sum_{j,k\in J}(\gamma_j\circ\gamma_\ell^{-1})^*(\psi_{\ell,1}\chi_j)\circ\gamma_j^{-1}{\oW}((\widetilde{\gamma}_j^{-1})^*(\widetilde{\chi}_ja))(\chi_j\widetilde{\psi}_\ell)\circ\gamma_j^{-1}(\gamma_\ell\circ\gamma_j^{-1})^*\\
\sum_{k\in J}(\gamma_k\circ\gamma_\ell^{-1})^*(\widetilde{\psi}_\ell\chi_k)\circ\gamma_k^{-1}{\oW}((\widetilde{\gamma}_k^{-1})^*(\widetilde{\chi}_kb))(\chi_k\psi_{\ell,2})\circ \gamma_k^{-1}(\gamma_\ell\circ\gamma_k^{-1})^*
+{\oW}(r_\ell),
\end{multline*}
where $\widetilde{\psi}_\ell\in\mathscr{C}^\infty_c(U_\ell)$ is identically equal to $1$ on the support of $\psi_{\ell,j}$, $j=1,2$, and where $r_\ell$ is an element in $S^{-\infty}(\mrm{T}^*\R^N)$ thanks to the kernel representation of pseudo-differential operators recalled in \S\ref{aa:pseudoRN}. The exact same calculation as in Remark~\ref{r:changecoord} shows that
$$
c_\ell=(ab)\circ\widetilde{\gamma}_\ell^{-1}(\psi_{\ell,1}\psi_{\ell,2})^2\circ\gamma_\ell^{-1}+\frac{1}{2i}\left\{a\psi_{\ell,1},b\psi_{\ell,2}\right\}\circ\widetilde{\gamma}_\ell^{-1}+\sum_{k=2}^{N_0}c_{k,\ell}+\mathcal{O}_{S^{m_1+m_2-N_0-1}}(1).
$$
Multiplying by $\chi_\ell^2$ and summing over $\ell$, we find that
\begin{equation}\label{eq:compositionmfd}
{\oM}(a){\oM}(b)={\oM}\left(ab+\frac{1}{2i}\{a,b\}\right)+\mathcal{O}_{\Psi^{m_1+m_2-2}(M_1)}(1). 
\end{equation}
More precisely, there exist bilinear differential operators\footnote{It means that $L_k(a,b)=\sum_{|\alpha|+|\beta|\leqslant 2k}c_{\alpha,\beta}(z_1;\zeta_1)\partial^\alpha a \, \partial^\beta b$.} $(L_k)_{k\geqslant 0}$ of order $\leqslant 2k$ (depending on our choice of coordinate charts) such that, for every $a\in S^{m_1}(\mrm{T}^*M_1)$ and for every $b\in S^{m_2}(\mrm{T}^*M_1)$, one has, for every $N_0\geqslant 0$,
\begin{equation}\label{eq:compositionmfd2}
R_{N_0}(a,b)= {\oM}(a){\oM}(b)-\sum_{k=0}^{N_0}{\oM}(L_k(a,b))\in\Psi^{m_1+m_2-N_0-1}(M_1),
\end{equation}
where
\begin{itemize}
 \item $L_0(a,b)=ab$,
 \item ${\mrm L}_1(a,b)=\frac{1}{2i}\{a,b\}$,
 \item for every $k\geqslant 2$ and for every $\alpha=(\alpha_1,\alpha_2,\alpha_3,\alpha_4)\in\Z^{8n-4}$, there exist smooth functions $c_{\alpha,k}$ on $M_1$ (depending on our various choices of coordinate charts and cutoff functions) such that
 \begin{equation}\label{eq:highordertermcomposition}
  L_k(a,b)=\sum_{|\alpha|\leqslant 2k, |\alpha_2|+|\alpha_4|=k}c_{\alpha,k}(z_1)\partial_{z_1}^{\alpha_1}\partial_{\zeta_1}^{\alpha_2}a\, \partial_{z_1}^{\alpha_3}\partial_{\zeta_1}^{\alpha_4}b \in S^{m_1+m_2-k}(\mrm{T}^*M_1).
 \end{equation}
 \end{itemize}
We observe that the formulas for the bracket $[{\oM}(a),{\oM}(b)]$ also translate to the case of manifolds thanks to the above discussion.

\begin{rema}\label{r:controlproduct} From the proof of~\eqref{eq:compositionmfd}, we can estimate the size of the remainder as a continuous operator from $H^s(M_1)$ to $H^{s+N_0+1-(m_1+m_2)}(M_1)$ in terms of the semi-norms of $a$ and $b$. Indeed, in view of estimating these norms, one can conjugate $R_{N_0}(a,b)$ by $(\gamma_{\ell}^{-1})^*\psi_{\ell,1}$ (on the left) and $(\gamma_{k})^*\psi_{k,2}$ (on the right) and verify that all the involved terms can be expressed in terms of a finite number of derivatives of $a$ and $b$ (with at least $N_0$ derivatives with respect to $\zeta_1$). More precisely, one can verify that, for every $N_0,N_1\geqslant 0$,
\begin{multline*}
\|R_{N_0+N_1}(a,b)\|_{H^s\rightarrow H^{s+N_0+1-(m_1+m_2)}}\leqslant C_{N_0, N_1,s,m_1,m_2}\\
\sup_{|\alpha_0|+|\alpha_1|=N_1}\sup_{|\beta_0|+|\beta_1|\leqslant K_{N_0, N_1,s,m_1,m_2}}\left\{p_{m_1,\beta_0}(\partial_{\zeta_1}^{\alpha_0}a)p_{m_2,\beta_1}(\partial_{\zeta_1}^{\alpha_1}b))\right\}.
\end{multline*}
Here, we took two indices $N_0$ and $N_1$ as, when working on $\R\times M_1$, we will have an extra parameter $\rho$ that we will consider as a kind of semiclassical gain through the small parameter $\langle\rho\rangle^{-1}$. For these symbols, this gain will occur with each derivative with respect to $\zeta_1$ so that the above remainder will be of size $\langle\rho\rangle^{-N_1}$. See for instance the proof of Lemma~\ref{l:pseudoinverse1} below.
\end{rema}

\subsubsection{Exponential of pseudo-differential operators on manifolds}\label{aaa:exponential}

We consider a smooth function $\lambda$ verifying the following assumptions:
\begin{equation}\label{eq:admissibleweight}
 \lambda\in S^{0+}(\mrm{T}^*M_1,\R)\quad\text{and}\quad \exists m,R\geqslant 0\ \text{such that}\ |\lambda(z_1;\zeta_1)| \leqslant m\log(R+\langle\zeta_1\rangle).
\end{equation}
Then, according to~\cite[Th.~8.6]{zworski} (adapted to the case of compact manifolds, see also~\cite[Th.~6.4]{BonyChemin1994} for general H\"ormander-Weyl symbols~\cite[\S18.4]{HormanderIII}), one finds that, the equation
$$
\partial_tB(t)={\oM}(\lambda)B(t),\quad B(0)=\text{Id}
$$
has a solution that we denote by $\exp(t{\oM}(\lambda))$. More precisely, it is a smooth function of $t$ with values in continuous linear mapping from $\mathscr{C}^\infty(M_1)$ to itself. In fact, the proof in~\cite{zworski} (adapted to the case of compact manifolds) which is of semiclassical nature shows that, for every $R_0\geqslant 1$, one can find $0<h_0<1$ such that, for every $t\in[-R_0,R_0]$, for every $s\in\mathbb{R}$ and for every $0<h<h_0$, the following holds:
\begin{itemize}
 \item there exists a bounded operator $Q_{h,1}(t):H^s\rightarrow H^s$ (with a norm that is bounded independently of $t\in[-R_0,R_0]$ and $0<h<h_{0}$) such that one can find a symbol $b_{h,1}$ which is equal to $1$ modulo in $S^{-1+}(\mrm{T}^*M_1)$ with the following property
\begin{equation}\label{eq:symbolexpright}
 \exp\left(t{\oM}(\lambda(z_1,h\zeta_1))\right)=Q_{h,1}(t){\oM}\left(e^{t\lambda(z_1,h\zeta_1)}b_{h,1}(z_1,h\zeta_1)\right),
\end{equation}
where we note that $b_{h,1}$ is independent of $t$ and has all its semi-norms (in $S^{0+}(\mrm{T}^*M_1)$) uniformly bounded in terms of $0<h<h_0$. Recall that the operator ${\oM}\left(e^{t\lambda(z_1,h\zeta_1)}b_{h,1}(z_1,h\zeta_1)\right)$ appears has a pseudo-inverse in the proof of~\cite[Lemma~8.4]{zworski}, namely
\begin{equation}\label{eq:approxinverse-exp}
{\oM}\left(e^{t\lambda(z_1,h\zeta_1)}b_{h,1}(z_1,h\zeta_1)\right){\oM}\left(e^{-t\lambda(z_1,h\zeta_1)}\right)=\text{Id}+\mathcal{O}_{H^s\rightarrow H^s}(h).
\end{equation}
Observe that the involved symbols lie in the class $S^{R_0m}(\mrm{T}^*M_1)$ of standard Kohn-Nirenberg symbols.
\item Similarly, there exists a bounded operator $Q_{h,2}(t):H^s\rightarrow H^s$ (with a norm that is bounded independently of $t\in[-R_0,R_0]$ and $0<h<h_0$) such that one can find a symbol $b_{h,2}$ which is equal to $1$ modulo  $S^{-1+}(\mrm{T}^*M_1)$ with the following property
\begin{equation}\label{eq:symbolexpleft}
 \exp\left(t{\oM}(\lambda(z_1,h\zeta_1))\right)={\oM}\left(e^{t\lambda(z_1,h\zeta_1)}b_{h,2}(z_1,h\zeta_1)\right)Q_{h,2}(t),
\end{equation}
with the same properties for $b_{h,2}$.
 \end{itemize}

\begin{rema} Recall that the proof in~\cite{zworski} goes as follows. Suppose that there exists a solution to
 $$
\partial_tB_h(t)={\oM}_h(\lambda)B_h(t),\quad B_h(0)=\text{Id},
$$
where ${\oM}_h(b(z_1;\zeta_1))={\oM}(b(z_1,h\zeta_1))$. Then, letting $U_h(t)={\ohM}(e^{t\lambda})$, one has 
$$
\partial_t\left(U_h(-t)B_h(t)\right)=\left(-{\ohM}(\lambda e^{-t\lambda})+{\ohM}(e^{-t\lambda}){\ohM}(\lambda)\right)B_h(t)=V_h(t)B_h(t).
$$
Hence, if one sets $C_h(t)=-V_h(t)U_h(-t)^{-1}$, one gets a smooth family of bounded operator on $H^s(M_1)$ using the composition rule and the Calder\'on-Vaillancourt Theorem. Now the Cauchy-Lipschitz Theorem ensures the existence and uniqueness of a smooth solution $Q_h(t):H^s\rightarrow H^s$ to the problem
$$
\partial_tQ_h(t)=C_h(t)Q_h(t),\quad Q_h(0)=\text{Id},
$$
and $B_h(t)$ is then given by $U_h(-t)^{-1}Q_h(t)$. Using~\eqref{eq:approxinverse-exp} with $h>0$ small enough, one can find an expression for $U_h(-t)^{-1}$ and this yields~\eqref{eq:symbolexpleft}. The same reasoning yields~\eqref{eq:symbolexpright}.
\end{rema}

\begin{rema} Note that it would require slightly more work to check that $Q_{h,j}(t)$ is a pseudo-differential operator but we do not discuss this issue as this is not necessary for our analysis. 
\end{rema}


\subsubsection{Pseudodifferential operators on $\R\times M_1$ as operator valued pseudo-differential operators on $\R$} In view of our analysis, it is convenient to use the framework of \S\ref{aa:operatorvalued} to define operators associated with symbols lying in 
$$
\overline{S}^m(\mrm{T}^*(\R\times M_1))=\left\{a\in\mathcal{C}^\infty(\mrm{T}^*(\R\times M_1)):\forall(\alpha,\beta),\ \|\langle(\rho,\zeta_1)\rangle^{|\beta|-m}\partial_{rz_1}^\alpha\partial_{\rho\zeta_1}^\beta a\|_\infty<\infty\right\}.
$$
Namely, we define
\begin{equation}\label{eq:pseudononcompact}
 \mathbf{Op}(a)u(r,z_1)=\frac{1}{2\pi}\int_{\R^2}e^{i(r-s)\rho}{\oM}\left(a\left(\frac{r+s}{2},\rho;.\right)\right)u(s,z_1)dsd\rho,
\end{equation}
which can be identified with a pseudo-differential operator with symbol taking values in $\mathscr{L}(H^\sigma(M_1),H^{\sigma-m}(M_1))$ (for every $\sigma\in\R$). Note that compared with the class $S^m$ we require the derivatives to be bounded uniformly on $\mathbb{R}$ (rather than on compact sets).

Letting 
$$
\mathscr{S}(\R\times M_1)=\left\{u\in\mathscr{C}^\infty(\R\times M_1):\ \forall(\alpha,\beta),\ \|(1+|r|)^\alpha\partial^\beta_{rz_1} u(r,z_1)\|_\infty<\infty\right\},
$$
this quantization procedure induces a bounded operator
$$
\mathbf{Op}(a):\mathscr{S}(\R\times M_1)\rightarrow \mathscr{S}(\R\times M_1).
$$

\begin{rema}\label{r:exoticsymbol}
 The advantage of devising such a definition is that we can quantize more exotic symbols than in the standard Kohn-Nirenberg classes $\overline{S}^m(\mrm{T}^*(\R\times M_1))$ or $S^m(\mrm{T}^*(\R\times M_1))$. For instance, one can pick in this definition functions of the form $(1+\rho^2)^{m_0}a(z_1;\zeta_1)$ (with $m_0\in\R$ and $a\in S^m(\mrm{T}^*M_1)$) and view the corresponding operator as an element of $\Psi^{m_0}(\R,\mathscr{L}(H^\sigma(M_1),H^{\sigma-m}(M_1)))$. 
\end{rema}

\begin{rema}\label{r:pseudononcompact} When picking symbols $a$ in $\overline{S}^m(\mrm{T}^*(\R\times M_1))$ or in $S^m(\mrm{T}^*(\R\times M_1))$, one recover using~\eqref{eq:changevariableRN} that $\mathbf{Op}(a)$ is a pseudo-differential operator of order $m$ on the noncompact manifold $\R\times M_1$ in the sense of~\cite[Prop.~E.13, Rk p.551]{dyatlov2017mathematical}. For instance, $\mathbf{Op}(a)$ verifies the Calder\'on-Vaillancourt Theorem between $H^{s}_{\text{comp}}$ and $H^{s-m}_{\text{loc}}$~\cite[Prop.~E.22]{dyatlov2017mathematical}. Moreover, if we fix a smooth cutoff function $\chi\in\mathcal{C}^{\infty}_c(\R)$, then $\mathbf{Op}(a)\chi$ is properly supported in the sense of~\cite[\S A.7]{dyatlov2017mathematical} and it is thus amenable to the composition rule for pseudo-differential operators on noncompact manifold~\cite[Prop.~E.17]{dyatlov2017mathematical}. Similarly, if $\widetilde{\chi}\in\mathcal{C}^{\infty}_c(\R)$, then $\chi\mathbf{Op}(a)\widetilde{\chi}$ is compactly supported. In that case, $\chi\mathbf{Op}(a)\widetilde{\chi}$ verifies the G\aa{}rding inequality for smooth $u$ that are  compactly supported in $\R\times M_1$~\cite[Prop.~E.23]{dyatlov2017mathematical}.
\end{rema}

\subsubsection{Cutoff functions in the $\R$ variable} 

In the above class of symbols, one has the following lemma:
\begin{lemm}\label{l:cutoff} Let $\chi_1,\chi_2$ be two smooth functions on $\R$ that have disjoint supports, all of whose derivatives are bounded and such that $\chi_2$ has compact support. Let $a\in\overline{S}^m(\mrm{T}^*(M_1\times \R))$. Then, one has, for every $N_0,N_1\geqslant 0$,
$$
\|\chi_1\mathbf{Op}(a)\chi_2\|_{H^{-N_1}(\R, H^{-N_0}( M_1))\rightarrow H^{N_1}(\R, H^{N_0}( M_1))}<\infty,
$$
and
$$
\|\chi_1\mathbf{Op}(a)X\chi_2\|_{H^{-N_1}(\R, H^{-N_0}( M_1))\rightarrow H^{N_1}(\R, H^{N_0}( M_1))}<\infty,
$$
where $X=rX_1$ is the geodesic vector field.
\end{lemm}
\begin{rema}\label{r:comparisonsobolevnorms} Note that, for every $N_0\geqslant 0,$ one has
$$\left\|u\right\|_{H^{N_0}(\R\times M_1)}\leqslant C_{N_0}\left\|u\right\|_{H^{N_0}(\R,H^{N_0}(M_1))},$$
and
 $$\left\|u\right\|_{H^{-N_0}(\R,H^{-N_0}(\times M_1))}\leqslant C_{N_0}\left\|u\right\|_{H^{-N_0}(\R\times M_1)}.$$
 Hence, Lemma~\ref{l:cutoff} can be translated to the standard Sobolev spaces on $\R\times M_1$.
\end{rema}

\begin{proof}
 We treat the case with $X$ (the other case works analoguously). To that aim, we write the composition formula~\eqref{eq:compositionoperatorvalued} for $\chi_1\mathbf{Op}(a)X_1$ and we find, for every $N_2\geqslant 0$,
 $$
 \chi_1\mathbf{Op}(a)X_1={\oW}\left(\sum_{k=0}^{N_2}\frac{1}{(-2i)^kk!}\chi^{(k)}_1{\oM}(\partial_\rho^ka)X_1 +R_{N_2}\right),
 $$
 where the remainder lies in $S^{m-N_2-1}(\R,\mathscr{L}(H^{s}(M_1),H^{s-m}(M_1)))$ with the semi-norm of $R_{N_0}$ (in that class) that depends on the semi-norms of $\chi_1^{(N_2)}(r){\oM}(\partial_\rho^{N_2}a)X_1$. We apply the composition formula one more time to this expression after multiplying this equality on the right by $r\chi_2(r)$. We find that
 $$
 \chi_1\mathbf{Op}(a)rX_1\chi_2={\oW}\left(\widetilde{R}_{N_2}(r,\rho)\right)\widetilde{\chi}_2,
 $$
 where $\widetilde{\chi}_2$ is compactly supported function which is equal to $1$ on the support of $\chi_2$ and $\widetilde{R}_{N_2}$ is an operator-valued symbol that can be expressed in terms ${\oM}(\partial_\rho^{N_2}a)$. As $a$ belongs to $\overline{S}^m(\mrm{T}^*(\R\times M_1))$, this defines a bounded operator from $H^{-(N_0+1)}(M_1)$ to $H^{N_0}(M_1)$ whose norm is bounded by $\langle\rho\rangle^{-N_1'}$ provided that $N_2\gg N_0+N_1'+m.$ Picking $N_1'\gg N_1$, this yields the expected upper bound.
\end{proof}

\subsubsection{Pseudo inverse on $\R\times M_1$}

We conclude this appendix with two observations showing that one can always find an (almost) inverse for $\mathbf{Op}(a)$ when $a>0$ is a simple enough symbol. More precisely, suppose that $\lambda(z_1;\zeta_1)$ verifies~\eqref{eq:admissibleweight} and that $m_0\in\R$. We define
$$
B_{\lambda,m_0}u(r,z_1)=\frac{1}{2\pi}\int_{\R^2}e^{i(r-s)\rho}(1+\rho^2)^{m_0}\exp{\oM}(\lambda)u(s,z_1)dsd\rho.
$$
With this notation at hand, one finds
\begin{equation}\label{eq:inverseeasy}
 \forall u\in\mathscr{S}(\R\times M_1),\quad B_{\lambda,m_0} B_{-\lambda,-m_0}u=u.
\end{equation}
Finally, we discuss the case of the more general symbols in $\overline{S}^{m}(\mrm{T}^*(\R\times M_1))$:
\begin{lemm}\label{l:pseudoinverse1} Let $\lambda\in \overline{S}^{+0}(\mrm{T}^*(M_1\times \R))$ that is independent of the variable $r$ and such that $|\lambda(z_1,\rho,\zeta_1)|\leqslant m\ln(R+ \langle(\rho,\zeta_1)\rangle)$ for some given $m,R\geqslant 0$. Then, $a=e^\lambda$ belongs to $\overline{S}^{m+}(\mrm{T}^*(\R\times M_1))$

Moreover, given $N_0,N_1\geqslant 0$, one can find $b_{N_0,N_1}\in\overline{S}^0(\mrm{T}^*(\R\times M_1))$ which is independent of $r$ and equal to $1$ modulo $\overline{S}^{-1+}(\mrm{T}^*(\R\times M_1))$ such that, for every $\chi\in\mathscr{C}^\infty_c(\R)$, the operator
$$
\left(\mathbf{Op}(e^{-\lambda}b_{N_0,N_1})\mathbf{Op}(e^{\lambda})-\operatorname{Id}\right)\chi: H^{-N_1}\left(\R,H^{-N_0}(M_1)\right)\rightarrow H^{N_1}\left(\R,H^{N_0}(M_1)\right)
$$
is bounded.
\end{lemm}
\begin{proof} The fact that $e^\lambda$ belongs to $\overline{S}^{m+}(\mrm{T}^*(\R\times M_1))$ follows directly from the definition of this class of symbols. We now fix some $b\in \overline{S}^0(\mrm{T}^*(\R\times M_1))$ which is equal to $1$ modulo $\overline{S}^{-1+}(\mrm{T}^*(\R\times M_1))$ and which is independent of the variable $r\in\R$. From the composition rule~\eqref{eq:compositionoperatorvalued}, one has that, for every $u\in\mathscr{S}(\R\times M_1)$,
$$
\mathbf{Op}(e^\lambda)\mathbf{Op}(e^{-\lambda}b)u(r,z_1))=\frac{1}{2\pi}\int_{\R^2}e^{i(r-s)\rho}{\oM}(e^{-\lambda(\rho)}b(\rho)){\oM}(e^{\lambda(\rho)})u(s,z_1)d\rho ds.
$$
Using the composition rule~\eqref{eq:compositionmfd} on $M_1$ together with Remark~\ref{r:controlproduct}, we can find $b_{N_0,N_1}(z_1,\rho,\zeta_1)\in \overline{S}^{0}(\mrm{T}^*(\R\times M_1))$ such that, for every $\rho\in\R$,
$$
{\oM}(e^{-\lambda(\rho)}b_{N_0,N_1'}(\rho)){\oM}(e^{\lambda(\rho)})=\text{Id}+\mathbf{R}_{N_0,N_1'}(\rho),
$$
where $\|\mathbf{R}_{N_0,N_1'}(\rho)\|_{H^{-N_0}\rightarrow{H}^{N_0}}\leqslant C_{N_0,N_1'}\langle\rho\rangle^{-N_1'}$. Applying this identity with $N_1'\gg N_1$, we obtain the expected result.
\end{proof}

\section{Proof of Theorem~\ref{t:dolgopyatliverani}}\label{a:liverani}

In this short appendix, we briefly explain how to deduce Theorem~\ref{t:dolgopyatliverani} from the upper bound~\eqref{eq:exp-mixing-classical},
$$
\left|\int_{\mrm{S}^*\Sigma} u(t)\psi \dd {\mrm L}_1- \int_{\mrm{S}^*\Sigma}\left(\int_{\mrm{S}^*\Sigma} u_0\dd {\mrm L}_1\right) \psi \dd {\mrm L}_1\right|\leqslant Ce^{-\vartheta_0 |t|}\|u_0\|_{\mathscr{C}^{N_0}}\|\psi\|_{\mathscr{C}^{N_0}}.
$$
Write
$$
 \int_{\mrm{T}^*\Sigma} u(t)\psi \dd {\mrm L} =c_n\int_0^\infty r^{n-1}\int_{\mrm{S}^*\Sigma} u_0(r,\varphi_{tr}(x,\xi_1))\psi(r,x,\xi_1) \dd {\mrm L}_1(x,\xi_1)\dd r,
$$
where $(x,\xi)=(r,x,\xi_1)$ are the spherical coordinates on $\mrm{T}^*\Sigma$ and $c_n$ is a normalizing constant depending only on $n$. Applying~\eqref{eq:exp-mixing-classical} to the integral over $\mrm{S}^*\Sigma$, one gets
$$
 \int_{\mrm{T}^*\Sigma} u(t)\psi \dd {\mrm L} =\int_{\mrm{T}^*\Sigma} \left(\int_{\mrm{S}^*\Sigma} u_0(r,x,\xi_1) \dd {\mrm L}_1(x,\xi_1)\right)\psi \dd {\mrm L}+\mathcal{O}_{\psi,u_0}\left(\int_{r_0}^\infty e^{-tr\vartheta_0 }r^{n-1}dr\right),
$$
where the constant depends linearly on the $\mathscr{C}^{N_0}$ norms of $u_0$ and $\psi$. This concludes the argument.

\bibliographystyle{alpha}
\bibliography{biblio}

\end{document}